\documentclass[reqno,a4paper]{amsart}
\usepackage[english,activeacute]{babel}
\usepackage{amssymb,amsmath,amsthm,amsfonts,mathrsfs,hyperref}
\usepackage{stackengine}
\usepackage{marvosym}
\usepackage{bm}
\usepackage[font=footnotesize]{caption}
\usepackage{upgreek}
\usepackage[T3,T1]{fontenc} 
\usepackage{tikz}
\usetikzlibrary{intersections,shapes,trees,arrows,spy,positioning,decorations,arrows.meta,decorations.pathreplacing,patterns,calc}

\usepackage{tikz-cd}
\usepackage{tkz-euclide}
\usepackage{pgfplots}
\pgfdeclarelayer{background}
\pgfsetlayers{background,main}
\tikzset{
    my box/.style = {
        , line cap = round
        , line join = round
    }
}

\newcommand{\highlight}[3]{
    \path [my box, line width = #1, draw = #2] #3;

    \pgfmathsetmacro{\innerlinewidth}{0.9 * #1}
    \path [my box, line width = \innerlinewidth, draw = #2!20] #3;
}

\newcommand{\nwedge}{\textrm{\rotatebox[origin=c]{180}{$\curlywedge$}}}

\newcommand{\defeq}{\mathrel{\mathop:}=}
\newcommand{\eqdef}{=\mathrel{\mathop:}}

\newcommand{\sroot}[1]{\scalebox{.4}{\begin{tikzpicture}
	\begin{scope}[every node/.style={circle,thick,draw}, minimum size=20pt]
	\node (h1) at (0,0) {\scalebox{2}{${#1}$}};
	\end{scope}
	\end{tikzpicture}}}
\newcommand{\ssroot}[1]{\scalebox{.8}{\sroot{#1}}}

\makeatletter																
\DeclareRobustCommand{\vbar}[1]{{\vbar@accent{#1}}}							
\newcommand\vbar@accent[1]{
	\overset{%
		\text{\smash[b]{\rule[-0.4ex]{0.5pt}{0.8ex}}}
	}{#1}%
}
\makeatother																
\newcommand{\chil}[1]{\grave{#1}}	
\newcommand{\chir}[1]{\acute{#1}}
\newcommand{\chim}[1]{\vbar{#1}}
\newcommand{\chimell}{\vbar{\! \ell}}

\newcommand{\rootdel}{{\scalebox{.5}{\begin{tikzpicture}
		\begin{scope}[every node/.style={regular polygon,regular polygon sides=7,thick,draw}, minimum size=20pt]
		\node (h1) at (0,0) {};
		\end{scope}
		\node at (0,0) {\scalebox{2}{$\times$}} ;				
		\end{tikzpicture}}}}
	
\newcommand{\rootben}{{\scalebox{.5}{\begin{tikzpicture}
			\begin{scope}[every node/.style={regular polygon,regular polygon sides=7,thick,draw}, minimum size=20pt]
			\node at (0,0) {};
			\end{scope}
			\draw (0,0) circle (1.5mm)  [fill=white!100];
			\end{tikzpicture}}}}
		
\newcommand{\rootsin}{{\scalebox{.5}{\begin{tikzpicture}
			\begin{scope}[every node/.style={regular polygon,regular polygon sides=7,draw}, minimum size=15pt]
			\node (h21) at (0,0) {};
			\end{scope}
			\end{tikzpicture}} }}	
\usepackage{stmaryrd}
\usepackage{multirow}

\tikzset{
	mystyle/.style={
		circle,
		inner sep=0pt,
		text width=6mm,
		align=center,
		draw=black,
		fill=white,
		minimum size = 27
	}
}

\pgfdeclarepatternformonly{soft horizontal lines}{\pgfpointorigin}{\pgfqpoint{100pt}{1pt}}{\pgfqpoint{100pt}{3pt}}%
{
	\pgfsetstrokeopacity{0.3}
	\pgfsetlinewidth{0.1pt}
	\pgfpathmoveto{\pgfqpoint{0pt}{0.5pt}}
	\pgfpathlineto{\pgfqpoint{100pt}{0.5pt}}
	\pgfusepath{stroke}
}

\pgfdeclarepatternformonly{soft crosshatch}{\pgfqpoint{-1pt}{-1pt}}{\pgfqpoint{4pt}{4pt}}{\pgfqpoint{3pt}{3pt}}%
{
	\pgfsetstrokeopacity{0.3}
	\pgfsetlinewidth{0.4pt}
	\pgfpathmoveto{\pgfqpoint{3.1pt}{0pt}}
	\pgfpathlineto{\pgfqpoint{0pt}{3.1pt}}
	\pgfpathmoveto{\pgfqpoint{0pt}{0pt}}
	\pgfpathlineto{\pgfqpoint{3.1pt}{3.1pt}}
	\pgfusepath{stroke}
}

\usepackage{makerobust}
\MakeRobustCommand\rotatebox
\MakeRobustCommand\reflectbox
\usepackage[numbers,square,sort]{natbib}
\usepackage{tabularx}
\usepackage{enumitem}  
\usepackage[T1]{fontenc}
\usepackage{url}  

\theoremstyle{plain}
\newtheorem{theorem}{Theorem}[section]
\newtheorem{lem}[theorem]{Lemma}
\newtheorem{proposition}[theorem]{Proposition}
\newtheorem{coro}[theorem]{Corollary}
\theoremstyle{definition}
\newtheorem{definition}[theorem]{Definition}

\theoremstyle{remark}
\newtheorem{remark}[theorem]{Remark}

\numberwithin{equation}{section}

\renewcommand{\leq}{\leqslant}
\renewcommand{\geq}{\geqslant}

\newcommand{\pitch}{{\textrm{\reflectbox{\rotatebox[origin=c]{180}{$\pitchfork$}}}}}
\newcommand{\pA}{{\reflectbox{\rotatebox[origin=c]{270}{\textrm{\Leftscissors}}}}}

\newcommand{\Nb}  {{\mathbb N}}
\newcommand{\Rb}  {{\mathbb R}}

\newcommand{\Zb}  {{\mathbb Z}}
\newcommand{\Pb}  {{\mathbb P}}
\renewcommand{\P}{\mathbb P}

\newcommand{\Cs} {{\mathcal C}}

\newcommand{\Fs} {{\mathcal F}}
\newcommand{\Gs} {{\mathcal G}}
\newcommand{\Hs} {{\mathcal H}}

\newcommand{\Os} {{\mathcal O}}

\newcommand{\Rs} {{\mathcal R}}

\newcommand{\Ts} {{\mathcal T}}
\newcommand{\Te} {{\mathscr T}}

\DeclareMathOperator{\Attr}{Attr}
\newcommand{\al}{\alpha}

\renewcommand{\phi}{\varphi}

\newcommand{\la}{\lambda}

\newcommand{\E}{\mathbb E}
\newcommand{\R}{\mathbb R}

\newcommand{\N}{\mathbb N}

\newcommand\F{\mbox{I\kern-2pt F}}
\newcommand\1{{\bf 1}}

\newcommand\tf{{\mathfrak t}}

\newcommand{\asg}{\alpha}

\newcommand{\dd}{\, \mathrm{d}}

\newcommand{\brea}{{{a}}}

\newcommand{\trip}{{\scalebox{.17}{\begin{tikzpicture}
\begin{scope}[every node/.style={circle,thick,draw}, minimum size=1pt]
				\node[fill=black!100, scale=0.5] (h00) at (0,0) {};
				\end{scope}
                    \begin{scope}[>={latex[black]}, every node/.style={fill=white,circle},
					every edge/.style={draw=black,line width=1.2mm}]
					
					\path [-] (0,0) edge (0,-0.5);
					\path [-] (0,0) edge (0.433,0.166);
					\path [-] (0,0) edge (-0.433,0.166);
					\end{scope}
					
	\end{tikzpicture}}}}
\newcommand{\ymin}{{y_{{\rm min}}^{}}}
\newcommand{\ymax}{{y_{{\rm max}}^{}}}
\newcommand{\sharrow}[1][3pt]{\mathrel{%
   \hbox{\rule[\dimexpr\fontdimen22\textfont2-.2pt\relax]{#1}{.4pt}}%
   \mkern-4mu\hbox{\usefont{U}{lasy}{m}{n}\symbol{41}}}}
   
\newcommand{\cpr}{{c^{\rotatebox[origin=c]{-90}{$\sharrow$}}}}
\newcommand{\cpv}[1]{{c_{#1}^{\rotatebox[origin=c]{-90}{$\sharrow$}}}}
\newcommand{\chr}{{\hat{c}^{\rotatebox[origin=c]{-90}{$\sharrow$}}}}
\newcommand{\chv}[1]{{\hat{c}_{#1}^{\rotatebox[origin=c]{-90}{$\sharrow$}}}}
\newcommand{\chrl}[1]{{\hat{c}^{\rotatebox[origin=c]{-90}{$\sharrow$},{#1}}}}
\newcommand{\Tp}{{\mathscr{T}}}

\newcommand{\ba}{{\bar{a}}}
\newcommand{\bal}{{\bar{\alpha}}}
\title[Lines of descent in the mutation--selection model with pairwise interaction]{Lines of descent in the deterministic mutation--selection model with pairwise interaction}
\author{Ellen Baake         \and
	Fernando Cordero \and
	Sebastian Hummel
}

\address{Faculty of Technology, Bielefeld University, Box 100131, 33501 Bielefeld, Germany}
\email{ebaake@techfak.uni-bielefeld.de} 
\email{fcordero@techfak.uni-bielefeld.de}
\email{shummel@techfak.uni-bielefeld.de} 
\date{\today}%

\begin{document}

\begin{abstract}
We consider the mutation--selection differential equation with pairwise interaction (or, equivalently, the diploid mutation--selection equation) and establish the corresponding ancestral process, which is a random tree and a variant of the ancestral selection graph. The formal relation to the forward model is given via duality. To make the tree tractable, we prune branches upon mutations, thus reducing it to its informative parts. The hierarchies inherent in the tree are encoded systematically via tripod trees with weighted leaves; this leads to the stratified ancestral selection graph. The latter also satisfies a duality relation with the mutation--selection equation. Each of the dualities provides a stochastic representation of the solution of the differential equation. This allows us to connect the equilibria and their bifurcations to the long-term behaviour of the ancestral process. Furthermore, with the help of the stratified ancestral selection graph, we obtain explicit results about the ancestral type distribution in the case of unidirectional mutation.
\end{abstract}
\maketitle

\section{Introduction}
Models of population genetics describe the evolution of biological populations under the interplay of various forces such as mutation, selection, recombination, and migration. Traditionally, they come in two categories, deterministic and stochastic. Deterministic approaches assume that the population is so large that random fluctuations may be neglected; the resulting models are (ordinary or partial) differential equations or (discrete-time) dynamical systems, describing the evolution in the usual forward direction of time. This has led to an elaborate body of theory, which is comprehensively surveyed in the monograph by B\"urger \cite{burger2000mathematical}. In contrast, stochastic approaches take into account the fluctuations due to finite population size; the resulting stochastic processes have a firm place in probability theory. Here, the corresponding \emph{ancestral processes}, which describe the ancestry of a sample of individuals from a population at  present, play an eminent role. This retrospective view is linked to the prospective one via duality relations, which have proven to serve as versatile tools to investigate the models in question. This area of research is comprehensively surveyed in the monographs by \citet{etheridge2011some} and \citet{durrett2008probability}.

\smallskip

The deterministic models of population genetics are related to their stochastic counterparts via a dynamical law of large numbers (also known as mean-field limit). Nevertheless, the two  model classes have largely led separate lives for many decades. Recently, however, a beginning has been made to build new bridges between them by introducing the genealogical picture into the deterministic equations~\citep{Baake_Baake_16,BCH17,Cordero2017590,Baake2018}. Here, the ancestral lines of an individual from the present population are described via (random) \emph{ancestral graphs}. This leads to stochastic representations of the solutions of the differential equations, thus providing new insight into the  dynamics and the long-term behaviour.  

\smallskip

Specifically, this program has been started for the \emph{mutation--selection differential equation}, one of the most well-known deterministic models of population genetics. It describes the interplay between selection (which tends to concentrate the population towards the set of fit(test) types) and mutation (which tends to randomise the population).
For example, the mutation--selection equation with unidirectional mutation and \emph{genic selection} applies when  haploid\footnote{that is, carrying only one copy of the genetic information} individuals reproduce independently of each other; the corresponding genealogical structure is the \emph{ancestral selection graph (ASG)} \citep{KroNe97,NeKro97,Ne99}. The differential equation displays a bifurcation of its equilibria; this could be explained by long-term properties of a variant of the ancestral selection graph \citep{BCH17,Cordero2017590,Baake2018}.

\smallskip

Beyond shedding new light on the solution of the differential equation, tracing back the ancestral lines may be used to determine the type distribution of the ancestors of today's population. They are of considerable interest -- after all, it is them that have been successful in the long run. The ancestral type distribution is not directly accessible in the differential equation context.  For the case of genic selection, it was investigated in \cite{HRWB,BaakeGeorgii}, building on concepts originally developed for multitype branching processes \cite{Jagers89,Jagers92}. The analysis was later complemented by an approach based on a variant of the ASG \cite{Cordero2017590,BCH17,Baake2018}. 

\smallskip

In this article, we extend the results for the mutation--selection equation with genic selection to the case with pairwise interaction between individuals. In the latter setting,  the reproduction rate depends on the type of a uniformly chosen partner. Biologically, this is a special case of \emph{frequency-dependent} selection. The resulting equation is equivalent to the \emph{diploid mutation--selection equation}, which describes individuals that carry two copies of the genetic information rather than one as in the haploid case. 

\smallskip

With pairwise interaction, the right-hand side is cubic as opposed to  quadratic  in  the case of genic selection. This leads to a richer bifurcation structure. In particular, one now observes bistability in certain parameter regions. While this is well known, the corresponding ancestral processes are largely unexplored and require new concepts. Starting from ideas in~\cite{Ne99}, we extend the ASG to the case with pairwise interaction. This results in a specific random tree marked with mutations. For our purposes, only the embedded tree structure together with the mutations is relevant; it is captured by what we call the \emph{embedded ASG}. This process satisfies a duality relation with the solution of the mutation--selection differential equation; thus leading to a stochastic representation of the solution. The underlying principle seems to be robust and has also been exploited in other stochastic models~\citep{gonzalezcasanova2018,Cordero2019}\citep[Ch.~5]{dawson2014spatial}. 

\smallskip 

The embedded ASG as such is rather unwieldy and using it to derive type and ancestral type distribution is difficult. To make things tractable, we \emph{prune} the tree upon mutations, thus reducing it to its informative parts; and we order the remaining graph and exploit a natural hierarchy in its leaves to further stratify it. This  results in a \emph{tripod tree with weighted vertices}, which we  call  \emph{stratified ASG}. The corresponding process is in duality with the forward dynamic and  is specifically tailored to determine the type distribution of a sample from the population at present. The stratified ASG is our workhorse to understand the bifurcations of the equilibria from an ancestral perspective. Indeed, it will turn out that the random genealogical trees have very different properties in the various parameter regimes.

\smallskip 

The derivation of the \emph{ancestral type distribution} requires tracing back ancestral lines beyond the time at which the type of the sample is determined. This is captured by a modification of the stratified ASG. With its help, we obtain an explicit expression for the ancestral type distribution in the biologically relevant case of unidirectional mutation to the deleterious type.

\smallskip

Our motivation to study ancestral graphs comes from population genetics. But these objects exhibit interesting connections to issues more rooted in pure probability theory. For example, our ancestral processes can be embedded into the general framework of recursive tree processes as systemically studied by \citet{MSS2018}, in parallel to our work. \citeauthor{MSS2018} identify a large class of differential equations that arise as mean-field limits of interacting particle systems on the complete graph. They apply their theory to the cooperative branching model with deaths (see also~\citep{Mach2017}), which corresponds to a special case of our mutation--selection model with interaction. The analyses via recursive tree processes and via the stratified ASG provide complementary insight. \citet{MSS2018} provide a detailed comparison of their work with ours in their Section~$2.1$. We will spell out the similarities and differences as we go along. 

\smallskip

 The article is organised as follows. Section~\ref{sec:mainresults} contains the formulation of our constructions and our main results. Proofs and more detailed results are deferred to the subsequent sections. Section~\ref{sec:detmodelandlaw} contains the proofs related to the stability analysis of the mutation--selection equation, and the proof of the law of large numbers for the underlying finite-population model. A detailed description of our ancestral processes, and the proofs for the connections between them are given in Section~\ref{sec:proofs:asg}. In Section~\ref{sec:applicationIproofs} we exploit our constructions to prove the results that lead to the probabilistic interpretation of the bifurcation structure of the mutation-selection equation. Finally,  the results related to the distribution of the ancestors of today's population are proved in Section~\ref{sec:ancestraltypedistriubtionproofs}. 

\section{Main results and constructions} \label{sec:mainresults}
\subsection{The mutation-selection equation with pairwise interaction}\label{sec:mainresults:sub:detmodel}
The general form of the mutation-selection
equation (including an arbitrary number of types) goes back to \citet{wrightadaptionselection} and is  intensively discussed by \citet{crow1956}, \citet{Akin79}, and \citet{Hofbauer85}. Starting in the 1990s, it has become a popular object of research in the physics literature, see \cite{Baakereview} for a review. The main interest is targeted towards the long-term behaviour. There may be one or several (stable or unstable) equilibria, but more complicated dynamical behaviour, such as  periodic solutions \cite{Akin79,Hofbauer85,Baake1997}, is also possible.  The model is also used for the analysis of gene-frequency data, see, e.g., \cite{Weghorn2019} and references therein. Yet another viewpoint  comes from evolutionary game theory. There, each genotype specifies a strategy played in a repeated game; the payoffs  determine the change in frequency of the strategies within the population over time (see, e.g., \citep[Ch.~22]{Hofbauer_Sigmund_03}). 

\smallskip

We now describe the version with pairwise interaction that forms the basis of our analysis. It is an ordinary differential equation (ODE) that describes the type-frequency evolution in an infinite population composed of two types, say type~$0$ and type~$1$. More precisely, if $y_0\in [0,1]$ is the initial frequency of type~$1$ and $y(t;y_0)$ is its frequency at time~$t$, then $y(t;y_0)$ is the solution of the ODE \begin{equation}\begin{aligned}
\frac{\dd y}{\dd t}(t)&=-y(t)\big(1-y(t)\big)\big[s+\gamma\big(1-y(t)\big)\big]+u\nu_1\big(1-y(t)\big)-u\nu_0y(t)\eqdef F(y(t))\
\end{aligned}\label{eq:dlimitdiffeq}\end{equation} 
that satisfies~$y(0)=y_0$, where~$u>0$,~$\gamma,s\geq 0$, and~$\nu_0,\nu_1\in[0,1]$ so that~$\nu_0+\nu_1=1$.
The underlying model is described as follows. Both types reproduce at a so-called \emph{neutral} rate of~$1$. On top of this, type~$0$ has a selective advantage reflected by an additional reproduction rate; we refer to type~$0$ as the \emph{fit} or \emph{beneficial} type, and to type~$1$ as \emph{unfit} or \emph{deleterious}.
The additional reproduction rate has two contributions: one depends on the current type frequencies, and one is independent of it. The former is called \emph{interactive} reproduction, and the latter \emph{selective} reproduction. The rate of selective reproduction is~$s$. Interactive reproduction occurs at rate~$\gamma (1-y(t;y_0))$, where~$\gamma$ is the interaction parameter. The interaction is called \emph{pairwise} because the rate reflects that a type-0 individual reproduces if a randomly-chosen partner is also of type 0. This is a special case of \emph{frequency-dependent selection}. Both types may mutate at rate~$u$, the resulting type being~$0$ (resp.~$1$) with probability~$\nu_0$ (resp.~$\nu_1$). In the ODE, the last two terms correspond to mutation; the first term describes the loss of type-$1$ individuals if type-$0$ individuals reproduce selectively or interactively  and the offspring replaces a type-1 individual. The neutral reproduction does not enter the equation since its net contribution is~$-(1-y(t))y(t)+y(t)(1-y(t))=0$; as a matter of fact, therefore, the same ODE results for any neutral reproduction rate~$c\geq 0$. 

\begin{remark}
Let us briefly connect the mutation-selection equation with pairwise interaction to the  \emph{diploid mutation-selection equation}. In the latter, one identifies the individuals in a diploid population with their genotypes, where a genotype is a pair~$(i,j), i,j \in \{0,1\}$, and $i$ and $j$ are combined independently. The corresponding reproduction rates $w_{ij}$ are~$w_{00}=1+2s+\gamma$, $w_{01}=w_{10}=1+s$, and~$w_{11}=1$; this choice of parameters corresponds to the case where type~$0$ is (partially) recessive, that is, it needs another~$0$ to fully play out its advantage (see also \cite{Baake1997}). The ODE~\eqref{eq:dlimitdiffeq} then describes the proportion of 1's averaged over all genotypes in the population.
\end{remark}

\begin{remark}\label{rem:cooperativebranching}
For a special choice of parameters, the ODE~\eqref{eq:dlimitdiffeq} corresponds to the mean-field limit (or law of large numbers) of the cooperative branching model on the complete graph as investigated by \citet{MSS2018} (see also \citep{Mach2017}). In the underlying interacting particle system, every pair of particles produces, independently at a rate proportional to some~$\alpha$, a new particle at another site if this site is empty; this is called a cooperative branching event. Every particle independently dies at rate~$1$. The authors identify occupied sites with~$1$ and free sites with~$0$. Our notation translates to theirs by interchanging the roles of type~$0$ and~$1$ (their particles are our fit individuals) and by setting~$\gamma=\alpha$,~$u=1$,~$\nu_0=0$,~$\nu_1=1$, and~$s=0$. This leads to the mean-field equation \cite[Eq. (1.36)]{MSS2018}.
\end{remark}

Existence and uniqueness of a global solution to~\eqref{eq:dlimitdiffeq} such that $y(0)=y_0\in\Rb$, and the positive invariance of $[0,1]$ (which is the biologically relevant domain) follow from standard theory. We will be particularly interested in the long-term behaviour of such a solution, which is determined by the equilibria of~\eqref{eq:dlimitdiffeq} and their stabilities.

\subsection*{Equilibria and bifurcation structure} \label{sec:equilibriabifurcationstructure}
\begin{figure}[t]
	\centering\scalebox{.5}{
		\includegraphics{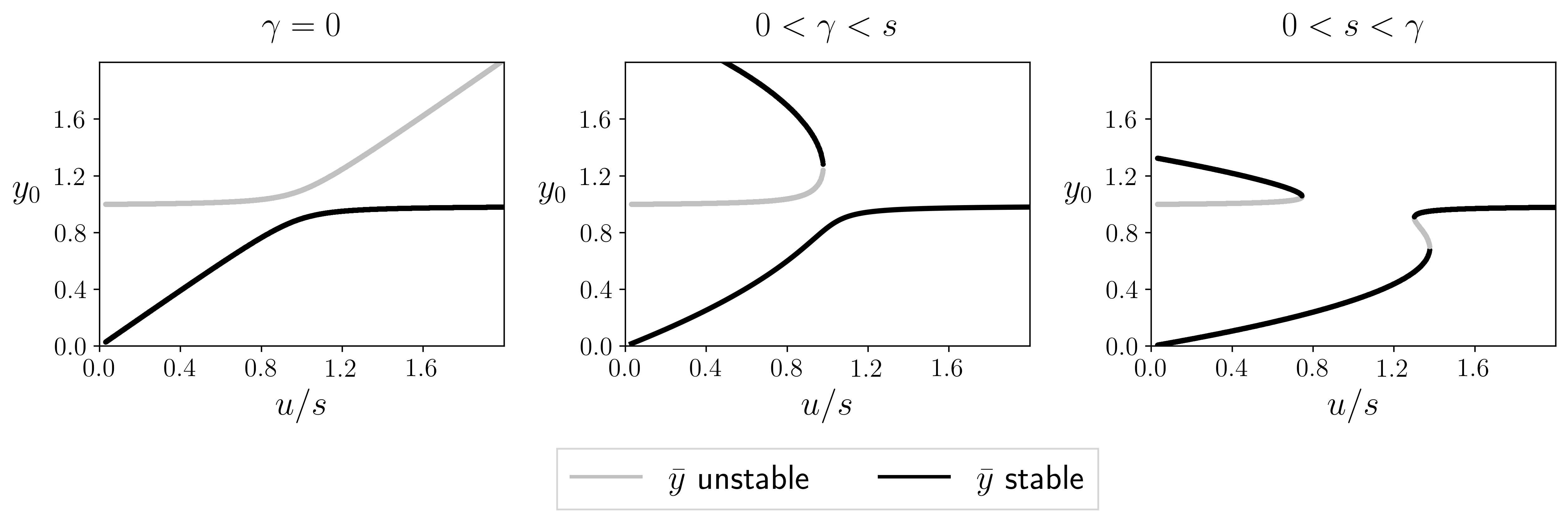}}
	\caption[Equilibria of mutation--selection equation with pairwise interaction ($\nu_0>0$)]{Equilibria of~\eqref{eq:dlimitdiffeq} evaluated numerically as functions of~$u/s$ for~$\nu_0=1/100$ and~$s=1/30$. The left, middle, and right cases correspond to~$\gamma=0$, $\gamma=1/40$, and $\gamma=1/10$, respectively.}
	\label{fig:plotnu0}
\end{figure}
We assume throughout that $u>0$. We now analyse the equilibria of the ODE \eqref{eq:dlimitdiffeq}, namely the (real) roots of $F$. We also discuss the stability of the equilibria in~$[0,1]$. We say that an equilibrium $\bar{y}$ is stable\footnote{also referred to as locally asymptotically stable or attracting} (resp. stable in an interval~$I\ni\bar{y}$) whenever there is $\varepsilon>0$ such that for all $y_0\in [\bar{y}-\varepsilon,\bar{y}+\varepsilon]$ (resp. $y_0\in[\bar{y}-\varepsilon,\bar{y}+\varepsilon]\cap I$), $\lim_{t\to\infty}y(t;y_0)=\bar{y}$. We say that an equilibrium $\bar{y}\in[0,1]$ is attracting from the left (resp. right) if $\bar{y}$ is stable in $(0,\bar{y}]$ (resp. $[\bar{y},1)$). An equilibrium $\bar{y}$ is unstable or repelling if there is $\varepsilon>0$ such that for some $y_0\in [\bar{y}-\varepsilon,\bar{y}+\varepsilon]$, $y(t;y_0) \notin [\bar{y}-\varepsilon,\bar{y}+\varepsilon]$ for some $t$, with the obvious extensions to instability in an interval, and being repelling to the left and right.
\smallskip

 Since~$F$ is continuous with $F(0)\geq 0$ and $F(1)\leq 0$, the ODE \eqref{eq:dlimitdiffeq} has at least one equilibrium in~$[0,1]$. Let~$\ymin$ and~$\ymax$ be the smallest and largest equilibrium in~$[0,1]$, respectively. Let us briefly discuss the case of bidirectional mutation, i.e. $\nu_0\in (0,1)$ (we will not explicitly discuss the case $\nu_0=1$, the biologically reasonable regime is $\nu_0\ll 1$). In this case, $\ymin$ and~$\ymax$ are both in~$(0,1)$ and they are attracting from the left and right, respectively. If there is an additional equilibrium between them, $F$ has a positive derivative at this point so that this equilibrium is unstable. Figure~\ref{fig:plotnu0} illustrates the equilibria and their stabilities if $\nu_0\in (0,1)$. For the remainder of Section \ref{sec:mainresults:sub:detmodel}, we assume $\nu_0=0$. 

\begin{proposition}[Equilibria and stability]\label{prop:equilibriumpoints}
	Suppose~$\nu_0=0$ and~$u>0$. If~$\gamma=s=0$, then $1$ is the only equilibrium and it is stable. If~$\gamma=0$ and $s>0$, the ODE~\eqref{eq:dlimitdiffeq} has equilibria $1$ and $y_-=u/s$. The minimum of the two is always stable in $[0,1]$; the other one (if distinct) is unstable (see Fig. \ref{fig:plots_l_gamma}, left). If $\gamma>0$ and $s\geq 0$, define
\[
  u^\star={u^\star}(\gamma,s)\defeq\frac{1}{\gamma}\left(\frac{s+\gamma}{2}\right)^2.
\]
Then for $u>{u^\star}$, $1$ is the only equilibrium and it is stable. For~$u\leqslant {u^\star}$, the ODE has equilibria
\begin{equation}\label{eq:equilibria}
1,\quad\bar{y}_-=\frac{\gamma+s-\sqrt{(\gamma+s)^2-4u\gamma}}{2\gamma} \quad \text{and}\quad
\bar{y}_+=\frac{\gamma+s+\sqrt{(\gamma+s)^2-4u\gamma}}{2\gamma}. 	
\end{equation}
The positions of the equilibria and their stabilities are summarised in Table~\ref{table:stability}.
\end{proposition}
\begin{table}[h!]
	\begin{center}
	\scalebox{.81}{
		\begin{tabular}{|c|c|c|c|c|c|c|}
			\hline
			& $\gamma<s$ & $\gamma=s$ & $\gamma>s$ \\
			\hline
			$u<s$ & $0<\bm{\bar{y}_-}<1<\bar{y}_+$ & $0<\bm{ \bar{y}_-}<1 <\bar{y}_+$ & $0<\bm{ \bar{y}_-}<1 <\bar{y}_+$\\
			$u=s$ & $\bm{ 1}=\bm{\bar{y}_{-}}<\bar{y}_+=s/\gamma$  & $\bm{ 1}=\bm{ \bar{y}_-}=\bm{ \bar{y}_+}$  & $0<\bm{ \bar{y}_-}=\bm{s/\gamma}  <1=\bar{y}_+$\\
			$s < u < {u^\star} $ &  $\mathbf{ 1}<\bar{y}_-<\bar{y}_+$&  --- &  $0<\bm{ \bar{y}_-}<\bar{y}_+<\mathbf{ 1}$ \\
			$u={u^\star}$ & $\bm{ 1}\leq \bar{y}_-=\bar{y}_+=(\gamma+s)/(2\gamma)$ & $\bm{ 1}=\bm{ \bar{y}_-}=\bm{ \bar{y}_+} $  & $\frac{1}{2}\leq \bar{y}_-=\bar{y}_+= (\gamma+s)/(2\gamma)<\bm{ 1}$ \\
			\hline
	\end{tabular}}
\end{center} \caption{Summary of the stability of the equilibria for $\nu_0=0$ and $0<u\leq{u^\star}$. Bold indicates equilibria that are stable in~$[0,1]$.}\label{table:stability}
\end{table}
The proof of Proposition \ref{prop:equilibriumpoints} will be given in Section~\ref{sec:detmodelandlaw}; the function $\gamma\mapsto u^\star(\gamma,s)$ is illustrated in Fig.~\ref{fig:u_star}. 

\smallskip

We now describe the long-term behaviour of the solutions of \eqref{eq:dlimitdiffeq}. To this end, it is convenient to introduce subsets of $(0,\infty)\times [0,\infty)^2$ that partition the parameter set according to the number of equilibria in~$[0,1]$ if~$\nu_0=0$ (see Figs.~\ref{fig:u_star} and~\ref{fig:plots_l_gamma}),
 \begin{equation}
\begin{aligned}
 \Theta_1^{}&=\{(u,s,\gamma):\ u>{u^\star}  \ \text{or} \   \gamma\leq s\leq u\leq{u^\star}\}, & \qquad \Theta_{2}^a&=\{(u,s,\gamma):  u<s \ \text{or}\ u=s<\gamma\},\\
  \Theta_{2}^b&=\{(u,s,\gamma):\ s<\gamma\ \text{and} \ u={u^\star} \},  & \qquad \Theta_{3}&=\{(u,s,\gamma):\ u\in(s,{u^\star}) \ \text{and} \ s<\gamma\},
\end{aligned}\label{eq:parameterregions}
\end{equation}
where if $\gamma=0$, $u^\star$ is to be understood as~$\infty$.

\smallskip

\begin{figure}[t]
	\centering\scalebox{.5}{
		\includegraphics{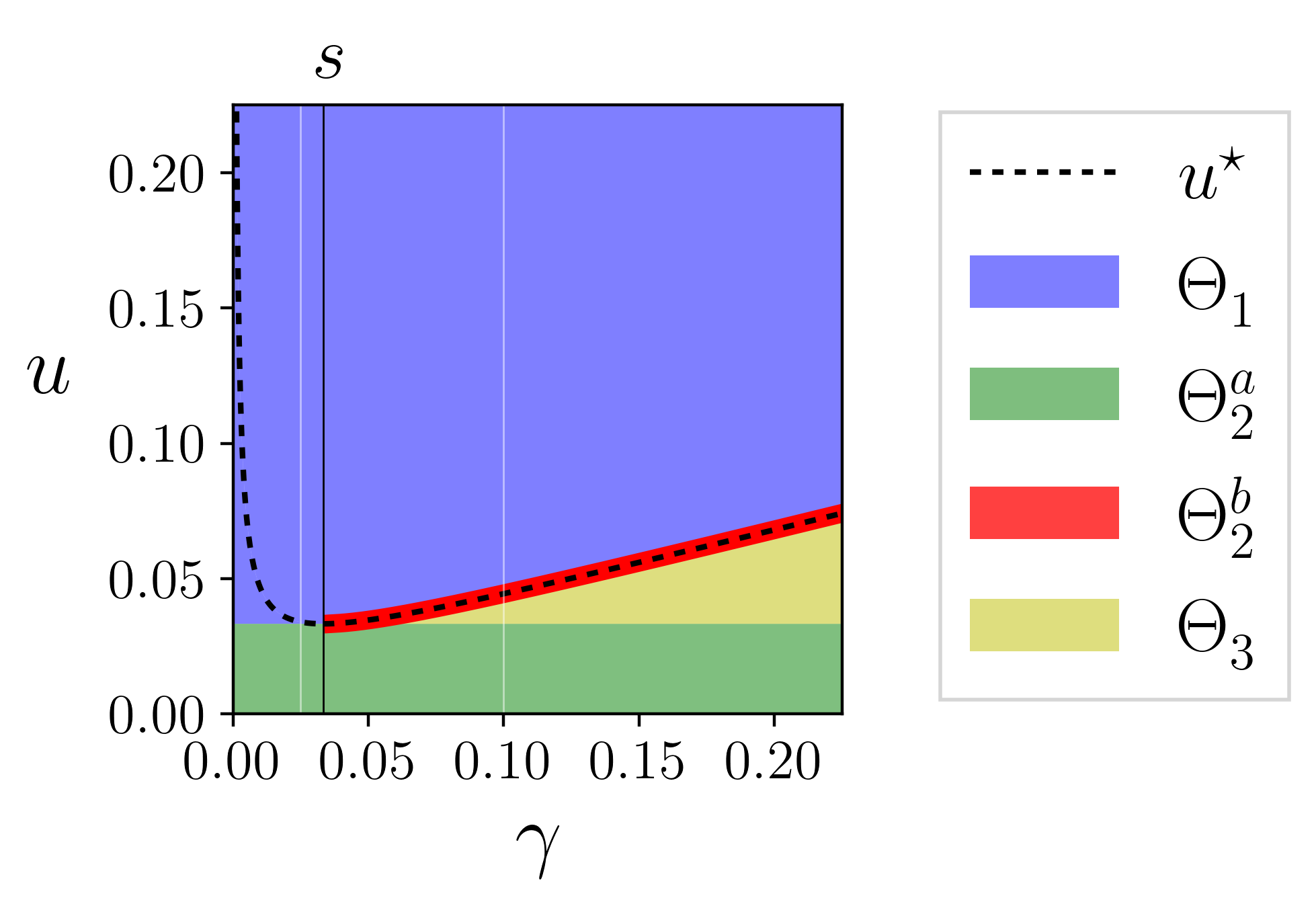}}
	\caption[Critical mutation rate]{Critical mutation rate~$u^\star$ as a function of $\gamma$ for $s=1/30$. Grey vertical lines indicate $\gamma=1/40$ and $\gamma=1/10$ to compare with Figure~\ref{fig:plots_l_gamma}.}
	\label{fig:u_star}
\end{figure}

 \begin{figure}[b]
	\centering\scalebox{.46}{
		\includegraphics{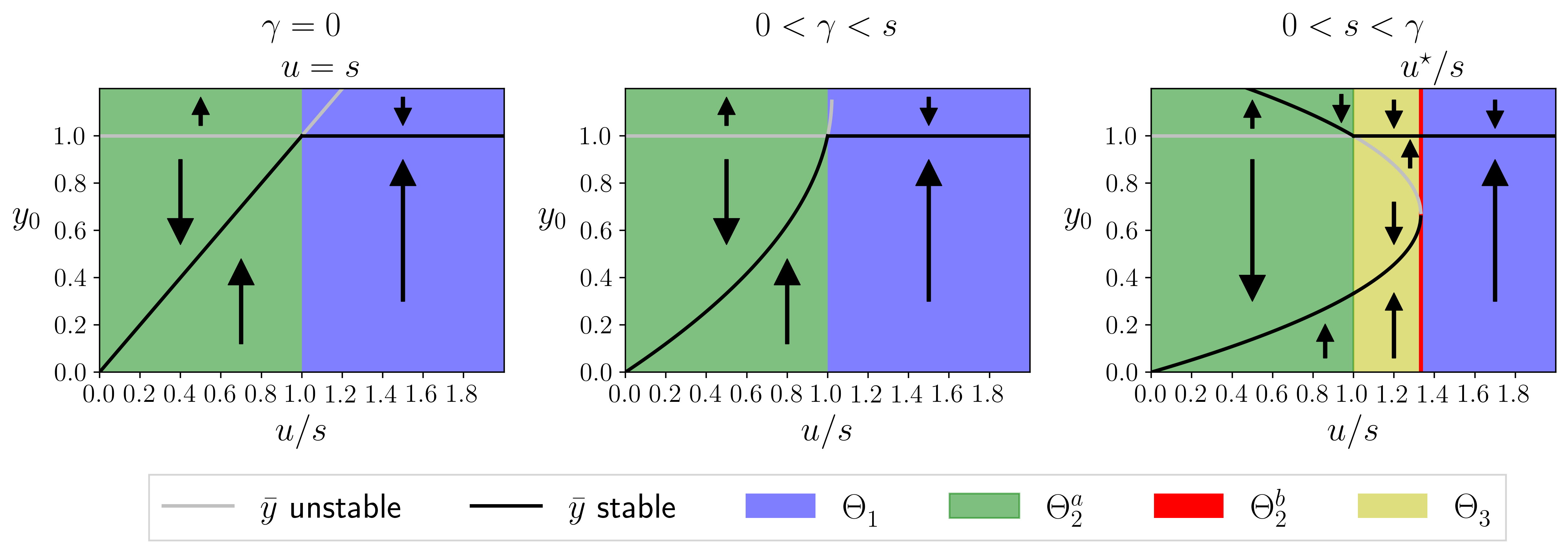}}
	\caption{The equilibria~$\bar{y}$ of~\eqref{eq:dlimitdiffeq} as functions of~$u/s$ for~$\nu_0=0$ and~$s=1/30$. The left, middle, and right panels correspond to~$\gamma=0$,~$\gamma=1/40$, and~$\gamma=1/10$, respectively. Arrows indicate the directions of attraction.}
	\label{fig:plots_l_gamma}
\end{figure}

Note that in $\Theta_1^{}$, only~$1$ is stable in $[0,1]$. In $\Theta_2^a$, $\ymin=\bar{y}_-$ is stable, and~$\ymax=1$ is unstable in $[0,1]$. In $\Theta_2^b$, 
$\ymin=\bar{y}_-=\bar{y}_+$ is attracting (only) from the left and~$\ymax=1$ is stable. In $\Theta_3^{}$, $\ymin=\bar{y}_-$ and~$\ymax=1$ are stable, and $\bar{y}_+\in(\ymin,\ymax)$ is unstable.
\smallskip

The monotonicity of~$y(\,\cdot\,;y_0)$ implies that $y_{\infty}^{}(y_0)\defeq\lim_{t\to\infty}y(t;y_0)$ exists and is always an equilibrium; this straightforwardly leads to the following corollary, which we state without proof.
\begin{coro}[Convergence]\label{coro:convergencetoequilibrium}
	Assume~$\gamma>0$ and $\nu_0=0$. In $\Theta_1^{}$, $y_{\infty}^{}(y_0)=1$ for all $y_0\in[0,1]$. In $\Theta_2^a$, $y_{\infty}^{}(1)=1$ and for 
	$y_0\in [0,1)$, $y_\infty^{}(y_0)=\bar{y}_-$.
	In $\Theta_2^b\cup\Theta_3^{}$, $$y_{\infty}^{}(y_0)=\begin{cases}
	\bar{y}_-, &\text{if }y_0\in[0,\bar{y}_+),\\
	\bar{y}_+, &\text{if }y_0=\bar{y}_+,\\
	1, &\text{if }y_0\in (\bar{y}_+,1].
	\end{cases}$$
\end{coro}

\smallskip

Let us recapitulate from \citep[Sect.~3]{Baake1997} the biological implications of Corollary~\ref{coro:convergencetoequilibrium}. In $\Theta_1^{}$, the fit type goes extinct regardless of its initial frequency. In $\Theta_2^a$, the fit type persists if its initial frequency is positive. In $\Theta_2^b\cup\Theta_3^{}$, it persists if its initial frequency is not below $1-\bar{y}_+$; otherwise it goes extinct. In particular, a beneficial mutant arising in small frequency in a population that is otherwise unfit dies out.
\smallskip

Recall that one speaks of a \emph{bifurcation} whenever  the variation of the parameter(s) of a differential equation leads to a qualitative change in the long-term behaviour of its solutions; the parameter values at which such a phenomenon occurs are called \emph{bifurcation values} -- see~\citep[Ch.~3]{guckenheimer2013nonlinear} for a general account of bifurcation theory for equilibria of ODEs. 
We close this section describing the bifurcations of~\eqref{eq:dlimitdiffeq} with~$u$ as the bifurcation parameter; see also \citep[Sect.~3]{Baake1997} and Fig. \ref{fig:plots_l_gamma} for an illustration. For~$0<\gamma<s$ (resp. $0<s<\gamma$), there is an \emph{exchange of stability} of the equilibria~$1$ and~$\bar{y}_-$ (resp. of $1$ and~$\bar{y}_+$) at~$u=s$; this is also known as a \emph{transcritical bifurcation}. At $u=s$, $1$ switches from unstable to stable; whereas $\bar{y}_-$ (resp. $\bar{y}_+$) switches from stable to unstable. If $\gamma<s$, this is an instance of the so-called error threshold \cite{Eigen1971}. For~$s<\gamma$, there is an additional \emph{saddle-node bifurcation} at~$u={u^\star}$, where the equilibria~$\bar{y}_-$ and~$\bar{y}_+$ (one stable, one unstable) collide and both vanish. If~$s={u^\star}$, we see a \emph{pitchfork bifurcation}, where the unstable equilibrium $1$ passes through the collision point of~$\bar{y}_-$ and~$\bar{y}_+$ and becomes stable (this corresponds to the simultaneous occurrence of the saddle node and the transcritical bifurcation). It is one of our main goals in the subsequent analysis to interpret the equilibria of the mutation-selection equation and their bifurcations in terms of an appropriate ancestral structure.

\begin{remark}
Recall from Remark~\ref{rem:cooperativebranching} the setting in the example of~\cite{MSS2018}. In this case, $u^\star=\gamma/4$ so that $\Theta_2^a$ is empty. In particular, if $\gamma<4$,  there is a unique, stable equilibrium. If $\gamma>4$,  there are three equilibria, two are stable and one is unstable.
\end{remark}

\subsection{Moran model with mutation, selection, and pairwise interaction and its ASG}\label{sec:mainresults:subsec:mm}
Before establishing ancestral structures for the mutation-selection equation, we connect the ODE to the Moran model, which has such a structure naturally embedded. We briefly recapitulate it in our setting.

\smallskip

The two-type Moran model with mutation, selection, and pairwise interaction describes the evolution of a population of $N\in \N$ haploid individuals in continuous time. It shares the types (0 and 1) and the set of parameters with the mutation--selection equation, i.e. $u>0$, $s,\gamma\geq 0$, and $\nu_0,\nu_1\in[0,1]$ with $\nu_0+\nu_1=1$. When an individual reproduces, its single offspring inherits the parent's type and replaces a uniformly chosen individual so that the population size remains constant. All individuals reproduce independently at the neutral rate~$1$; but individuals of type $0$ independently reproduce at an additional rate $s+\gamma(N-k)/N$ with $k$ the current number of type-$1$ individuals, where $(N-k)/N$ reflects the probability that a uniformly chosen partner individual is fit. In particular, the fit type can reproduce by interacting with another fit individual. Each individual mutates at rate~$u$; its type after the event is~$i$ with probability~$\nu_i,\ i\in \{0,1\}$.

\smallskip

Let~$Y_t^{(N)}$ be the (random) number of type-$1$ individuals at time~$t$ in a population of size~$N$. The process~${Y_{}^{(N)}\defeq (Y_t^{(N)})_{t\geq 0}}$ is a continuous-time birth-death process with transition rates
\[ q_{Y^{(N)}}^{}\left(k,k+1\right)=k\,\frac{N-k}{N}+(N-k)u \nu_1,\qquad 
q_{Y^{(N)}}^{}\left(k,k-1\right)=k\,\frac{N-k}{N}\left(1+s +\gamma \, \frac{N-k}{N}\right)+ku \nu_0,\] 
where~$k\in [N]_0\defeq [N]\cup\{0\}$ with $[N]\defeq\{1,\ldots,N\}$. 

\begin{remark}\label{sec:mainresults:subsec:mm:rem:parameterregimecooperativebranching}
Recall from the discussion following~\eqref{eq:dlimitdiffeq} that the Moran model with pairwise interaction can be translated into the cooperative branching process on a complete graph~\citep[Chs.~I.1.2.3 and~I.2.1]{Mach2017}, see also \citep{MSS2018}. A deleterious mutation corresponds to a death event and an interactive reproduction to a cooperative branching event. Other variants of such dynamics may be found in~\citep{noble1992,neuhauser94,sturm2015particle}.
\end{remark}
\begin{figure}[t]
	
	\centering
	\scalebox{0.7}{\begin{tikzpicture}
		\draw[line width=.5mm ] (14.5,4) -- (0,4);
		\draw[line width=.5mm ] (12.9,1) -- (1,1);
		\draw[line width=.5mm ] (8.5,2) -- (3.6,2);
		\draw[line width=.5mm ] (9.4,1) -- (1,1);
		\draw[line width=.5mm ] (8.5,3) -- (0,3);
		\draw[line width=.5mm ] (2.9,0) -- (0,0);
		
		\draw[line width=.5mm ] (12.9,1) -- (12.9,3.75);
		\draw[line width=.5mm ] (2.9,3.1) -- (2.9,3.75);
		\draw[line width=.5mm ] (2.9,2.9) -- (2.9,1.1);
		\draw[line width=.5mm ] (2.9,.9) -- (2.9,0);
		\draw[line width=.5mm ] (1,1.2) -- (1,3);
		\draw[line width=.5mm ] (6.2,4) -- (6.2,3.1);
		\draw[line width=.5mm ] (6.2,2.9) -- (6.2,2.1);
		\draw[line width=.5mm ] (6.2,1.9) -- (6.2,1.25);
		\draw[line width=.5mm ] (3.6,1) -- (3.6,1.8);
		
		\node[opacity=0.4] at (11,3) {\scalebox{2}{$\times$}} ;
		\node[opacity=0.4] at (13.3,1) {\scalebox{2}{$\times$}} ;
		\node[opacity=0.4] at (12,0) {\scalebox{2}{$\times$}} ;
		\node at (7.7,4) {\scalebox{2}{$\times$}} ;
		\node[opacity=0.4] at (2,2) {\scalebox{2}{$\times$}} ;
		\node[opacity=0.4] at (5,0) {\scalebox{2}{$\times$}} ;    
		\draw[dashed] (0,-0.5) --(0,4.5);
		\draw[dashed] (14.5,-0.5) --(14.5,4.5);    
		\node [right] at (0,-0.5) {$0$};
		\node [right] at (0,4.5) {$t$};
		\node [right] at (14.5,4.5) {$0$};
		\node [right] at (14.5,-0.5) {$t$};
		\draw[-{angle 60[scale=5]}] (5.25,-0.6) -- (9.25,-0.6) node[text=black, pos=.5, yshift=6pt]{};
		\draw[-{angle 60[scale=5]}] (9.25,4.5) -- (5.25,4.5) node[text=black, pos=.5, yshift=6pt]{$r$};
		
		\draw[opacity=0.4] (0,0) -- (14.5,0);
		\draw[opacity=0.4] (0,1) -- (14.5,1);
		\draw[opacity=0.4] (14.5,2) -- (0,2);
		\draw[opacity=0.4] (14.5,3) -- (0,3);
		\draw[opacity=0.4] (0,4) -- (14.5,4);
		\draw[-{triangle 45[scale=5]},semithick,opacity=0.4] (10,4) -- (10,3);
		\draw[-{triangle 45[scale=5]},semithick,opacity=0.4] (11.5,1) -- (11.5,0);
		\draw[-{triangle 45[scale=5]},thick,opacity=1] (3.6,1) -- (3.6,2);
		\draw[-{triangle 45[scale=5]},semithick,opacity=1] (1,3) -- (1,1);
		\draw[-{open triangle 45[scale=5]},thick,opacity=1] (12.9,1) -- (12.9,4);
		\draw[-{open triangle 45[scale=5]},thick,opacity=1] (2.9,3.1) -- (2.9,4);
		\draw[-{open triangle 45[scale=5]},thick,opacity=1] (6.2,1.5) -- (6.2,1);
		\draw[semithick,opacity=0.4] (7,2) -- (7,1.1);
		\draw[-{open triangle 45[scale=5]},semithick,opacity=0.4] (7,.9) -- (7,0);
		\coordinate (S2)  at (2.9,3);
		\coordinate (I12)  at (2.9,3.1);
		\coordinate (I22)  at (2.9,2.9);
		
		\coordinate (S4)  at (2.9,1);
		\coordinate (I14)  at (2.9,1.1);
		\coordinate (I24)  at (2.9,0.9);
		
		\coordinate (S1)  at (6.2,3);
		\coordinate (I11)  at (6.2,3.1);
		\coordinate (I21)  at (6.2,2.9);
		
		\coordinate (S3)  at (6.2,2);
		\coordinate (I13)  at (6.2,2.1);
		\coordinate (I23)  at (6.2,1.9);

		\draw[-{Stealth[length=2mm,width=2mm,open]},thick,opacity=1, opacity=.4] (12.4,2) -- (12.4,3.1);
		\draw[opacity=0.4] (12.4,2) circle (.6mm)  [fill=black!100];
		\fill[white, opacity=.4, draw=black] (12.3,2.9) rectangle (12.5,3.1); 
		\filldraw[semithick,opacity=0.4,white, draw=black] (12.4,1.1) -- (12.5,1) -- (12.4,0.9) -- (12.3,1) -- (12.4,1.1);
		
		\draw[semithick, opacity=0.4] (12.4,1) .. controls (12.4,1.5) and (12.55,1.5) .. (12.55,2);
		\draw[semithick, opacity=0.4] (12.55,2) .. controls (12.55,2.5) and (12.4,2.5) .. (12.4,3);

		\draw (8.5,3) circle (.6mm)  [fill=black!100];
		\fill[white, draw=black] (8.4,3.9) rectangle (8.6,4.1); 
		\draw[-{Stealth[length=2mm,width=2mm,open]},opacity=1,line width=.5mm] (8.5,3) -- (8.5,4.1);
		
		\draw[line width=.5mm] (8.5,2) .. controls (8.5,2.5) and (8.65,2.5) .. (8.65,3);
		\draw[line width=.5mm] (8.65,3) .. controls (8.65,3.5) and (8.5,3.5) .. (8.5,4);

		\filldraw[white, draw=black] (8.5,2.1) -- (8.6,2) -- (8.5,1.9) -- (8.4,2) -- (8.5,2.1);
		
		\draw[line width=.5mm, opacity=0.4] (7,0.875) arc(-100:100:.125) ;
		\draw[line width=.5mm ] (2.9,2.875) arc(-100:100:.125) ;
		\draw[line width=.5mm ] (2.9,0.875) arc(-100:100:.125) ;
		\draw[line width=.5mm ] (6.2,2.875) arc(-100:100:.125) ;
		\draw[line width=.5mm ] (6.2,1.875) arc(-100:100:.125) ;
		
		\draw (2.2,0) circle (1.5mm)  [fill=white!100];    
		\draw (3.2,4) circle (1.5mm)  [fill=white!100];
		\draw (9.4,1)[opacity=1] circle (1.5mm)  [fill=white!100];
		\draw (10.4,2)[opacity=0.4] circle (1.5mm)  [fill=white!100];
		\draw (1.5,-1.1) circle (1.5mm)  [fill=white!100] node at (1.5,-1.1) [right] {\ Mutation to type~$0$};
		\node at (1.5 ,-1.6) {\scalebox{2}{$\times$}} node at (1.5,-1.6)[right] {\ Mutation to type~$1$};    
		\draw[-{open triangle 45[scale=2.5]},semithick] (5.5,-1.1) -- (6,-1.1) node [right] {Selective arrow};
		\draw[-{triangle 45[scale=2.5]},semithick] (5.5,-1.6) -- (6,-1.6) node [right] {Neutral arrow};

		\fill[white, draw=black] (9.8,-1.2) rectangle (9.6,-1); 
		\draw[-{Stealth[length=2mm,width=2mm,open]},opacity=1,line width=.5mm] (9.3,-1.1) -- (9.8,-1.1)  node [right] {Interactive arrow};

		\fill[white, draw=black] (9.8,-1.7) rectangle (9.6,-1.5); 
		\draw[-{Stealth[length=2mm,width=2mm,open]},opacity=1,line width=.5mm] (9.3,-1.6) -- (9.8,-1.6)  node [right] {Checking arrow};
		\filldraw[fill=white!100, draw=black] (9.3,-1.5) -- (9.4,-1.6) -- (9.3,-1.7) -- (9.2,-1.6) -- (9.3,-1.5);
		\draw[opacity=1] (9.3,-1.1) circle (.6mm)  [fill=black!100];
		\draw (1,-0.8) -- (1,-1.9) -- (13,-1.9) -- (13,-0.8) -- (1,-0.8);
		\end{tikzpicture}    }
	\caption[Graphical representation of the Moran interacting particle system and its ASG]{A realisation of the Moran interacting particle system (all lines) for a population of size~$N=5$ and the ASG (bold lines) for a sample of size~$1$. Time runs forward in the Moran model ($\rightarrow$) and backward in the ASG ($\leftarrow$). An arrowhead inscribed into a square marks the joint tip of an interactive and a checking arrow.}
	\label{sec:mainresults:sec:mm:fig:graphrepASG}
\end{figure}
The Moran model with pairwise interaction has a well-known graphical representation as an interacting particle system, see Fig.~\ref{sec:mainresults:sec:mm:fig:graphrepASG}. Here, individuals are represented by pieces of horizontal lines. Time runs from left to right in the figure. We first describe the \emph{untyped} version, where no types have been assigned. Potential reproduction events are depicted by arrows between the lines with the potential parent at the tail of the arrow. If the arrow is \emph{used} to produce offspring, this offspring replaces the individual at the tip. We decompose reproduction events into neutral, selective, and interactive ones. Neutral arrows appear at rate~$1/N$ per ordered pair of lines; selective arrows appear at rate~$s/N$ per ordered pair. Interactive arrows occur at rate~$\gamma/N$ per ordered pair of lines and  are always accompanied by a checking arrow. This arrow shares the tip with the corresponding interactive arrow; but its tail is connected to a uniformly chosen line. That is, these arrow pairs occur at rate~$\gamma/N^2$ per triple of lines. All kinds of arrows (including the interactive/checking pairs) are laid down via Poisson point processes independently of each other. 
The rules for their use are as follows. All individuals use the neutral arrows. In addition, fit individuals use selective arrows. Interactive arrows are used by fit individuals if there is a fit individual at the tail of the associated checking arrow. Note that whenever an individual uses a neutral or selective arrow, it becomes the parent of the individual at the tip. This naturally introduces the concept of ancestry into the graphical representation. 

\smallskip

It will be helpful to give names to the lines that are involved in reproduction events. The line to the right of the tip of a selective or interactive arrow  carries a descendant and will be referred to as the \emph{descendant line}, the line left of the tip is called the~\emph{continuing line}, and the line at the tail of this arrow is called the \emph{incoming line}. Moreover, the line at the tail of a checking arrow is called the \emph{checking line}. Figure~\ref{fig:peckingordersel} and~\ref{fig:peckingorderint} illustrate the propagation of types and ancestry; together with the names of the lines.

\begin{figure}[b]
	\begin{center}
			\begin{minipage}{0.23 \textwidth}
			\centering
			\scalebox{1}{\begin{tikzpicture}
					\draw[black,line width=.5mm] (0,0)-- (2,0);
					\node at (1,0) {\scalebox{2}{$\times$}} ;
					\node[left] at (0,0) {$\star$};
					\node[right] at (2,0) {$1$};
			\end{tikzpicture}}
		\end{minipage}\hfill
		\begin{minipage}{0.23 \textwidth}
			\centering
			\scalebox{1}{\begin{tikzpicture}
					\draw[black,line width=.5mm] (0,0)-- (2,0);
					\draw (1,0) circle (1.5mm)  [fill=white!100]; 
					\node[left] at (0,0) {$\star$};
					\node[right] at (2,0) {$0$};
			\end{tikzpicture}}
		\end{minipage}\hfill
		\begin{minipage}{0.23 \textwidth}
			\centering
			\scalebox{1}{
				\begin{tikzpicture}
					\draw[line width=0.5mm, dotted] (0,0.7) -- (1,0.7);
					\draw[line width=0.5mm] (2,0.7) -- (1,0.7);
					\draw[line width=0.5mm] (0,0) -- (1,0);
					\draw[line width=0.5mm] (1,0) -- (1,0.45);
					\draw[-{open triangle 45[scale=5]},thick,color=black] (1,0) -- (1,0.7) node[text=black, pos=.6, xshift=7pt]{};
					\node[left] at (0,.7) {$\star$};
					\node[left] at (0,0) {$0$};
					\node[right] at (2,0.7) {$0$};
					\node[above] at (1.6,0.7) {D};
					\node[above] at (0.4,0.7) {C};
					\node[above] at (0.4,0) {I};
			\end{tikzpicture}}
		\end{minipage}\hfill
		\begin{minipage}{0.23 \textwidth}
			\centering
			\scalebox{1}{
				\begin{tikzpicture}
					\draw[line width=0.5mm] (0,0.7) -- (1,0.7);
					\draw[line width=0.5mm] (2,0.7) -- (1,0.7);
					\draw[line width=0.5mm,dotted] (0,0) -- (1,0);
					\draw[line width=0.5mm,dotted] (1,0) -- (1,0.45);
					\draw[-{open triangle 45[scale=5]},thick,color=black,dotted] (1,.45) -- (1,0.7) node[text=black, pos=.6, xshift=7pt]{};
					\node[left] at (0,.7) {$0/1$};
					\node[left] at (0,0) {$1$};
					\node[right] at (2,0.7) {$0/1$};
					\node[above] at (1.6,0.7) {D};
					\node[above] at (0.4,0.7) {C};
					\node[above] at (0.4,0) {I};
			\end{tikzpicture}}
		\end{minipage}\hfill
		
	\end{center}
	\caption[Pecking order (lookdown)]{Propagation of types across mutations and binary branchings. In the branching event, the descendant line (D) splits into the continuing line (C) and the incoming line (I). The solid line is parental. $\star$ stands for an arbitrary type, while $0/1$ means that the scheme applies with C and D both 0, or both 1.}
	\label{fig:peckingordersel}
\end{figure}

\smallskip

Mutation events are depicted by circles and crosses on the lines. A circle~(cross) indicates a mutation to type~$0$~(type~$1$), which means that the type on the line is~$0$~(is~$1$) after the mutation, see Fig.~\ref{fig:peckingordersel}. This occurs at rate~$u\nu_0$~(at rate~$u\nu_1$) on every line, again by way of independent Poisson point processes.

\smallskip

Given a realisation of the untyped particle system and an initial type configuration (that is, a type assigned to each line at~$t=0$), one determines the types on the lines for all~$t>0$ via the propagation rules explained above, thus \emph{typing} the particle system. The distribution of the initial types and the law of the graphical elements (arrows, circles, and crosses) are independent of each other.

\smallskip

The graphical representation gives rise to the ASG via the following construction. Consider a realisation of the untyped interacting particle system in the  interval~$[0,t]$ for some time~$t>0$, to which we refer as the present. We now construct a process that starts at present and runs \emph{backward} in time. Backward time is indicated by the variable~$r$, where $r=0$ ($r=t$) corresponds to forward time~$t$  (forward time~$0$). Now, pick an untyped sample  at present and trace back the lines of individuals whose type may have an influence on the type of the sampled individuals where, at this stage, we only take into account the information contained in the reproduction events, and ignore the additional information due to mutation. The ASG consists of these lines in~$[0,t]$. See Fig.~\ref{sec:mainresults:sec:mm:fig:graphrepASG} for the  ASG embedded into the interacting particle system.

\begin{figure}[t]
	\begin{center}
		\begin{minipage}{0.19 \textwidth}
			\centering
			\scalebox{1}{
				\begin{tikzpicture}			
					\draw[line width=0.5mm,dotted] (0,1.4) -- (1,1.4);
					\draw[line width=0.5mm] (1,1.4) -- (2,1.4);
					\draw[line width=0.5mm] (0,.7) -- (1,.7);
					\draw[line width=0.5mm,dotted] (0,0) -- (1,0);
					
					\draw (1,.7) circle (.6mm)  [fill=black!100];
					\fill[white, draw=black] (0.9,1.5) rectangle (1.1,1.3); 
					\draw[line width=.5mm,dotted] (1,0) .. controls (1,.2) and (1.15,.5) .. (1.15,.7);
					\draw[line width=.5mm,dotted] (1.15,.7) .. controls (1.15,.9) and (1,1.05) .. (1,1.35);
					\draw[-{Stealth[length=2mm,width=2mm,open]},opacity=1,line width=.5mm] (1,.7) -- (1,1.5);
					\filldraw[white, draw=black] (1,0.1) -- (1.1,0) -- (1,-0.1) -- (.9,0) -- (1,0.1);

					\node[left] at (0,.7) {$0$};
					\node[left] at (0,1.4) {$\star$};
					\node[left] at (0,0) {$0$};
					\node[right] at (2,1.4) {$0$};
					\node[above] at (1.6,1.4) {D};
					\node[above] at (0.4,0.7) {I};
					\node[above] at (0.4,1.4) {C};
					\node[above] at (0.4,0) {J};
			\end{tikzpicture}}
		\end{minipage}\hspace{1.2cm}\begin{minipage}{0.19 \textwidth}
			\centering
			\scalebox{1}{
				\begin{tikzpicture}			
					\draw[line width=0.5mm] (0,1.4) -- (1,1.4);
					\draw[line width=0.5mm] (1,1.4) -- (2,1.4);
					\draw[line width=0.5mm,dotted] (0,.7) -- (1,.7);
					\draw[line width=0.5mm,dotted] (0,0) -- (1,0);
					
					\draw (1,.7) circle (.6mm)  [fill=black!100];
					\fill[white, draw=black] (0.9,1.5) rectangle (1.1,1.3); 
					\draw[line width=.5mm,dotted] (1,0) .. controls (1,.2) and (1.15,.5) .. (1.15,.7);
					\draw[line width=.5mm,dotted] (1.15,.7) .. controls (1.15,.9) and (1,1.05) .. (1,1.35);
					\draw[-{Stealth[length=2mm,width=2mm,open]},opacity=1,line width=.5mm,dotted] (1,.7) -- (1,1.5);
					\filldraw[white, draw=black] (1,0.1) -- (1.1,0) -- (1,-0.1) -- (.9,0) -- (1,0.1);

					\node[left] at (0,.7) {$0$};
					\node[left] at (0,1.4) {$0/1$};
					\node[left] at (0,0) {$1$};
					\node[right] at (2,1.4) {$0/1$};
					\node[above] at (1.6,1.4) {D};
					\node[above] at (0.4,0.7) {I};
					\node[above] at (0.4,1.4) {C};
					\node[above] at (0.4,0) {J};
			\end{tikzpicture}}
		\end{minipage}\hspace{1.2cm}\begin{minipage}{0.19 \textwidth}
			\centering
			\scalebox{1}{
				\begin{tikzpicture}			
					\draw[line width=0.5mm] (0,1.4) -- (1,1.4);
					\draw[line width=0.5mm] (1,1.4) -- (2,1.4);
					\draw[line width=0.5mm,dotted] (0,.7) -- (1,.7);
					\draw[line width=0.5mm,dotted] (0,0) -- (1,0);
					
					\draw (1,.7) circle (.6mm)  [fill=black!100];
					\fill[white, draw=black] (0.9,1.5) rectangle (1.1,1.3); 
					\draw[line width=.5mm,dotted] (1,0) .. controls (1,.2) and (1.15,.5) .. (1.15,.7);
					\draw[line width=.5mm,dotted] (1.15,.7) .. controls (1.15,.9) and (1,1.05) .. (1,1.35);
					\draw[-{Stealth[length=2mm,width=2mm,open]},opacity=1,line width=.5mm,dotted] (1,.7) -- (1,1.5);
					\filldraw[white, draw=black] (1,0.1) -- (1.1,0) -- (1,-0.1) -- (.9,0) -- (1,0.1);

					\node[left] at (0,.7) {$1$};
					\node[left] at (0,1.4) {$0/1$};
					\node[left] at (0,0) {$\star$};
					\node[right] at (2,1.4) {$0/1$};
					\node[above] at (1.6,1.4) {D};
					\node[above] at (0.4,0.7) {I};
					\node[above] at (0.4,1.4) {C};
					\node[above] at (0.4,0) {J};
			\end{tikzpicture}}
		\end{minipage}
	\end{center}
	\caption[Type propagation in ternary branchings (lookdown)]{Ternary branching of the descendant line (D) into the continuing line (C), the checking line (J), and the incoming line (I) along with the associated type propagation rule. The solid line is parental. $\star$ stands for an arbitrary type, while $0/1$ means that the scheme applies with C and D both 0, or both 1.}
	\label{fig:peckingorderint}
\end{figure} 
\smallskip

The true ancestry of the initial sample is obtained after assigning types to all lines in the ASG at forward time~$0$ (i.e. $r=t$), without replacement from a population consisting of $Y^{(N)}_0$ unfit and $N-Y^{(N)}_0$ fit individuals. Then propagate the types and resolve the ancestry at branching events forward in time up to time~$t$ ($r=0$) according to the  propagation rules of the Moran model. This way, the types of the sampled individuals are recovered together with their ancestry. For a detailed construction of the ASG in the Moran model we refer to Section~\ref{sec:ASGMoran}.

\subsection{Large population limit of the Moran model and the  ASG}\label{sec:mainresults:subsec:detModelvsMoranModel}
Let us now relate the Moran model with pairwise interaction to the mutation--selection equation with pairwise interaction. To this end, consider a sequence of Moran models, indexed by their population size $N$, sharing the same parameters $s,\gamma,u,\nu_0,\nu_1$, and let $N$ tend to infinity without rescaling time or parameters. The existence of such a limit and its relation with the mutation--selection equation is given in the following proposition, which we will prove in Section~\ref{sec:detmodelandlaw}.
\begin{proposition}[Dynamical law of large numbers]\label{sec:mainresults:prop:lln}
	Suppose~$\lim_{N\rightarrow\infty}Y_0^{(N)}/N=y_0\in[0,1]$. Then for all~$\varepsilon>0$ and $t\geq 0$, we have\begin{equation*}
	\lim\limits_{N\rightarrow\infty} P\Big (\sup\limits_{\xi\leq t}\Big | \frac{Y_{\xi}^{(N)}}{N}-y(\xi;y_0)\Big | >\varepsilon\Big ) =0,
	\end{equation*}
	where~$y(\,\cdot\,;y_0)$ is the solution of the IVP~\eqref{eq:dlimitdiffeq}.
\end{proposition}

The above limit is   natural  for large populations. After all, mutation is a molecular mechanism and is reasonably assumed as independent of population size. Likewise, reproduction rates may be assumed to be independent of population size as long as available resources (space, food, \textellipsis) scale linearly with~$N$. Mutation, selection, and interaction are said to be \emph{strong} in this case; in contrast to \emph{weak} parameters that scale inversely with population size and give rise to a diffusion limit of the Moran model (see \cite[Sec.~7.2]{durrett2008probability} for a review). 

\smallskip

The connection between the (finite-$N$) Moran model and its deterministic counterpart leads to a candidate for an ancestral structure for the mutation--selection equation with pairwise interaction. We start with the ancestral picture of the stochastic model and consider the limit~$N\to\infty$. The resulting process will still be stochastic. Let us first describe it in the spirit of the graphical representation of Section~\ref{sec:mainresults:subsec:mm} and only then formalise it.

\smallskip

The asymptotic ASG has tree structure. More precisely, it has only binary and ternary branchings, and mutations. This is because arrows that connect two lines in the finite-$N$ ASG occur at a rate of order~$\mathcal{O}(1/N)$ (per ordered pair of lines); see Sec.~\ref{sec:ASGMoran}. The lines in the asymptotic ASG are therefore also conditionally independent. So in the limit, the distribution of an ASG that starts from~$n$ individuals is distributed as~$n$ independent copies of an ASG started from a single individual. Thus, we can restrict our analysis to an ASG starting with a single line.
\smallskip

The asymptotic ASG has the following transitions (see Figs.~\ref{fig:peckingordersel}  and \ref{fig:peckingorderint}). Binary branchings occur in each line independently at rate~$s$, thus increasing the number of lines by one. Independently of the other lines and independently of binary branchings, each line has ternary branchings at rate~$\gamma$, which increase the number of lines by two. Each line mutates to type~$0$ at rate~$u\nu_0$ and to type~$1$ at rate~$u\nu_1$. Mutations occur independently on each line and independently of all other events. We will from now on refer to the ASG in this large population limit just as the ASG, unless stated otherwise.

\smallskip

As in the finite case, the untyped picture turns into a typed one by assigning a type to each line of the ASG at backward time~$r=t$ (the past), this time by sampling independently according to~$(1-y_0,y_0)$. We then propagate types and resolve the ancestry according to the usual rules up to time~$r=0$ (the present), see Figs.~\ref{fig:peckingordersel} and \ref{fig:peckingorderint}. In this way, one determines the distribution of the type of the sampled individual along with its ancestry.

\subsection{The embedded ASG, type propagation, and a sampling duality} \label{sec:mainresults:eASG}
One way to formally construct the  ASG is via real trees with marks (see e.g.~\citep[p. 70]{Hummel2019}). But for our purposes it is enough to encode the tree structure embedded in the ASG together with the mutation marks. We do this using an appropriate
class of trees with marks.

\smallskip

A directed rooted tree is a finite, directed, acyclic, connected graph with a special vertex called the root, where the edges are directed away from the root. Vertices are unlabelled\footnote{That is, we consider a tree as an equivalence class with respect to relabelling of the vertices.} and edges have no length. For a directed rooted tree $\tau$, we write $V(\tau)$ and $L(\tau)$ for the set of vertices and the set of leaves, respectively. Moreover, we write $\rho(\tau)$ for the root of $\tau$. The outdegree of a vertex~$v$ is denoted by~$\deg(v)$. For a directed edge $(v, w)$ we call $w$ the child of $v$, and $v$ the
parent of $w$. Similarly, if  a vertex $u$ is on the path from $\rho(\tau)$ to a vertex $w$ in $\tau$, we write $u\preceq_\tau w$; this induces a partial order on $V(\tau)$. We write $u\prec_\tau w$ if $u\preceq_\tau w$ but $u\neq w$; we then
call $u$ an ancestor of $w$, and $w$ a descendant of $u$. (It is unfortunate that these standard notions for trees are contrary to the genealogical interpretation. However, in what follows, we stick to the tree notation, unless specified otherwise.) An ordered rooted tree is a directed rooted tree in which an ordering is specified for the children of each internal vertex. Such an ordering is equivalent to an embedding of the tree in the plane, with the root at the bottom and the children of each vertex  ordered from left to right and placed above their parent. This allows us to speak of the left, middle, and right (the left and right) child of a vertex with outdegree $3$ (outdegree $2$).
\smallskip

Let $\Xi$ be the set of ordered rooted trees having: i) vertices with outdegree at most~$3$, and ii) vertices with outdegree~$1$ marked either with $\times$ or $\circ$. Since vertices are unlabelled, $\Xi$ is countable, and thus, a Polish space if equipped with the discrete topology.
\smallskip

For $\alpha\in \Xi$ and $v\in V(\alpha)$ with $\deg(v) = 3$ (with $\deg(v) = 2$), we denote by $\chil{v}$, $\chim{v}$, and $\chir{v}$ (by $\chil{v}$ and $\chir{v}$) the left, middle and right (the left and right) child of $v$. If $\deg(v) = 1$, its single child is  $\chim{v}$.  We symbolise by~$\rootsin$ the tree that is only an unmarked root.
\smallskip

Let us now explain how to associate to an ASG in $[0, t]$ a unique element $a_t\in\Xi$ (see Fig.~\ref{eASG}). First, we identify each line segment in the ASG enclosed between two consecutive events, i.e. between branching, mutation, or the terminal time $t$, with a unique vertex in $V(a_t)$, and vice versa. If the segment ends at time $t$, the vertex is a leaf (so has outdegree 0). If the segment ends in a beneficial (resp. deleterious) mutation, the vertex has outdegree 1, and we equip it with mark $\circ$ (resp. $\times$). If the segment ends in a binary branching,  the vertex has outdegree 2, and we make it the parent of the vertices corresponding to the continuing (left child) and incoming line (right child). Finally, if the segment ends in a ternary branching, then the vertex has outdegree 3, and we make it the parent of the vertices corresponding to the continuing (left child), checking (middle child) and incoming line (right child).  We call $a_t$ the \emph{embedded ASG (eASG)}.

\begin{figure}[t!]
	\begin{minipage}{.4\textwidth}
		\begin{center} \scalebox{0.65}{\begin{tikzpicture}
				\draw[line width=.5mm ] (8,4) -- (1,4);
				\draw[line width=.5mm ] (6.5,2) -- (1,2);
				\draw[line width=.5mm ] (2.9,1) -- (1,1);
				\draw[line width=.5mm ] (6.5,3) -- (1,3);
				\draw[line width=.5mm ] (2.9,3.1) -- (2.9,3.75);
				\draw[line width=.5mm ] (2.9,1.9) -- (2.9,1);
				\draw[line width=.5mm ] (2.9,2.1) -- (2.9,2.9);
				
				\draw[dashed,opacity=0.5]{(1,4) -- (1,0.5)};
				\node [below] at (1,0.5) {$t$};
				\draw[dashed, opacity=0.5]{(2,4) -- (2,0.5)};
				\node [below] at (2,0.5) {$t_4$};
				\draw[dashed, opacity=0.5]{(2.9,4) -- (2.9,0.5)};
				\node [below] at (2.9,0.5) {$t_3$};
				\draw[dashed, opacity=0.5]{(4.8,4) -- (4.8,0.5)};
				\node [below] at (4.8,0.5) {$t_2$};
				\draw[dashed, opacity=0.5]{(6.5,4) -- (6.5,0.5)};
				\node [below] at (6.5,0.5) {$t_1$};
				\draw[dashed, opacity=0.5]{(8,4) -- (8,0.5)};
				\node [below] at (8,0.5) {$0$};
				
				\node at (4.8,4) {\scalebox{2}{$\times$}} ;
				
				\draw[-{open triangle 45[scale=5]},thick,opacity=1] (2.9,3.1) -- (2.9,4);
				
				\coordinate (S2)  at (2.9,3);
				\coordinate (I12)  at (2.9,3.1);
				\coordinate (I22)  at (2.9,2.9);
				
				\coordinate (S4)  at (2.9,1);
				\coordinate (I14)  at (2.9,1.1);
				\coordinate (I24)  at (2.9,0.9);
				
				\coordinate (S1)  at (6.2,3);
				\coordinate (I11)  at (6.2,3.1);
				\coordinate (I21)  at (6.2,2.9);
				
				\coordinate (S3)  at (6.2,2);
				\coordinate (I13)  at (6.2,2.1);
				\coordinate (I23)  at (6.2,1.9);
				
				\draw (6.5,3) circle (.6mm)  [fill=black!100];
				\fill[white, draw=black] (6.4,3.9) rectangle (6.6,4.1); 
				\draw[-{Stealth[length=2mm,width=2mm,open]},opacity=1,line width=.5mm] (6.5,3) -- (6.5,4.1);
				
				\draw[line width=.5mm] (6.5,2) .. controls (6.5,2.5) and (6.65,2.5) .. (6.65,3);
				\draw[line width=.5mm] (6.65,3) .. controls (6.65,3.5) and (6.5,3.5) .. (6.5,4);
				
				\filldraw[white, draw=black] (6.5,2.1) -- (6.6,2) -- (6.5,1.9) -- (6.4,2) -- (6.5,2.1);
				
				\draw[line width=.5mm ] (2.9,2.875) arc(-100:100:.125) ;
				\draw[line width=.5mm ] (2.9,1.875) arc(-100:100:.125) ;
				
				\draw (2,2) circle (1.5mm)  [fill=white!100];    
				\end{tikzpicture}}
		\end{center}
	\end{minipage}\hspace{0.5cm} \begin{minipage}{0.25\textwidth}
		\begin{center}
			\scalebox{.65}{\begin{tikzpicture}


				\draw[dashed,opacity=0.5]{(1,4) -- (1,0.5)};
				\node [below] at (1,0.5) {$t$};
				\draw[dashed, opacity=0.5]{(2,4) -- (2,0.5)};
				\node [below] at (2,0.5) {$t_4$};
				\draw[dashed, opacity=0.5]{(2.9,4) -- (2.9,0.5)};
				\node [below] at (2.9,0.5) {$t_3$};
				\draw[dashed, opacity=0.5]{(4.8,4) -- (4.8,0.5)};
				\node [below] at (4.8,0.5) {$t_2$};
				\draw[dashed, opacity=0.5]{(6.5,4) -- (6.5,0.5)};
				\node [below] at (6.5,0.5) {$t_1$};
				\draw[dashed, opacity=0.5]{(8,4) -- (8,0.5)};
				\node [below] at (8,0.5) {$0$};
				
				\node at (4.8,4) {\scalebox{2}{$\times$}} ;
				

				\coordinate (S2)  at (2.9,3);
				\coordinate (I12)  at (2.9,3.1);
				\coordinate (I22)  at (2.9,2.9);
				
				\coordinate (S4)  at (2.9,1);
				\coordinate (I14)  at (2.9,1.1);
				\coordinate (I24)  at (2.9,0.9);
				
				\coordinate (S1)  at (6.2,3);
				\coordinate (I11)  at (6.2,3.1);
				\coordinate (I21)  at (6.2,2.9);
				
				\coordinate (S3)  at (6.2,2);
				\coordinate (I13)  at (6.2,2.1);
				\coordinate (I23)  at (6.2,1.9);
				
                 \draw[line width=.5mm] (6.32,4) -- (6.68,4);
                 \draw[line width=.5mm] (6.33,3) -- (6.66,4);
				\draw[line width=.5mm] (6.33,2) -- (6.66,4);
				
				\draw[line width=.5mm] (4.6,4) -- (4.8,4);
				
				\draw[line width=.5mm] (2.77,4) -- (3,4);
				\draw[line width=.5mm] (2.799,1) -- (3,4);
				
				\draw[line width=.5mm] (1.8,2) -- (2.1,2);
				
				\draw (2,2) circle (1.5mm)  [fill=white!100];   

				\begin{pgfonlayer}{background}
                   \highlight{5mm}{cyan}{(7.75,4) -- (6.9,4)}
                   \highlight{5mm}{cyan}{(6.1,4) -- (5.05,4)}
                   \highlight{5mm}{cyan}{(4.35,4) -- (3.25,4)}
                   \highlight{5mm}{cyan}{(2.55,4) -- (1.2,4)}
                   
                    \highlight{5mm}{cyan}{(6.1,2) -- (2.25,2)}
                    \highlight{5mm}{cyan}{(1.2,2) -- (1.55,2)}
                   
                   \highlight{5mm}{cyan}{(6.1,3) -- (1.2,3)}

                   \highlight{5mm}{cyan}{(2.55,1) -- (1.2,1)}
                   
                 \end{pgfonlayer}
				
				\end{tikzpicture}}
			
	\end{center}	\end{minipage}\hspace{0.7cm}\begin{minipage}{0.25\textwidth}
		\begin{center}
			\scalebox{.40}{\begin{tikzpicture}
				\draw[opacity=0] (.5,.5)--(2.5,7.5);
				\begin{scope}[every node/.style={regular polygon,regular polygon sides=7,draw}, minimum size=30pt]
				\node (h00) at (1.5,6.2) {};
				\node (h01) at (3.5,6.2) {};
				\node (h02) at (2.5,4.5) {};
				\node (h03) at (5.5,4.5) {};
				\node (h04) at (2.5,3) {};
				\node (h05) at (4,3) {};
				\node (h06) at (5.5,3) {};
				\node (h07) at (4,1.4) {};
				\end{scope}

				\begin{scope}[>={latex[black]}, every node/.style={fill=white,circle},
				every edge/.style={draw=black,very thick}]
				
				\path [-] (h02) edge (h00);
				\path [-] (h02) edge (h01);
				
				\path [-] (h04) edge (h02);
				
				\path [-] (h07) edge (h04);
				\path [-] (h07) edge (h05);
				\path [-] (h07) edge (h06);

				\path [-] (h06) edge (h03);

				\end{scope}

				\draw (5.5,3) circle (3 mm)  [fill=white!100];

				\node at (2.5,3) {\scalebox{4}{$\times$}} ;

				\end{tikzpicture}}
			
	\end{center}	\end{minipage}
	
	\caption[ASG vs embedded ASG]{In the classical visualisation of the ASG (left), horizontal line segments delimited by events translate into vertices in the embedded ASG (right). Vertices corresponding to segments ending in a beneficial (resp. deleterious) mutations are marked with a $\circ$ (resp. $\times$). The  middle picture shows an intermediate stage where the segments appear as sausages, which then shrink into vertices with edges between them that have no (meaningful) lenghts.}
	\label{eASG}
\end{figure}
\smallskip

The type propagation along the lines in the ASG, as illustrated in Figs.~\ref{fig:peckingordersel} and~\ref{fig:peckingorderint}, translates into the following notion of type propagation along the vertices of a tree in $\Xi$.

\begin{definition}[Type propagation]\label{typro}
A \emph{leaf-type configuration} of $\al\in\Xi$ is a vector ${c}\defeq ({c}_\ell)_{\ell\in L(\al)}\in\{0,1\}^{L(\al)}$. The \emph{vertex-type propagation} of ${c}$ in $\al$ is the vector $\cpr\defeq (\cpv{v})_{v\in V(\al)}\in\{0,1\}^{V(\al)}$ constructed from~${c}$ recursively\footnote{The notation $\cpv{v}$ hints at the construction from the leaves to the root of the tree.} as follows. 
\begin{enumerate}
    \item If ~$\deg(v)=0$ (that is, $v\in {L}(\al)$), set $\cpv{\ell}= {c}_\ell$.
    \item If ~$\deg(v)=1$ and $v$ is marked with $\times$ (resp. $\circ$), set $\cpv{v}= 1$ (resp. $= 0$).  
	\item If~$\deg(v)=2$, then $\cpv{v}= 1$ if and only if $\cpv{\chil{v}}=\cpv{\chir{v}}=1$.	
	\item If~$\deg(v)=3$, then $\cpv{v}= 1$ if and only if $\cpv{\chil{v}}=1$ and $\cpv{\chil{v}}+\cpv{\chir{v}}>0$.
\end{enumerate}
If $\cpv{v}= 1$ (resp. $=0$), we say that $v\in V(\al)$ has the unfit (resp. fit) type under $c$.
\end{definition}

Note that the mapping $c\mapsto\cpv{\varrho(\alpha)}$ is a Boolean function, which is conveniently encoded via the underlying tree. 

\smallskip

Let now~$\al\in\Xi$, $z\in[0,1]$, and for each~$\ell\in L(\al)$, let $C_\ell(z)$ be a Bernoulli random variable with parameter~$z$; so $C(z)\defeq (C_\ell(z))_{\ell\in {L}(\al)}$ is a random leaf-type configuration. Define~$H(\al, z)$ to be the probability that the root of $\al$ gets the unfit type under $C(z)$. 

\smallskip 

Our next task is to translate the transitions of the ASG into transitions of the embedded ASG. To this end, we introduce the following transformations on~$\Xi$, see  Fig.~\ref{fig:asgprocess}. For $\asg\in {\Xi}$ and $\ell\in L(\asg)$, define \begin{enumerate}
	\item[(i)] $\asg^{\nwedge}_{\ell}\in {\Xi}$, the tree that arises if we add two children $\chil{\ell}$ (left) and $\chir{\ell}$ (right) with no mark to $\ell$ in $\al$. In particular, $V(\asg_{\ell}^{\nwedge})=V(\asg) \cup\{\chil{\ell},\chir{\ell}\}$ and $L(\asg_{\ell}^{\nwedge})=(L(\asg)\setminus\{\ell\})\cup\{\chil{\ell},\chir{\ell}\}$.
	\item[(ii)] $\asg_{\ell}^{\pitch}\in {\Xi}$, the tree that arises if we add three children $\chil{\ell}$ (left), $\chimell$ (middle), and $\chir{\ell}$ (right) with no mark to~$\ell$ in~$\asg$. In particular, $V(\asg_{\ell}^{\pitch})=V(\asg) \cup\{\chil{\ell},\chim{\!\ell},\chir{\ell}\}$ and $L(\asg_{\ell}^{\pitch})=(L(\asg)\setminus\{\ell\})\cup\{\chil{\ell},\chimell,\chir{\ell}\}$.
	\item[(iii)] $\asg_{\ell}^{\times}\in {\Xi}$, the tree that arises if we add a child~$\chim{\ell}$ to $\ell$ in~$\asg$ and mark $\ell$ with $\times$. In particular, $V(\asg_{\ell}^{\times})=V(\asg) \cup\{\chim{\ell}\}$ and $L(\asg_{\ell}^{\times})=(L(\asg)\setminus\{\ell\})\cup\{\chim{\ell}\}$. 
	\item[(iv)] $\asg_{\ell}^{\circ}\in {\Xi}$, the tree that arises if we add a child~$\chim{\ell}$ to $\ell$ in~$\asg$ and mark $\ell$ with $\circ$. In particular, $V(\asg_{\ell}^{\circ})=V(\asg) \cup\{\chim{\ell}\}$ and $L(\asg_{\ell}^{\circ})=(L(\asg)\setminus\{\ell\})\cup\{\chim{\ell}\}$. 
\end{enumerate}
\begin{figure}[b]
	
	\begin{minipage}{.15\textwidth}
		\begin{center}
			\scalebox{.6}{\begin{tikzpicture}
				\draw[opacity=0] (0,-.2)--(0,2.7);
				\begin{scope}[every node/.style={regular polygon,regular polygon sides=7,draw}, minimum size=13pt]
				\node[label=90:{\scalebox{1.5}{$\rho$}}] (h21) at (0,0) {};
				\end{scope}
				\end{tikzpicture}}~~\\
			$\asg_{}^{}$~~
		\end{center}
	\end{minipage}\begin{minipage}{.23\textwidth}
	\begin{center}
		\scalebox{.6}{\begin{tikzpicture}
			\draw[opacity=0] (0,-.2)--(0,2.7);
			\begin{scope}[every node/.style={regular polygon,regular polygon sides=7,draw}, minimum size=13pt]
				\node (h11) at (0,0) {};
			\node[label=90:{\scalebox{1.5}{$\chir{\rho}$}}] (h21) at (0.5,1) {};			\node[label=90:{\scalebox{1.5}{$\chil{\rho}$}}] (h31) at (-0.5,1) {};
			\end{scope}
			\begin{scope}[>={latex[black]}, every node/.style={fill=white,circle},
			every edge/.style={draw=black,very thick}]
			\path [-] (h11) edge (h21);
			\path [-] (h11) edge (h31);
			\end{scope}
			\end{tikzpicture}}~~\\
		$(i)~\asg_{\rho}^{\nwedge}$
	\end{center}
\end{minipage}\begin{minipage}{.23\textwidth}
	\begin{center}
		\scalebox{.6}{\begin{tikzpicture}
			\draw[opacity=0] (0,-.2)--(0,2.7);
			\begin{scope}[every node/.style={regular polygon,regular polygon sides=7,draw}, minimum size=13pt]
			\node (h11) at (0,0) {};
			\node[label=90:{\scalebox{1.5}{$\chim{\rho}$}}] (h21) at (0,1) {};
			\node[label=90:{\scalebox{1.5}{$\chil{\rho}$}}] (h31) at (-1,1) {};
			\node[label=90:{\scalebox{1.5}{$\chir{\rho}$}}] (h41) at (1,1) {};
			\end{scope}
			\begin{scope}[>={latex[black]}, every node/.style={fill=white,circle},
			every edge/.style={draw=black,very thick}]
			\path [-] (h11) edge (h21);
			\path [-] (h11) edge (h31);
			\path [-] (h11) edge (h41);
			\end{scope}
						
			\end{tikzpicture}}~~\\
		$(ii)~\asg_{\rho}^{\pitch}$
	\end{center}
\end{minipage}\begin{minipage}{.18\textwidth}
\begin{center}
	\scalebox{.6}{\begin{tikzpicture}
		\draw[opacity=0] (0,-.2)--(0,2.7);
		\begin{scope}[every node/.style={regular polygon,regular polygon sides=7,draw}, minimum size=13pt]
		\node (h11) at (0,0) {};
		\node[label=90:{\scalebox{1.5}{$\chim{\rho}$}}] (h21) at (0,1) {};
		\end{scope}
		\node at (0,0) {\scalebox{1.6}{$\times$}} ;
		\begin{scope}[>={latex[black]}, every node/.style={fill=white,circle},
		every edge/.style={draw=black,very thick}]
		\path [-] (h11) edge (h21);
		\end{scope}
		\end{tikzpicture}}~~\\
	\vspace{0.15cm} $(iii)~\asg_{\rho}^{\times}$
\end{center}
\end{minipage}\begin{minipage}{.18\textwidth}
\begin{center}
	\scalebox{.6}{\begin{tikzpicture}
		\draw[opacity=0] (0,-.2)--(0,2.7);
		\begin{scope}[every node/.style={regular polygon,regular polygon sides=7,draw}, minimum size=13pt]
		\node (h11) at (0,0) {};
		\node[label=90:{\scalebox{1.5}{$\chim{\rho}$}}] (h21) at (0,1) {};
		\end{scope}
		\begin{scope}[>={latex[black]}, every node/.style={fill=white,circle},
		every edge/.style={draw=black,very thick}]
		\path [-] (h11) edge (h21);
		\end{scope}
		\draw (0,0) circle (1.3mm)  [fill=white!100]; 
		\end{tikzpicture}}~~\\
	$(iv)~\asg_{\rho}^{\circ}$
\end{center}
\end{minipage}
\caption[Transition eASG]{A tree~$\asg\in \Xi$ that is a single root $\rho$ and its transformations (i)-(iv) used in the eASG process.}
\label{fig:asgprocess}
\end{figure}

Note that ${\Xi}$ is invariant under these transformations. We can now define the process that captures the embedded tree structure of an evolving ASG.
\begin{definition}\label{def:eASGprocess}
The \emph{embedded ASG (eASG) process} is the continuous-time Markov chain $(\brea_t)_{t\geq 0}$ on~${\Xi}$ with the following transition rates. For $\asg\in {\Xi}$ and $\ell\in L(\asg)$, \begin{align*}
q_a(\asg,\asg_\ell^\nwedge)&=s,\qquad  q_a(\asg,\asg_\ell^\pitch)=\gamma,\qquad q_a(\asg,\asg_\ell^\circ)=u\nu_0,\qquad q_a(\asg,\asg_\ell^\times)=u\nu_1.
\end{align*}
\end{definition}
The connection to the ASG is the reason we will occasionally refer to the marks~$\times$ and~$\circ$ also as deleterious and beneficial mutations. The next result relates the eASG process to the mutation--selection equation.
 \begin{theorem}[Duality eASG] \label{sec:mainresults:thm:dualityASG}
Let $(\brea_t)_{t\geq 0}$ be the embedded ASG process and let~$y(\cdot;y_0)$ be the solution of the ODE~\eqref{eq:dlimitdiffeq} with initial value $y_0\in [0,1]$. Then, for $\asg \in {\Xi}$ and $t\geq 0$, $$H\big (\asg,y(t;y_0) \big )=\E_\asg[H(\brea_t,y_0)],$$ where the subscript indicates the initial value. In particular,
$y(t;y_0)=\E_{\rootsin}[H(\brea_t,y_0)].$
\end{theorem}
The proof of Theorem~\ref{sec:mainresults:thm:dualityASG} is provided in Section~\ref{sec:proofseASG}.

\begin{remark}
	The duality is a special case of a more general result within the framework of recursive tree processes. More precisely, Theorem~\ref{sec:mainresults:thm:dualityASG} follows from~\citet[Thm. 1.6]{MSS2018} together with~\citep[Rem. 1.7]{MSS2018}. The authors apply their result in the setup of Remark~\ref{rem:cooperativebranching}; then our deleterious mutations and ternary branchings translate to their local maps $\mathtt{dth}$ (`deaths') and $\mathtt{cob}$ (`cooperative branchings'), respectively. The propagation rule in Definition~\ref{typro}~(2) and~(4) coincides with \citep[Eq.~$(1.13)$]{MSS2018} (with their~$1$'s being our $0$'s).
\end{remark}

\smallskip

A natural way of computing~$H(\brea_t,y_0)$ is to determine first those leaf-type configurations of $\brea_t$ that lead to an unfit root and then to evaluate the probability of observing these leaf-type configurations if each leaf type is independently sampled $0$ and $1$ according to $(1-y_0,y_0)$. This is the approach pursued by~\citet{Mach2017}, but the general idea is also present in the work of~\citet[Ch.~5.5]{dawson2014spatial}. 

\smallskip

In contrast, we aim at resolving all information contained in the evolving tree on the spot. This requires to transform (i.e. prune and graft) the eASG upon mutations via suitable operations, which we introduce in the next subsection; this leads to a second duality for~\eqref{eq:dlimitdiffeq}. It turns out that the resulting trees can be condensed even further, leading to a simpler process and to a third duality for~\eqref{eq:dlimitdiffeq}. We use this more tractable process  in Section~\ref{sec:mainresults:subsec:application} to establish the connection between the long-term behaviour of the eASG and the bifurcation structure of~\eqref{eq:dlimitdiffeq}. Before we lay out the details of the transformations, we briefly explain the simplifications in the non-interactive case ($\gamma=0$), where the concepts become particularly transparent. This is the model treated in~\citep{BCH17}.
\smallskip

The root of an eASG \emph{without ternary branchings and mutations} gets the unfit type if and only if all leaves are assigned the unfit type; this follows from the type propagation and holds regardless of the tree structure. Now allow for mutations. The type of a marked vertex is determined by the type of the mutation mark. In particular, the types of the descendants of such a vertex do not propagate beyond that vertex and are thus irrelevant for  the type of the root, so that we can remove (or \emph{prune}) them. If all such descendants are removed from the tree, only leaves can have marks. If one of them is marked with~$\circ$, the root gets the fit type, irrespective of the types of the other leaves; we thus stop reading the eASG and kill it, that is, send it to a cemetery state~$\Delta$. If there is no leaf marked with $\circ$, the root gets the unfit type if and only if all unmarked leaves are assigned the unfit type --- irrespective of the tree structure. The pruning can be implemented dynamically as the eASG process evolves, and the information required for the root type can be condensed by only counting the number of unmarked leaves.

\smallskip

Let $\Rs_r$ be the number of unmarked leaves in the eASG process at time~$r$, where $\Rs_r\defeq\Delta$ if the eASG has been sent to~$\Delta$. The (generalised) leaf-counting process $\Rs\defeq (\Rs_r)_{r\geq 0}$ is a continuous-time Markov chain on~$\N_0^\Delta\defeq\N_0\cup\{\Delta\}$ with transition rates $$q_{\Rs}^{}(k,k+1)=ks,\quad q_{\Rs}^{}(k,k-1)=ku\nu_1,\quad q_{\Rs}^{}(k,\Delta)=ku\nu_0, \qquad k\in \N_0,$$
which reflect that selection leads to an additional leaf, a leaf is pruned immediately when it experiences a deleterious mutation, and the entire process is killed when a beneficial mutation arrives. The process has absorbing states $0$ and $\Delta$, where absorption in~$0$ (in $\Delta$) implies that the root of the eASG gets the unfit (fit) type. The process~$\Rs$ is in moment duality with the mutation--selection model without interaction (see \citep[Thm.2]{BCH17}), that is we have for $y_0\in [0,1]$ and~$n\in \N_{0}^{\Delta}$,\begin{equation}\label{eq:killedASGduality}
y(t;y_0)^n=\E_{n}\big[y_0^{\Rs_t}\big].
\end{equation}

\subsection{Profiting from mutations: the pruned ASG}\label{sec:mainresults:subsec:pASG}
The reasoning underlying~\eqref{eq:killedASGduality} does not directly translate if the tree has ternary branchings. We can still safely remove all descendants of marked vertices without altering the type at the root; but a leaf with mark~$\circ$ in the remaining tree does not necessarily imply that the root is fit (recall Fig.~\ref{fig:peckingorderint}). To circumvent this problem, we now introduce appropriate pruning operations. In the Boolean function $c\mapsto\cpv{\varrho(\alpha)}$, these operations correspond to the removal of variables, which do not alter the value of the function. First, we explain the state space of the pruned trees.
\smallskip

Since we will remove the descendants of marked vertices, our pruned trees will only have marks on the leaves. In addition, we will get rid of any mark $\circ$ arising in the eASG, unless it propagates to the root, which then results in the tree  $\rootben$ consisting only of the root marked with $\circ$ (playing the role of $\Delta$ in~$\Rs$). We will also get rid of any mark $\times$, unless it is on a leaf that is the left child of a vertex with outdegree $3$ or it propagates to the root. In the latter case, it becomes the tree $\rootdel$ that consists only of the root marked with $\times$ (playing the role of $0$ in~$\Rs$). The resulting set of \emph{pruned trees} is denoted by~$\Xi^\pA$; it consists of $\rootdel$, $\rootben$, and all ordered rooted trees with vertices of outdegree $0$, $2$ or $3$, and leaves that can have  mark $\times$ only if they are the left child of a vertex with outdegree $3$. For $\bal\in\Xi^\pA$, denote by $\hat{L}(\bal)$ the set of its unmarked leaves. The notion of type propagation given in Definition \ref{typro} can be extended to $\Xi^\pA$ as follows.
\begin{definition}(Type propagation in $\Xi^\pA$)\label{typrob}
Let $\bal\in\Xi^\pA$. A reduced leaf-type configuration of $\bal$ is a vector $\hat{c}:=(\hat{c}_\ell)_{\ell\in {\hat{L}}(\bal)}\in\{0,1\}^{\hat{L}(\bal)}$. The leaf-type configuration induced by $\hat{c}$ is the vector $c:=(c_\ell)_{\ell\in {L}(\bal)}$ defined via $c_\ell\defeq \hat{c}_{\ell}$ for $\ell\in\hat{L}(\bal)$ and $c_\ell=1$ ($c_\ell=0$) for $\ell\in L(\bal)\setminus\hat{L}(\bal)$ with mark $\times$ (with mark $\circ$). The vertex-type propagation of $\hat{c}$ is the vector $\chr:=(\chv{v})_{v\in V(\bal)}$ obtained as the vertex-type configuration of $c$ in $\bal$ (after removing the marks from $\bal$) in the sense of Definition \ref{typro}.
\end{definition}
\smallskip

Let~$\bal\in\Xi^\pA$, $z\in[0,1]$, and for each~$\ell\in \hat{L}(\bal)$, let $\hat{C}_\ell(z)$ be a Bernoulli random variable with parameter~$z$; so $\hat{C}(z)\defeq (\hat{C}_\ell(z))_{\ell\in \hat{L}(\bal)}$ is a random reduced leaf-type configuration. Define~$H(\bal, z)$ to be the probability that the root of $\bal$ gets the unfit type under the type propagation, given $\hat{C}(z)$.

\smallskip

\begin{definition}[Admissible pruning]\label{adp} We say that $\bal\in\Xi^{\pA}$ is an \emph{admissible pruning} of $\alpha\in\Xi$ if:
	\begin{enumerate}
		\item[(1)] $V(\bal)\subset V(\alpha)$ and $\hat{L}(\bal)\subset L(\alpha)$,
		\item[(2)] for any $u,v\in V(\bal)$, $u \prec_\bal v$ implies $u \prec_\alpha v$,
		\item[(3)] for any leaf-type configuration ${c}=({c}_\ell)_{\ell\in L(\alpha)}$ of $\alpha$, the type assigned to $\rho(\alpha)$ under $c$ coincides with the type assigned to $\rho(\bal)$ under the reduced leaf-type configuration $(c_\ell)_{\ell\in \hat{L}(\bal)}$.
	\end{enumerate}
\end{definition}

In what follows, we construct a process $(\ba_t)_{t\geq 0}$ on~$\Xi^\pA$ that can be coupled to the eASG process $(\brea_t)_{t\geq0}$ such that for any~$t\geq 0$, $\ba_t$ is an admissible pruning of $\brea_t$. 

\smallskip 

The operators $\asg\mapsto \asg^\curlyvee_\ell$ and $\asg\mapsto \asg^\pitch_\ell$ that we have defined to act on ${\Xi}$ translate to operators on $\Xi^\pA$ in the obvious way. For mutations, we define modified mutation operators that prune away all vertices that become irrelevant for the type of the root. We start with the beneficial mutations.
To this end, note that a beneficial mutation in some vertex~$v$ determines the types of all  ancestors of~$v$ up to the first ancestor that is a middle or right child in a ternary branching event.
\begin{figure}[t]\begin{minipage}{0.3\textwidth}
		\begin{center}
			\scalebox{.6}{\begin{tikzpicture}
					\draw[opacity=0] (.5,0)--(2.5,4.5);
					\begin{scope}[every node/.style={regular polygon,regular polygon sides=7,draw}, minimum size=7pt]
						\node[fill=gray!100] (h0) at (4,0) {};
						
						\node[fill=gray!100] (h1) at (3,1) {};
						\node[fill=gray!100] (h2) at (5,1) {};
						
						\node[fill=gray!100] (h3) at (2,2) {};
						\node[fill=gray!100] (h4) at (2.75,2) {};
						\node (h5) at (3.5,2) {};
						\node (h6) at (4.5,2) {};
						\node[fill=gray!100] (h7) at (5.5,2) {};

						\node (h31) at (2.45,3) {};
						
						\node (h33) at (3.05,3) {};
						
						\node[fill=gray!100] (h9) at (4.75,3) {};
						\node (h10) at (5.5,3) {};
						\node   (h11) at (6.25,3) {};
						
					\end{scope}

					\begin{scope}[>={latex[black]}, every node/.style={fill=white,circle},
						every edge/.style={draw=black,very thick}]
						
						\path [-] (h0) edge (h1);
						\path [-] (h0) edge (h2);
						
						\path [-] (h1) edge (h3);
						\path [-] (h1) edge (h4);
						\path [-] (h1) edge (h5);
						
						\path [-] (h4) edge (h31);
						\path [-] (h4) edge (h33);

						\path [-] (h2) edge (h6);
						\path [-] (h2) edge (h7);

						\path [-] (h7) edge (h9);
						\path [-] (h7) edge (h10);
						\path [-] (h7) edge (h11);

					\end{scope}
					\begin{pgfonlayer}{background}
						\highlight{4.5mm}{cyan}{(4.75,3) -- (5.5,2) -- (5,1) -- (4.5,2) -- (5,1) -- (4,0) -- (3,1) -- (2,2)}
						\highlight{4.5mm}{cyan}{(5.5,2.95) -- (5.5,3.05)}
						\highlight{4.5mm}{cyan}{(6.25,2.95) -- (6.25,3.05)}
						\highlight{4.5mm}{cyan}{(3.5,1.95) -- (3.5,2.05)}
						\highlight{4.5mm}{cyan}{(2.75,1.95) -- (2.45,3) -- (2.75,1.95) -- (3.05,3)--(2.75,2.05)}
					\end{pgfonlayer}		
			\end{tikzpicture}}
	\end{center}	\end{minipage}\begin{minipage}{0.3\textwidth}
		\begin{center}
			\scalebox{.6}{\begin{tikzpicture}
					\draw[opacity=0] (.5,0)--(2.5,4.5);
					\begin{scope}[every node/.style={regular polygon,regular polygon sides=7,draw}, minimum size=7pt]
						\node[fill=gray!100]  (h0) at (4,0) {};
						
						\node[fill=gray!100] (h1) at (3,1) {};
						\node[fill=gray!100] (h2) at (5,1) {};
						
						\node[fill=gray!100] (h3) at (2,2) {};
						\node  (h4) at (2.75,2) {};
						\node  (h5) at (3.5,2) {};
						\node (h6) at (4.5,2) {};
						\node[fill=gray!100] (h7) at (5.5,2) {};

						\node[fill=gray!100] (h31) at (1.25,3) {};
						\node  (h32) at (2,3) {};
						\node (h33) at (2.75,3) {};
						
						\node[fill=gray!100] (h9) at (4.75,3) {};
						\node (h10) at (5.5,3) {};
						\node  (h11) at (6.25,3) {};
						
					\end{scope}

					\begin{scope}[>={latex[black]}, every node/.style={fill=white,circle},
						every edge/.style={draw=black,very thick}]
						
						\path [-] (h0) edge (h1);
						\path [-] (h0) edge (h2);
						
						\path [-] (h1) edge (h3);
						\path [-] (h1) edge (h4);
						\path [-] (h1) edge (h5);
						
						\path [-] (h3) edge (h31);
						\path [-] (h3) edge (h32);
						\path [-] (h3) edge (h33);

						\path [-] (h2) edge (h6);
						\path [-] (h2) edge (h7);

						\path [-] (h7) edge (h9);
						\path [-] (h7) edge (h10);
						\path [-] (h7) edge (h11);

					\end{scope}
					\node at (1.25,3) {\scalebox{1.5}{$\times$}} ;						\begin{pgfonlayer}{background}
						\highlight{4.5mm}{cyan}{(4.75,3) -- (5.5,2) -- (5,1) -- (4.5,2) -- (5,1) -- (4,0) -- (3,1) -- (2,2)--(1.25,3)}
						\highlight{4.5mm}{cyan}{(5.5,2.95) -- (5.5,3.05)}
						\highlight{4.5mm}{cyan}{(6.25,2.95) -- (6.25,3.05)}
						\highlight{4.5mm}{cyan}{(3.5,1.95) -- (3.5,2.05)}
						\highlight{4.5mm}{cyan}{(2.75,1.95)--(2.75,2.05)}
						\highlight{4.5mm}{cyan}{(2,2.95) -- (2,3.05)}
						\highlight{4.5mm}{cyan}{(2.75,2.95) -- (2.75,3.05)}
					\end{pgfonlayer}		
					
			\end{tikzpicture}}
	\end{center}	\end{minipage} \begin{minipage}{0.3\textwidth}
		\begin{center}
			\scalebox{.6}{\begin{tikzpicture}
					\draw[opacity=0] (.5,0)--(2.5,4.5);
					\begin{scope}[every node/.style={regular polygon,regular polygon sides=7,draw}, minimum size=7pt]
						
						\node[fill=gray!100]  (h0) at (4,0) {};
						
						\node[fill=gray!100](h1) at (3,1) {};
						\node[fill=gray!100] (h2) at (5,1) {};

						\node (h3) at (2.5,2) {};
						\node (h5) at (3.5,2) {};
						\node (h6) at (4.5,2) {};
						\node (h7) at (5.5,2) {};

						\node[fill=gray!100] (h31) at (1.75,3) {};
						\node (h32) at (2.5,3) {};
						\node(h33) at (3.25,3) {};
					\end{scope}
					
					\begin{scope}[>={latex[black]}, every node/.style={fill=white,circle},
						every edge/.style={draw=black,very thick}]
						
						\path [-] (h0) edge (h1);
						\path [-] (h0) edge (h2);
						
						\path [-] (h1) edge (h3);
						
						\path [-] (h1) edge (h5);
						
						\path [-] (h3) edge (h31);
						\path [-] (h3) edge (h32);
						\path [-] (h3) edge (h33);
						
						\path [-] (h2) edge (h6);
						\path [-] (h2) edge (h7);

					\end{scope}
					\node at (1.75,3) {\scalebox{1.5}{$\times$}} ;			\begin{pgfonlayer}{background}
						\highlight{4.5mm}{cyan}{(5.5,2) -- (5,1) -- (4.5,2) -- (5,1) -- (4,0) -- (3,1) -- (2.5,2)--(1.75,3)}
						
						\highlight{4.5mm}{cyan}{(3.5,1.95) -- (3.5,2.05)}
						
						\highlight{4.5mm}{cyan}{(2.5,2.95) -- (2.5,3.05)}
						\highlight{4.5mm}{cyan}{(3.25,2.95) -- (3.25,3.05)}
					\end{pgfonlayer}			
			\end{tikzpicture}}
	\end{center}	\end{minipage}
	
	\caption[Firewalls]{Three pruned trees. Grey vertices are firewalls, and subtrees delimited by sausage-like sets are regions.} 	\label{fire}
\end{figure}

\begin{definition}[Region]
Let $\bal\in\Xi^\pA$. For every vertex~$w\in V(\bal)$ with $\deg(w)=3$, remove the edges $(w,\chim{w})$ and $(w,\chir{w})$. We refer to the connected components in the resulting graph as the \emph{regions} of $\bal$. For~$v\in V(\bal)$, denote by $R_v(\bal)$ the region of $\bal$ that contains~$v$. 
\end{definition}
Clearly, a region is a (maximal) subtree within which the type propagation works as in the non-interactive case. In particular, if~$v$ gets the fit type, then $\rho(R_{v}(\bal))$ gets the fit type as well; however, the type of the parent of $\rho(R_{v}(\bal))$ remains undetermined, so $\rho(R_{v}(\bal))$ acts as a `barrier' to type propagation.
\smallskip

The set of vertices affected by a deleterious mutation is more complicated, and a case by case analysis is required. However, here too, some vertices  act as barriers for the effect of deleterious mutations. 

\begin{definition}[Firewall]\label{def:firewall}
For $\bal\in\Xi^\pA$, $v\in V(\bal)$ is said to be a \emph{firewall} if either (i) $v=\rho(\bal)$, (ii) $\deg(v)=2$, (iii) the parent $w$ of $v$ has $\deg(w)=3$ and $v=\chil{w}$, or (iv) $\deg(v)=3$ and $\chil{v}$ has no mark. For $\ell\in \hat{L}(\bal)$, let $w_{\ell}(\bal)$ be the most recent ancestor of~$\ell$ that is a firewall (note that $w_{\ell}(\bal) \prec_\bal \ell$).  
\end{definition}

Figure~\ref{fire} pictures three trees with their regions and firewalls. Building on these notions, we now introduce the operators that will help us to resolve the mutations arising in the eASG, see also Fig.~\ref{fig:trcirc}.

\begin{definition}[pruning operations]\label{def:pruningoperations}
 Let $\bal\in\Xi^\pA$ and $\ell\in \hat{L}(\bal)$. Define $\pi^\times_\ell(\bal)\in \Xi^\pA$ as follows.
 \begin{enumerate}
  \item[(1a)]  If $\ell$ is a firewall, then $\pi^\times_\ell(\bal)$ is obtained by marking  $\ell$ with~$\times$.
  \item[(1b)]  If $\ell$ is not a firewall and
    \begin{itemize}
  \item if $\deg(w_\ell(\bal))=2$, then $\pi^\times_\ell(\bal)$ is obtained by replacing the subtree rooted in~$w_\ell(\bal)$ by the subtree rooted in the child of $w_\ell(\bal)$ that is not in the path from $w_\ell(\bal)$ to $\ell$.\footnote{Alternatively, this operation may be understood as removing the subtree rooted in the child of $w_\ell(\bal)$ that is the ancestor of $\ell$ and contracting the edge between $w_\ell(\bal)$ and its other child.}
  \item if $\deg(w_\ell(\bal))=3$, then $\pi_\ell^\times(\bal)$ is obtained by replacing the subtree rooted  in~$w_\bal(\ell)$ by the subtree rooted in the left child of ~$w_\ell(\bal)$.
  \end{itemize}
  
\end{enumerate}
  Define $\pi_\ell^\circ(\bal)\in \Xi^\pA$ as follows.
 \begin{enumerate}
  \item[(2a)] If $\rho(\bal)\in R_\ell(\bal)$, set $\pi_\ell^\circ(\bal)\defeq \rootben$,
  \item[(2b)] If $\rho(\bal)\notin R_\ell(\bal)$, let $\bal^*$ be the tree that results when removing from $\bal$ the subtree  rooted in $\rho(R_\ell(\bal))$. Let $v$ be the parent of $\rho(R_\ell(\bal))$ in $\bal$.
  \begin{itemize}
  \item If $\chil{v}$ has no mark, set $\pi_\ell^\circ(\bal)\defeq\bal^*$.
  \item If $\chil{v}$ has a mark (which then necessarily is $\times$), set 
  $\pi_\ell^\circ(\bal)\defeq \pi_{\chil{v}}^\times(\bal^*)$. 
  \end{itemize}
 \end{enumerate}
\end{definition}
\begin{figure}[t]
	\begin{center}
		\begin{minipage}[b]{.33\linewidth}
			\centering\scalebox{.7}{\begin{tikzpicture}
				\begin{scope}[every node/.style={regular polygon,regular polygon sides=7, very thick,draw}, minimum size=10pt]
				\node (h1) at (0,1.5) {};
				\node (l2) at (0,2.5) {};
				\node (h2) at (-2.5,2.5) {};
				\node[fill=gray!100,opacity=0.3] (l3)at (2.5,2.5) {};
				
				\node (l4)at (-3.1,3.5) {};
				\node (l6)at (-1.9,3.5) {};
				\node (l7)at (-0.8,3.5) {};
				\node (l8)at (0,3.5) {};
				\node (l9)at (0.8,3.5) {};
				\node(l10)at (1.7,3.5) {};
				\node[opacity=0.2] (l11)at (2.5,3.5) {};
				\node[opacity=0.2] (l12)at (3.3,3.5) {};
				
				\node[opacity=0.2] (l13)at (3.3,4.5) {};
				\node[opacity=0.2] (l14)at (2.5,4.5) {};
				\node[opacity=0.2] (l15)at (4.1,4.5) {};
				\end{scope}
				\node[opacity=0.2] at (2.5,4.5) {\scalebox{1.5}{$\times$}};
				\node at (3.3,5.2) {\scalebox{1.5}{$\overset{\times}{\downarrow}$}};
				\begin{scope}[>={latex[black]}, every node/.style={regular polygon,regular polygon sides=7,draw},
				every edge/.style={draw=black,very thick}]
				\path [-] (h1) edge (h2);
				\path [-] (h1) edge (l2);
				\path [-] (h1) edge (l3);
				\path [-] (h2) edge (l4);
				\path [-] (h2) edge (l6);
				
				\path [-] (l2) edge (l7);
				\path [-] (l2) edge (l8);
				\path [-] (l2) edge (l9);
				
				\path [opacity=0.3] (l3) edge (l10);
				\path [opacity=0.3] (l3) edge (l11);
				\path [opacity=0.3](l3) edge (l12);
				
				\path [opacity=0.3] (l12) edge (l13);
				\path [opacity=0.3] (l12) edge (l14);
				\path [opacity=0.3] (l12) edge (l15);
				\end{scope}
				\end{tikzpicture}}
		\end{minipage}
		\begin{minipage}[b]{.33\linewidth}
			\centering\scalebox{.7}{\begin{tikzpicture}
				\begin{scope}[every node/.style={regular polygon,regular polygon sides=7,very thick,draw}, minimum size=7pt]
				\node (h1) at (0,1.5) {};
				\node[fill=gray!100,opacity=0.3] (l2) at (0,2.5) {};
				\node (h2) at (-2.5,2.5) {};
				\node (l3)at (2.5,2.5) {};
				\node (l4)at (-3.1,3.5) {};

				\node (l6)at (-1.9,3.5) {};
				\node[opacity=0.3] (l7)at (-0.8,3.5) {};
				\node[opacity=0.3] (l8)at (0,3.5) {};
				\node[opacity=0.3] (l9)at (0.8,3.5) {};
				\end{scope}
								\node at (-0.8,4.2) {\scalebox{1.5}{$\overset{\circ}{\downarrow}$}};
				\begin{scope}[>={latex[black]}, every node/.style={regular polygon,regular polygon sides=7,draw},
				every edge/.style={draw=black,very thick}]
				\path [-] (h1) edge (h2);
				\path [opacity=0.3] (h1) edge (l2);
				\path [-] (h1) edge (l3);
				\path [-] (h2) edge (l4);
				\path [-] (h2) edge (l6);
				
				\path [opacity=0.3] (l2) edge (l7);
				\path [opacity=0.3] (l2) edge (l8);
				\path [opacity=0.3] (l2) edge (l9);
				\end{scope}
				\end{tikzpicture}}
		\end{minipage}\begin{minipage}[b]{.19\linewidth}
			\centering\scalebox{.7}{\begin{tikzpicture}
				\begin{scope}[every node/.style={regular polygon,regular polygon sides=7,very thick, draw}, minimum size=7pt]
				\node[fill=gray!25] (h1) at (-1.5,1.5) {};
				\node[opacity=0.3] (h2) at (-2.5,2.5) {};
				\node[opacity=0.3] (l3)at (-0.5,2.5) {};
				\node[opacity=0.3] (l4)at (-3.1,3.5) {};

				\node[opacity=0.3] (l6)at (-1.9,3.5) {};
				\end{scope}
								\node at (-1.9,4.2) {\scalebox{1.5}{$\overset{\circ}{\downarrow}$}};
				\begin{scope}[>={latex[black]}, every node/.style={regular polygon,regular polygon sides=7,draw},
				every edge/.style={draw=black,very thick}]
				\path [opacity=0.3] (h1) edge (h2);
				\path [opacity=0.3](h1) edge (l3);
				\path [opacity=0.3] (h2) edge (l4);
				\path [opacity=0.3](h2) edge (l6);
				\end{scope}
				\end{tikzpicture}}
		\end{minipage}\begin{minipage}[b]{.05\linewidth}
			\centering\scalebox{.7}{\begin{tikzpicture}
 \begin{scope}[every node/.style={regular polygon,regular polygon sides=7,very thick,draw}, minimum size=7pt]
				\node (h1) at (0,1.5) {};
				\end{scope}
\draw (0,1.5) circle (1mm)  [fill=white!100];
				\end{tikzpicture}}
		\end{minipage}
	\end{center}
	\caption[Pruning operators]{Each tree in the picture (except the first) is obtained from its left neighbour via the pruning operator~$\pi_\ell^\star$, where $\ell$ is the leaf indicated by $\downarrow$, and $\star\in\{\times,\circ\}$ is the symbol above $\downarrow$. The filled grey vertex in the leftmost tree is $w_\ell(\bal)$, i.e. the most recent ancestor of $\ell$ that is a firewall; the filled grey vertex in the second and third tree is~$\rho(R_\ell(\bal))$, i.e. the root of the region containing $\ell$.}
	\label{fig:trcirc}
\end{figure}

Now, we construct the process $(\ba_t)_{t\geq 0}$, which we call the  \emph{pruned ASG (or pASG) process}. The idea is to dynamically resolve the mutations arising in the eASG by replacing the mutation operations by the  pruning operations. More precisely, let~$\bal\in\Xi^{\pA}$ and set $\ba_0\defeq \bal$. Define $\asg\in{\Xi}$ by adding a child without a mark to every marked leaf of $\bal$. Note that $\bal$ is an admissible pruning of $\asg$. Let $\brea=(\brea_t)_{t\geq0}$ be the eASG process starting at $\asg$. Construct $(\ba_t)_{t\geq 0}$ by updating its state at any transition of $(\brea_t)_{t\geq0}$ as follows. Assume $(\ba_t)_{t\in[0,r)}$ has been constructed in such a way that, for any $t\in[0,r)$, $\ba_t$ is an admissible pruning of $\brea_t$, and that at time~$r$ a transition occurs in $\brea$ at some leaf $\ell\in L(\brea_{r-})$. If $\ell\notin \hat{L}(\ba_{{r-}})$, set $\ba_{{r}}=\ba_{{r-}}$. If $\ell\in \hat{L}(\ba_{r-})$ and $\brea_r= {(\brea_{r-})}^{\star}_\ell$ with $\star\in\{\curlyvee,\pitch\}$, set $\ba_r:= {(\ba_{r-})}^{\star}_\ell$. If $\ell\in \hat{L}(\ba_{r-})$ and $\brea_r= {(\brea_{r-})}^{\star}_\ell$ with $\star\in\{\times,\circ\}$, set $\ba_r:= \pi_\ell^\star(\ba_{r-})$.

\smallskip

In Section~\ref{sec:proofspASG} we show that the pASG process is a continuous-time Markov chain on $\Xi^\pA$ with absorbing states $\rootdel$ and $\rootben$, and with the following transition rates. For $\bal\in\Xi^\pA$ and $\ell\in \hat{L}(\bal)$,
  $$q_{\ba}(\bal,\bal^\curlyvee_\ell)=s,\quad q_{\ba}(\bal,\bal^\pitch_\ell)=\gamma,\quad q_{\ba}(\bal,\pi_{\ell}^\times(\bal))=u\nu_1,\quad q_{\ba}(\bal,\pi_{\ell}^\circ(\bal))=u\nu_0.$$

Moreover, we show that the so-constructed process $(\ba_t)_{t\geq 0}$ satisfies (1), (2) and (3) of Definition \ref{adp}. Property (3) translates the duality in Theorem~\ref{sec:mainresults:thm:dualityASG} to the pASG process. 
 \begin{coro}[Duality pASG] \label{dualpASG}
Let $(\ba_t)_{t\geq 0}$ be the pASG process and let~$y(\cdot;y_0)$ be the solution of the ODE~\eqref{eq:dlimitdiffeq} with initial value $y_0\in [0,1]$. Then, for $\bal \in \Xi^{\pA}$ and $t\geq 0$, $$H \big (\bal,y(t;y_0) \big )=\E_{\bal}[H(\ba_t,y_0)].$$
\end{coro}
The proof of Corollary~\ref{dualpASG} is also provided in Section~\ref{sec:proofspASG}.

\begin{remark}
	If $s=\nu_0=0$ and $u=1$, Corollary \ref{dualpASG} resembles a duality of the mean-field limit of a cooperative branching process on the complete graph~\citep[Prop.~I.2.1.4]{Mach2017}. The main difference is that the dual process in~\citep{Mach2017} does \emph{not} resolve all information contained in the mutation marks on the spot.
\end{remark}

\subsection{The stratified ASG}\label{sec:mainresults:subsec:sASG}
We have seen above that within each region, types propagate as in the non-interactive case. To determine the type distribution at the root of a region, it therefore suffices to count  the leaves within the region; the tree structure within it is irrelevant. Fix a sampling distribution at the leaves. For the type distribution of the root of the entire tree, it is then sufficient to count leaves within each region and to determine the connections between the regions. We encode this structure by collapsing regions into single vertices, which leads us to a new kind of trees. The corresponding transformations will simplify the representation of the Boolean function $c\mapsto\cpv{\varrho(\alpha)}$.

\begin{definition}[Primary and secondary vertices, tripod tree]
A \emph{tripod tree} is a directed rooted tree, where every vertex at an odd distance from the root has outdegree 2. Let $\Xi^\trip$ denote the set of tripod trees, and for every $\tau \in \Xi^\trip$, let $V^1(\tau)$ and $V^2(\tau)$ be the set of vertices at an even and odd distance to the root, respectively. $V^1(\tau)$ and $V^2(\tau)$ will be called primary and secondary vertices, respectively.
A \emph{weighted tripod tree} is a pair $\Ts = (\tau,m)$ such that $\tau\in \Xi^\trip$ and $m:V^1(\tau)\to\N_0$ with $m(\ell)>0$ for every $\ell\in L(\tau)$. Let $\Upsilon$ denote the set of weighted tripod trees. Moreover, $\Upsilon_\Delta\defeq\Upsilon\cup\{\Delta\}$, where $\Delta$ is an isolated point.
\end{definition}
We denote the weighted tripod tree that consists only of a root of weight~$n$ by $\sroot{n}$. The set $\Upsilon_\Delta$ is countable and therefore becomes a Polish space when equipped with the discrete topology. In contrast to the  trees from the previous section, the trees here are unordered. Each primary vertex corresponds to a region in a pruned tree with the weight indicating the number of leaves in that region. The secondary vertices encode the connections between the regions.
\begin{figure}[h!]
\begin{minipage}{0.3\textwidth}
			\begin{center}
				\scalebox{.5}{\begin{tikzpicture}
					\begin{scope}[every node/.style={circle,thick,draw}, minimum size=20pt]
					\node[fill=gray!80] (h0) at (4,0) {$0$};

					\node (h2) at (3,2) {$2$};
					\node (h3) at (5,2) {$5$};
		
				\end{scope}
					\begin{scope}[every node/.style={circle,thick,draw}, minimum size=1pt]
                    \node[fill=black!80]  (h1) at (4,1) {};
					\end{scope}

					\begin{scope}[>={latex[black]}, every node/.style={fill=white,circle},
					every edge/.style={draw=black,very thick}]
					
					\path [-] (h0) edge (h1);
					\path [-] (h1) edge (h2);
					\path [-] (h1) edge (h3);
					\end{scope}

					\end{tikzpicture}}
		\end{center}	\end{minipage}\begin{minipage}{0.3\textwidth}
			\begin{center}
				\scalebox{.5}{\begin{tikzpicture}
					\begin{scope}[every node/.style={circle,thick,draw}, minimum size=20pt]
					\node[fill=gray!80] (h0) at (4,0) {$1$};

					\node (h2) at (3,2) {$2$};
					\node (h3) at (5,2) {$1$};

					\node (h5) at (6,1) {$4$};
					\node (h6) at (6,-1) {$1$};
					
					\node (h8) at (2,1) {$5$};
					\node (h9) at (2,-1) {$1$};
		
				\end{scope}
					\begin{scope}[every node/.style={circle,thick,draw}, minimum size=1pt]
                    \node[fill=black!80]  (h1) at (4,1) {};
                    \node[fill=black!80](h4) at (5,0) {};
                    \node[fill=black!80](h7) at (3,0) {};
					\end{scope}

					\begin{scope}[>={latex[black]}, every node/.style={fill=white,circle},
					every edge/.style={draw=black,very thick}]
					
					\path [-] (h0) edge (h1);
					\path [-] (h1) edge (h2);
					\path [-] (h1) edge (h3);
					
                    \path [-] (h0) edge (h4);
					\path [-] (h4) edge (h5);
					\path [-] (h4) edge (h6);

					\path [-] (h0) edge (h7);
					\path [-] (h7) edge (h8);
					\path [-] (h7) edge (h9);
					
					\end{scope}

					\end{tikzpicture}}
		\end{center}	\end{minipage} \begin{minipage}{0.3\textwidth}
			\begin{center}
				\scalebox{.5}{\begin{tikzpicture}
					\begin{scope}[every node/.style={circle,thick,draw}, minimum size=20pt]
					\node[fill=gray!80] (h0) at (4,0) {$1$};

					\node (h2) at (3,2) {$8$};
					\node (h3) at (5,2) {$1$};

					\node (h5) at (6,1) {$4$};
					\node (h6) at (6,-1) {$1$};
					
					\node (h8) at (0,2) {$5$};
					\node (h9) at (0,0) {$1$};

					\node  (h11) at (2,-1) {$1$};
					\node  (h12) at (2,1)  {$0$};
		
				\end{scope}
					\begin{scope}[every node/.style={circle,thick,draw}, minimum size=1pt]
                    \node[fill=black!80] (h1) at (4,1) {};
                    \node[fill=black!80] (h4) at (5,0) {};
                    \node[fill=black!80] (h7) at (1,1) {};
                    \node[fill=black!80] (h10) at (3,0) {};
					\end{scope}

					\begin{scope}[>={latex[black]}, every node/.style={fill=white,circle},
					every edge/.style={draw=black,very thick}]
					
					\path [-] (h0) edge (h1);
					\path [-] (h1) edge (h2);
					\path [-] (h1) edge (h3);
					
                    \path [-] (h0) edge (h4);
					\path [-] (h4) edge (h5);
					\path [-] (h4) edge (h6);

					\path [-] (h12) edge (h7);
					\path [-] (h7) edge (h8);
					\path [-] (h7) edge (h9);
					
                    \path [-] (h0) edge (h10);
					\path [-] (h10) edge (h11);
					\path [-] (h10) edge (h12);
					\end{scope}
					\end{tikzpicture}}
		\end{center}	\end{minipage}
			
	\caption[Tripods]{Some examples of weighted tripods. Secondary vertices are small bullets, primary vertices are large open circles and contain a number indicating the weight; the root has a grey filling. }\label{tripod}
\end{figure}

More precisely, a pruned tree is mapped into a tripod tree in the following way (see Fig.~\ref{stratus}). Recall that first all vertices within a region of the pruned tree collapse to form a primary vertex; its weight is the number of unmarked leaves in that region. The region containing the root in the pruned tree is the root of the tripod tree. A region pair arising at a ternary branching in the pruned tree corresponds to the two children of a secondary vertex in the tripod tree, with the parent of this vertex   corresponding to the region parental  to the region pair. The following definition formalises this mapping.

\begin{definition}[Stratification]\label{def:stratmap} The \emph{stratification map} $s:\Xi^{\pA}\to \Upsilon_\Delta$ is defined as follows. For $\bal=\rootben$, set $s(\bal)=\Delta$. For $\bal\neq\rootben$, set $s(\bal)=(\tau,m)$, where $\tau$ has vertex set $V(\tau)\defeq V^1(\tau)\cup V^2(\tau)$,
 $$V^1(\tau):=\{R:\, \textrm{$R$ is a region of $\bal$}\},\quad V^2(\tau):=\{v\in V(\bal) : \deg(v)=3\},$$ 
and, for any $R\in V^1(\tau)$ and $v\in V^2(\tau)$:
\begin{itemize}
\item $v$ is the child of $R$ in $\tau$ if and only if $v\in R$ in $\bal$,
\item $R$ is a child of $v$ in $\tau$ if and only if $\rho(R)$ is a child of~$v$ in $\bal$,
\item $m(R)$ is the number of unmarked leaves in $R$.
\end{itemize}
In particular, if $\bal\neq \rootben$ has no ternary branchings and $n$ unmarked leaves, then $s(\bal)=\sroot{n}$. 
\end{definition}

\begin{figure}[t]
\begin{minipage}{0.3\textwidth}
			\begin{center}
				\scalebox{.5}{\begin{tikzpicture}
					\begin{scope}[every node/.style={regular polygon,regular polygon sides=7,draw}, minimum size=7pt]
					\node (h0) at (4,0) {};
					
					\node[fill=black!80](h1) at (3,1) {};
					\node(h2) at (5,1) {};
					
					\node[fill=black!80] (h3) at (2,2) {};
					\node (h4) at (2.75,2) {};
					\node (h5) at (3.5,2) {};
					\node (h6) at (4.5,2) {};
					\node[fill=black!80] (h7) at (5.5,2) {};

					\node (h31) at (1.25,3) {};
					\node (h32) at (2,3) {};
					\node[fill=black!80] (h33) at (2.75,3) {};
					
					\node (h9) at (4.75,3) {};
					\node (h10) at (5.5,3) {};
					\node (h11) at (6.25,3) {};
					
					\node (h13) at (5.75,4) {};
					\node (h12) at (6.75,4) {};
					
					\node (h331) at (2.25,4) {};
					\node (h332) at (2.75,4) {};
					\node (h333) at (3.25,4) {};
					\end{scope}

					\begin{scope}[>={latex[black]}, every node/.style={fill=white,circle},
					every edge/.style={draw=black,very thick}]
					
					\path [-] (h0) edge (h1);
					\path [-] (h0) edge (h2);
					
					\path [-] (h1) edge (h3);
					\path [-] (h1) edge (h4);
					\path [-] (h1) edge (h5);
					
					\path [-] (h3) edge (h31);
					\path [-] (h3) edge (h32);
					\path [-] (h3) edge (h33);

					\path [-] (h2) edge (h6);
					\path [-] (h2) edge (h7);

					\path [-] (h7) edge (h9);
					\path [-] (h7) edge (h10);
					\path [-] (h7) edge (h11);
					
					\path [-] (h11) edge (h12);
					\path [-] (h11) edge (h13);
					
                    \path [-] (h33) edge (h331);
					\path [-] (h33) edge (h332);
					\path [-] (h33) edge (h333);

					\end{scope}

        \begin{pgfonlayer}{background}
        \highlight{5mm}{cyan}{(4.75,3) -- (5.5,2) -- (5,1) -- (4.5,2) -- (5,1) -- (4,0) -- (3,1) -- (2,2) -- (1.25,3)}
        \highlight{5mm}{green}{(2.75,3) -- (2.25,4)}
        \highlight{5mm}{blue}{(3.25,3.9) -- (3.25,4.1)}
        \highlight{5mm}{red}{(2.75,3.9) -- (2.75,4.1)}
        \highlight{5mm}{orange}{(2.75,1.9)--(2.75,2.1)}
        \highlight{5mm}{purple}{(3.5,1.9)--(3.5,2.1)}
        \highlight{5mm}{magenta}{(2,2.9)--(2,3.1)} 
        \highlight{5mm}{brown}{(5.5,2.9)--(5.5,3.1)} 
        \highlight{5mm}{violet}{(5.75,4) -- (6.25,3) --(6.75,4) }
        
        \end{pgfonlayer}
                                    \end{tikzpicture}}
		\end{center}	\end{minipage} \begin{minipage}{0.3\textwidth}
			\begin{center}
				\scalebox{.5}{\begin{tikzpicture}
				
				    \begin{scope}[every node/.style={circle,thick,draw}, minimum size=1pt]
					\node[fill=black!80, scale=0.5] (t0) at (3.1,1.5) {};
					\node[fill=black!80, scale=0.5] (t1) at (2.4,2.5) {};
					\node[fill=black!80, scale=0.5] (t2) at (2.95,3.5) {};
                    \node[fill=black!80, scale=0.5] (t3) at (5.9,2.5) {};
					\end{scope}
				
				   \begin{scope}[>={latex[black]}, every node/.style={fill=white,circle},
					every edge/.style={draw=black,very thick}]
					\path [-] (3,1.19) edge (t0);
					\path [-] (t0) edge (2.75,1.75);
					\path [-] (t0) edge (3.50,1.75);
					\path [-] (2.1,2.1) edge (t1);
					\path [-] (t1) edge (2,2.75);
                    \path [-] (t1) edge (2.75,2.85);
                    
                    \path [-] (2.85,3.15) edge (t2);
                    
                    \path [-] (t2) edge (2.75,3.75);
                     \path [-] (t2) edge (3.25,3.75);
                     \path [-] (5.6,2.1) edge (t3);
                    
                    \path [-] (t3) edge (5.5,2.75);
                     \path [-] (t3) edge (6.25,2.85);
                     
					\end{scope}

							 \begin{pgfonlayer}{background}
        \highlight{3mm}{cyan}{(4.75,3) -- (5.5,2) -- (5,1) -- (4.5,2) -- (5,1) -- (4,0) -- (3,1) -- (2,2) -- (1.25,3)}
        \highlight{3mm}{green}{(2.75,3) -- (2.25,4)}
        \highlight{3mm}{blue}{(3.25,3.9) -- (3.25,4.1)}
        \highlight{3mm}{red}{(2.75,3.9) -- (2.75,4.1)}
        \highlight{3mm}{orange}{(2.75,1.9)--(2.75,2.1)}
        \highlight{3mm}{purple}{(3.5,1.9)--(3.5,2.1)}
        \highlight{3mm}{magenta}{(2,2.9)--(2,3.1)} 
        \highlight{3mm}{brown}{(5.5,2.9)--(5.5,3.1)} 
        \highlight{3mm}{violet}{(5.75,4) -- (6.25,3) --(6.75,4) }
        
        \end{pgfonlayer}		
					\end{tikzpicture}}
		\end{center}	\end{minipage}\begin{minipage}{0.3\textwidth}
			\begin{center}
				\scalebox{.5}{\begin{tikzpicture}
					\begin{scope}[every node/.style={circle,thick,draw}, minimum size=20pt]
					\node[fill=cyan!30, line width=1mm] (h0) at (4,0) {$3$};

					\node[fill=orange!30] (h2) at (3,2) {$1$};
					\node[fill=purple!30] (h3) at (5,2) {$1$};

					\node[fill=brown!30] (h5) at (6,1) {$1$};
					\node[fill=violet!30] (h6) at (6,-1) {$2$};
					
					\node[fill=green!30] (h8) at (2,1) {$1$};
					\node[fill=magenta!30] (h9) at (2,-1) {$1$};
					
                    \node[fill=blue!30] (h11) at (0,2) {$1$};
					\node[fill=red!30] (h12) at (0,0) {$1$};

				\end{scope}
					\begin{scope}[every node/.style={circle,thick,draw}, minimum size=1pt]
                    \node[fill=black!80]  (h1) at (4,1) {};
                    \node[fill=black!80](h4) at (5,0) {};
                    \node[fill=black!80](h7) at (3,0) {};
                    \node[fill=black!80](h10) at (1,1) {};
                    
					\end{scope}

					\begin{scope}[>={latex[black]}, every node/.style={fill=white,circle},
					every edge/.style={draw=black,very thick}]
					
					\path [-] (h0) edge (h1);
					\path [-] (h1) edge (h2);
					\path [-] (h1) edge (h3);
					
                    \path [-] (h0) edge (h4);
					\path [-] (h4) edge (h5);
					\path [-] (h4) edge (h6);

					\path [-] (h0) edge (h7);
					\path [-] (h7) edge (h8);
					\path [-] (h7) edge (h9);
					
                     \path [-] (h10) edge (h8);
					   \path [-] (h10) edge (h11);
					      \path [-] (h10) edge (h12);
					\end{scope}

					\end{tikzpicture}}
		\end{center}	\end{minipage}
			
	\caption[Stratification]{A pruned tree and its regions (left), an illustration of the stratification procedure (middle), and the stratification (right); the root of the tripod tree has a bold outline.}	\label{stratus}
\end{figure}

To determine the root type of $\bal\in\Xi^{\pA}$ via the stratification $s(\bal)$, we proceed in two steps. First, we assign a pre-type $c_R\in\{0,1\}$ to each region $R$. For a given reduced leaf-type configuration of~$\bal$, this pre-type is defined as the root type of~$R$ when types are propagated within a region as in the non-intercative case, and using the rule that the type of a vertex $v$ with $\deg(v)=1$ in $R$ (i.e. $\deg(v)=3$ in $\bal$) is the type of $\chil{v}$. Clearly the pre-types and the root types under the type-propagation for the entire tree can differ. In any case, this procedure yields an assignment of types to the primary vertices in $s(\bal)$. Next, we associate to this assignment  an \emph{output type} that takes into account the connections between the regions. 

\begin{definition}[Output type] Let $\tau\in\Xi^\trip$. A \emph{primary vertex-type configuration} of $\tau$ is a vector $c\defeq(c_v)_{v\in V^1(\tau)}\in\{0,1\}^{V^1(\tau)}$. The \emph{output type} in $\tau$ under $c$ is the type $\tf(c,\tau)\in\{0,1\}$ defined recursively as follows.
\begin{itemize}
 \item If $\tau$ consists only of the root $\rho(\tau)$, then $\tf(c,\tau)=c_{\rho(\tau)}$.
 \item If $\tau$ has at least one secondary vertex, 
 $$\tf(c,\tau)=c_{\rho(\tau)} \prod_{v\,\text{child of }\rho(\tau) } 
\max\left\{\tf(c,\tau_{w}):\, w\, \text{ is child of }v\right\},$$
where $\tau_{w}$  is the subtree of $\tau$ rooted in $w$.
\end{itemize}
\end{definition}

In Lemma~\ref{out} in Section~\ref{sec:proofssASG}, we prove that, under any reduced leaf-type configuration, the root type of $\bal \in  \Xi^{\pA}$ with $\bal\neq\rootben$  equals the output type in $s(\bal)$, if every primary vertex $R$ is assigned type $1$ when all unmarked leaves in $R$ have type $1$. 

\smallskip

For $z\in[0,1]$, $\Ts\defeq (\tau,m)\in\Upsilon$ with vertex set~$V(\tau)$, and $v\in V^1(\tau)$, let $\Cs_v(z)$ be a Bernoulli random variable with parameter $z^{m(v)}$; we understand $\Cs(z)=(\Cs_v(z))_{v\in V^1(\tau)}$ as a primary-vertex configuration. Define $\Hs(\Ts,z)$ to be the probability that the output type in $\tau$ under $\Cs(z)$ is $1$. Moreover, we set $\Hs(\Delta,z)\defeq 0$ for all $z\in[0,1]$. Clearly, this is the tripod-tree analogue to the function~$H$ of the previous subsection. Indeed, it will be a consequence of Lemma~\ref{out} that for $\bal \in \Xi^{\pA}$ and $z\in[0,1]$, 
\begin{equation}\label{root=winnertype}
H(\bal,z) = \Hs(s(\bal),z).
\end{equation}
Now, $\Tp=(\Tp_t)_{t\geq 0}$ with $\Tp_t\defeq s(\bar{a}_t)$ is called the \emph{sASG process}. Thanks to \eqref{root=winnertype} and  Corollary \ref{dualpASG},  $\Tp$ satisfies the following duality with the mutation-selection equation.
\begin{theorem}[Duality sASG]\label{thm:duality}
	The sASG process~$\mathscr{T}$ and the solution~$(y(t;y_0))_{t\geq 0}$ of~\eqref{eq:dlimitdiffeq} satisfy the duality relation 
	\begin{equation}
	\Hs(\Ts,y(t;y_0))=\E_{\Ts}[\Hs(\mathscr{T}_t,y_0)] \quad \text{for} \ y_0\in [0,1],\ \Ts\in \Upsilon_\Delta, \  t\geq 0. \label{eq:dualityternarytree}
	\end{equation}
	In particular, for~$t\geq 0$ and~$y_0\in [0,1]$, we obtain the stochastic representation
	\begin{equation} \label{eq:dualityternarytreesimple}
	y(t;y_0)=\E_{\ssroot{1}}[\Hs(\mathscr{T}_t,y_0)].
	\end{equation}
\end{theorem}
Theorem~\ref{thm:duality} is proved in Section~\ref{sec:proofssASG}.

\smallskip
 
The sASG process is defined on the basis of the pASG process. But the stratification mapping was in fact chosen so as to make $\Tp$ a Markov process. The following operators provide the corresponding transitions.
\begin{definition}\label{def:sASGoperations}
For $\Ts\in\Upsilon$ and $v\in V^1(\tau)$, define $\Ts^{\nwedge}_{v}$, $\Ts^{\pitch}_{v}$, $\Ts^{\times}_{v},\Ts^{\circ}_{v}\in\Upsilon_\Delta$ as follows.
\begin{itemize}
 \item $\Ts^{\nwedge}_{v}$ is obtained from~$\Ts$ by increasing the weight of~$v$ by $1$; the tripod tree and all other weights remain as in~$\Ts$.
 \item $\Ts^{\pitch}_{v}$ arises from~$\Ts$ by adding a secondary vertex $\bar{v}$ as a child of~$v$; and adding the primary vertices $\chim{v}$ and $\chir{v}$ as  children of~$\bar{v}$. $\chim{v}$ and $\chir{v}$ obtain weight~$1$. The remaining tripod tree and weights remain unchanged. (Here, $\chim{v},\chir{v}$ are not meant to indicate orientation of the tree; they are rather convenient notation for  these vertices.)
 \item $\Ts^{\times}_{v}$ arises from~$\Ts=(\tau,m)$ as follows. If $m(v)>1$, or if $v$ is not a leaf, or if $v$ is the root, then $\Ts^{\times}_{v}=(\tau,\tilde{m})$, where for $z\in V(\tau)\setminus\{v\}$, $\tilde{m}(z)=m(z)$, and $\tilde{m}(v)=m(v)-1$. If $v$ is a leaf that is not the root and $m(v)=1$, let~$w$ be the closest ancestor of~$v$ that is the root of $\Ts$ or a child of a primary vertex with positive weight; remove from~$\Ts$ the subtree rooted in~$w$ to obtain~$\Ts^\times_v$.
 \item $\Ts^{\circ}_{v}$ is set to $\Delta$ if $v$ is the root of $\Ts$. If~$v$ is not the root, let $w$ and $g_v$ be the sibling and the grandparent of $v$, respectively. 
 Replace  the subtree rooted in $g_v$ by the subtree rooted in~$w$; add the weight of~$g_v$ to the weight of~$w$.
\end{itemize}
\end{definition}

In Section~\ref{sec:proofssASG} we show that $\Tp$ is a continuous-time Markov chain with values in $\Upsilon_\Delta$
and transition rates
	\begin{align*}
	q_{\Te}^{}(\Ts,\Ts^{\nwedge}_{v})\defeq s\,m(v),\quad  q_{\Te}^{}(\Ts,\Ts^{\pitch}_{v})\defeq\gamma\, m(v),\quad q_{\Te}^{}(\Ts,\Ts^{\times}_{v})\defeq u\nu_1\, m(v),\quad q_{\Te}^{}(\Ts,\Ts^{\circ}_{v})\defeq u\nu_0\, m(v),
	\end{align*}
	for~$\Ts=(\tau,m)\in \Upsilon$ and~$v\in V^1(\tau)$. The states~$\sroot{0}$ and~$\Delta$ are absorbing. 	

\subsection{Applications I - long-term type frequency}\label{sec:mainresults:subsec:application}
The long-term behaviour of the forward process can be described in terms of the backward process using our constructions. The basic idea is to take $t\to\infty$ in the duality relation \eqref{eq:dualityternarytree}. To this end, we consider the long-term behaviour of $\Hs(\Tp_t,y_0)$ as $t\to\infty$. The sASG process $\Tp$ has two traps $\sroot{0}$ and $\Delta$, where the function $\Hs(\cdot,y_0)$ takes the values $1$ and $0$, respectively, irrespective of the value of $y_0$. Now, define
\[T_{\ssroot{0}}\defeq \inf\big\{r>0:\mathscr{T}_r=\sroot{0}\big\},\quad T_{\Delta}\defeq \inf\{r>0:\mathscr{T}_r=\Delta\}\quad \text{and}\quad T_{\mathrm{abs}}\defeq\min \{T_{\ssroot{0}},T_{\Delta}\}.\]
If $T_{\mathrm{abs}}<\infty$ almost surely, then for any $y_0\in[0,1]$, $\Hs(\Tp_t,y_0)$ converges as $t\to\infty$ to a Bernoulli random variable with parameter $\Pb(T_{\sroot{0}}<\infty)$. In this case, the duality \eqref{eq:dualityternarytree} also implies that $\Pb(T_{\sroot{0}}<\infty)$ is the unique equilibrium of \eqref{eq:dlimitdiffeq}
and is stable. The analysis is more involved if $T_{\mathrm{abs}}=\infty$ has positive probability.
\begin{proposition}\label{prop:treegrowsinfty}
Define the total weight of $\Ts=(\tau,m)\in\Upsilon$ as $M(\Ts)\defeq\sum_{v\in V^1(\tau)} m (v)$. On~$\{T_{\mathrm{abs}}=\infty\}$, we have~$\lim_{r\to\infty} M(\Te_r)=\infty$.
\end{proposition}
Proposition~\ref{prop:treegrowsinfty} is proved in Section~\ref{sec:applicationIproofs}.
\smallskip

If $\nu_0=0$ and $M(\Te_t)\to\infty$, it is reasonable to expect the following behaviour of $\Hs(\Te_t,y_0)$. If $s>u$, the weight of a primary vertex that is never pruned (for example the root) is a birth-death process with birth rate $s$ and death rate $u$ and hence converges to $0$ or $\infty$ almost surely. In the first case, the probability that such a vertex gets the unfit type converges to $1$ for all $y_0\in[0,1]$; in the second case, for all $y_0\in[0,1)$, this probability is~$0$. So, we may expect that $\Hs(\Tp_t,y_0)$ has the same limit for all $y_0\in[0,1)$. If $s<u$ the size of a given primary vertex will tend to $0$ almost surely, but if $\gamma$ is sufficiently large, the number of vertices with positive weight can grow to infinity; so the limit of $\Hs(\Tp_t,y_0)$ can depend on $y_0$ in a more subtle way. 
\smallskip

Recall that $y_\infty(y_0)=\lim_{t\to\infty}y(t;y_0)$ and that $\ymin$ and $\ymax$ denote the smallest and largest equilibrium of~\eqref{eq:dlimitdiffeq}. The long-term behaviour of $\Hs(\Tp_t,y_0)$ and its connection with the mutation-selection equation are given in the next theorem, which is the main result of this section.
\begin{theorem}[Stochastic representation of equilibria]\label{thm:representationequilibrium}
For any~$y_0\in [0,1]$, $\Hs(\mathscr{T}_t,y_0)$ converges almost surely as $t\to\infty$ to a random variable $\Hs_{\infty}(y_0)\in[0,1]$. In addition, if~$y_0$ is not an unstable equilibrium of \eqref{eq:dlimitdiffeq} located in $(0,1)$ and $\Te_0=\ssroot{1}$, then~$\Hs_{\infty}(y_0)$ has a Bernoulli distribution with parameter~$y_{\infty}^{}(y_0)$. 
Moreover, for any~$y_0\in[0,1]$, we have
	\begin{equation}\label{eq:yinftyw1}
		y_{\infty}^{}(y_0)=\Pb_{\ssroot{1}}\big(T_{\ssroot{0}}<\infty\big)+\E_{\ssroot{1}}\big[\1_{\{T_{\mathrm{abs}}=\infty \}} \Hs_{\infty}(y_0)\big].\end{equation}
In particular, $\ymin=\Pb_{\ssroot{1}}(T_{\ssroot{0}}<\infty)$ and $\ymax=\Pb_{\ssroot{1}}(T_{\Delta}=\infty)$.
\end{theorem}
Note that in particular $\P_{\ssroot{1}}(T_{\mathrm{abs}}<\infty)=1$ if and only if $\ymin=\ymax$. Let $\Attr(\bar{y})\defeq\big\{y_0\in [0,1]: \lim_{t\to\infty}y(t;y_0)=\bar{y}\big\}$ denote the domain of attraction of an equilibrium $\bar{y}$ of \eqref{eq:dlimitdiffeq}. The next result provides a more precise description of $\Hs_\infty(y_0)$. 
\begin{coro}\label{coro:refinebernoulli} Assume $\Te_0=\sroot{1}$ and $y_0 \in (0,1)$ is not an unstable equilibrium of~\eqref{eq:dlimitdiffeq}. Then, almost surely,
	$$\Hs_{\infty}(y_0)=\begin{cases}
		\1_{\{T_{\ssroot{0}}<\infty\}},&\text{if }y_0\in\Attr(\ymin),\\	\1_{\{T_{\Delta}=\infty\}},&\text{if }y_0\in\Attr(\ymax).\\	\end{cases}$$
\end{coro}
Theorem~\ref{thm:representationequilibrium} and Corollary~\ref{coro:refinebernoulli} are proved in Section~\ref{sec:applicationIproofs}.
Corollary~\ref{coro:refinebernoulli} tells us that when $\Tp$ is not trapped, $\Hs_\infty(y_0)$ is $0$ for $y_0$ close enough to $0$ (for $y_0\in\Attr(\ymin)$), and $1$ for $y_0$ close enough to $1$ (for $y_0\in\Attr(\ymax)$). The critical value associated to this dichotomy is described in the next result.

\begin{proposition}\label{prop:ycritical}
	Assume that~$\P_{\ssroot{1}}(T_{\mathrm{abs}}=\infty)>0$. Let $$y_c\defeq\inf\big\{y_0\in[0,1]:\, \P_{\ssroot{1}}\big(\Hs_{\infty}(y_0)=1\mid T_{\mathrm{abs}}=\infty\big)=1\big\}.$$ Then $y_c \in [\ymin,\ymax]$ is an equilibrium and $y_c=\sup\big\{y_0\in[0,1]:\, \P_{\ssroot{1}}\big(\Hs_{\infty}(y_0)=0\mid T_{\mathrm{abs}}=\infty\big)=1\big\}.$ There are no equilibria in $[0,1]$ other than $\ymin,y_c,\ymax$.
\end{proposition}
Proposition~~\ref{prop:ycritical} is also proved in Section~\ref{sec:applicationIproofs}.

\smallskip

Finally, the results of this subsection connect the genealogical picture with the bifurcation structure described in Section~\ref{sec:mainresults:sub:detmodel} in the case~$\nu_0^{}=0$. Recall the parameter domains in~\eqref{eq:parameterregions}.
\begin{itemize}
	\item[$\Theta_1^{}$:] Here, $y_{\infty}(y_0)=1$ for all~$y_0\in [0,1]$. The sASG process absorbs almost surely in~$\sroot{0}$ and its long-term output type is $1$ almost surely, regardless of~$y_0$, see Theorem~\ref{thm:representationequilibrium}.
	
	\item[$\Theta_2^a$:] One has $y_{\infty}(y_0)\in [0,1)$ unless~$y_0=1$. The sASG process either absorbs in $\sroot{0}$ or its total weight tends to~$\infty$, by Proposition~\ref{prop:treegrowsinfty}. In the latter case, the long-term output type is $0$ (resp. $1$) for all $y_0\in[0,1)$ (resp. $y_0=1$) by Corollary~\ref{coro:refinebernoulli}.
	
	\item[$\Theta_2^b$:] One has $y_{\infty}(y_0)=\bar{y}_-$ for~$y_0\in [0,\bar{y}_-]$ and~$y_{\infty}(y_0)=1$ for~$y_0\in (\bar{y}_-,1]$. If the sASG process does not absorb, its total weight tends to~$\infty$ by Proposition~\ref{prop:treegrowsinfty}. In this case, the long-term output type depends on the initial type frequency: it is $0$ if~$y_0\leq\bar{y}_-$, and it is 1 if~$y_0>\bar{y}_-$.
	\item[$\Theta_3^{}$:] There are three equilibria in~$[0,1]$. Again, if the sASG process does not absorb, the total weight tends to~$\infty$, in which case the long-term output type depends on the initial type frequency: it is $0$ if~$y_0<\bar{y}_+$; it is 1 if~$y_0>\bar{y}_+$.
\end{itemize}

\begin{remark}[Connection to recursive tree processes and endogeny]
	Let us elaborate on a connection to~\citep{MSS2018} and endogeny. A (rather simplified) description of a \emph{recursive tree process} (RTP) is as follows. Consider a tree that has a randomly chosen function attached  to each internal vertex. The leaves of this tree are assigned types. Given the types of the children, the random function determines the type of the parent, and in this way the types propagate from the leaves, through the tree, to the root. 
	If the type distribution in each generation satisfies a consistency condition induced by a certain \emph{recursive distributional equation (RDE)} (which prescribes the way the functions at the internal vertices are chosen), one can apply Kolmogorov's extension theorem to extend this type propagation in some precise sense to infinite trees (e.g.~\citep[around Eq.~(2.7)]{mach2018a}); leading to an RTP associated to the solution of the RDE. This RTP is said to be \emph{endogenous} if the type at the root of the infinite tree if measurable with respect to the $\sigma$-field of functions at the internal vertices. Loosely speaking, the RTP is endogenous if the effect of the typing at the leaves for the root type disappears in the infinite tree. In our context, the leaves are independently typed according to a Bernoulli distribution with parameter~$y_0$. So if then the state at the root is measurable with respect to the $\sigma$-algebra generated by the functions attached to the internal vertices, then the RTP is said to be endogenous for~$y_0$ (see~\citep{aldous2005,mach2018a} for more details). 
	
	\smallskip

	It follows from~\citep[Prop.~1.16, Prop.~1.17]{MSS2018} that for $y_0$ an equilibrium of~\eqref{eq:dlimitdiffeq}, $\Hs_{\infty}(y_0)$ has a Bernoulli distribution if and only if the RTP corresponding to~$y_0$ is endogenous. In particular, the first part of Theorem~\ref{thm:representationequilibrium} implies that in our setup, the RTP is endogenous for all equilibria that are not unstable. Moreover, for $y_0\in \{\ymin,\ymax\}$, Corollary~\ref{coro:refinebernoulli} makes the endogeny of the underlying recursive tree process corresponding to $y_0$ explicit. An alternative way to recover the convergence statement in Theorem~\ref{thm:representationequilibrium} is via~\citep[Prop.~1.20, see also Sect.~2.1]{MSS2018} (alternatively,~\citep[Lem.~15]{aldous2005}).
	
	\smallskip
	
	For $s=\nu_0=0$ and $u=1$, the first part of Theorem~\ref{thm:representationequilibrium} can be also recovered from~\citep[Thm.~1.18]{MSS2018}. Moreover, if~$y_0$ is an unstable equilibrium, \citet[Lem.~1.19]{MSS2018} determine the first and second moments of~$\Hs_{\infty}(y_0)$ (to avoid confusion, we remark here that they in fact consider $1-\Hs_{\infty}(y_0)$). This allows them to infer that $\Hs_{\infty}(y_0)$ is then not Bernoulli distributed. Furthermore, they complement the result by numerical evaluations of the distribution function~\citep[Fig.~2]{MSS2018}. 
\end{remark}

\subsection{Applications II --- ancestral type distribution}\label{sec:mainresults:subsec:ancestraltype}
So far we have been concerned with a randomly chosen individual at present and determined its type either via the eASG, pASG, or sASG. In this section, we change perspective and look at the type of the \emph{ancestor} of the chosen individual at time~$r$ before the present, where we now return to the notion of ancestry  in the genealogical rather than the tree sense. By construction, the ancestor of the sampled individual together with its type can be extracted from the eASG after assigning types to its leaves, as inherent in Figs.~\ref{fig:peckingordersel} and \ref{fig:peckingorderint} and formalised as follows.

\begin{definition}[Ancestral leaf and ancestral type]
 Let $\asg\in\Xi$ and $c\in\{0,1\}^{L(\asg)}$ be a leaf-type configuration for $\asg$. Let $\cpr\in\{0,1\}^{V(\asg)}$ be the vertex-type propagation of $c$ (see Definition~\ref{typro}). Define the vector $(\lambda_v^c(\asg))_{v\in V(\asg)}\in{L(\asg)}^{V(\asg)}$ recursively as follows. For $\ell\in L(\asg)$, set $\la^c_{\ell}(\asg)=\ell$. For $v\in V(\asg)$ 
 \begin{itemize}
 	\item with $\deg(v)=1$ and child~$\chim{v}$, set $\la^c_v(\asg)=\la^c_{\chim{v}}(\asg)$. 
  \item with $\deg(v)=2$, set $\la^c_v(\asg)=\la^c_{\chir{v}}(\asg)$ if $\cpv{\chir{v}}=0$, otherwise set $\la^c_v(\asg)=\la^c_{\chil{v}}(\asg)$.
  \item with $\deg(v)=3$, , set $\la^c_v(\asg)=\la^c_{\chir{v}}(\asg)$ if $\cpv{\chim{v}}=\cpv{\chir{v}}=0$, otherwise set $\la^c_v(\asg)=\la^c_{\chil{v}}(\asg)$.
\end{itemize}
For any $v\in V(\alpha)$, we refer to the leaf $\lambda_v^c(\alpha)$ as the \emph{ancestral leaf} of $v$ under $c$, and to $J^c_v(\asg)\defeq c^{}_{\la^c_v(\asg)}$ as the \emph{ancestral type} of $v$ under $c$.
\end{definition}

Let $\asg\in\Xi$, $y_0\in[0,1]$, and $C(y_0)= (C_\ell(y_0))_{\ell\in L{(\asg)}}$ be a random leaf-type configuration consisting of iid Bernoulli random variables with success parameter~$y_0$. Define $H_\star(\asg,y_0)$ to be the probability that the ancestral type of the root of~$\asg$ is $1$ under $C(y_0)$.
\smallskip

Consider the eASG process $a=(a_r)_{r\geq 0}$ starting in $\rootsin$. For $r\geq 0$ and $y_0\in[0,1]$, let $C(r,y_0)$ be a random leaf-type configuration of $a_r$ consisting of iid Bernoulli  variables with parameter $y_0$, independent of $a_r$ ($y_0$ then represents the proportion of type-$1$ individuals at time $r$ before the present). For $y_0\in[0,1]$ and $r\geq 0$, we define
$$g_{r}(y_0)\defeq\Pb_{\rootsin} \big (J_{\rho(\brea_r)}^{C(r,y_0)}(\brea_r)=1 \big )=\E_{\rootsin}[H_\star(\brea_r,y_0)].$$
We refer to $g_r(y_0)$ as \emph{ancestral type distribution} at (backward) time $r$ with unfit frequency $y_0$. The \emph{long-term ancestral type distribution} at frequency $y_0\in[0,1]$ is defined
as $g_{\infty}(y_0)\defeq \lim_{r\to\infty} g_{r}(y_0)$; we show in Corollary~\ref{coro:ancestraltypedistribution} that this limit always exists.

\smallskip


It is more involved to find a useful representation for the ancestral type distribution than for the type distribution, because  the former  requires to identify the parental branch at every branching event. Moreover, some ancestral lines need to be traced back beyond the first mutation. In the non-interactive case (i.e. $\gamma=0$), this can be done via the pruned lookdown ancestral selection graph (pLD-ASG), which is a variant of the ASG in which the lines have an additional labelling that helps to prune the lines in an appropriate way (see~\citep{BCH17,Cordero2017590}; or~\citep{Lenz2015,Baake2018}). We now provide an alternative approach building on the eASG in the interactive case.
\smallskip

In the remainder of this section, we assume $\nu_0=0$. We call the \emph{immune line} of $\asg\in \Xi$ the path connecting the root of $\asg$ to its leftmost leaf (i.e. the path that is continuing to every branching event). Since~$\nu_0=0$, the ancestral type of the root is $1$ if and only if the leftmost leaf of $\asg$ is the ancestral leaf of the root and is unfit. Moreover, by construction, the leftmost leaf is the ancestral leaf of the root if and only if: (i) for any vertex $v$ with $\deg(v)=2$ on the immune line, $\chir{v}$ has the unfit type, and (ii) for any vertex $v$ with $\deg(v)=3$ on the immune line,  $\chim{v}$ or $\chir{v}$ has the unfit type. To check (i) we keep track, for any vertex $v$ with $\deg(v)=2$ on the immune line, of the eASG rooted in $\chir{v}$. Similarly, to check (ii) we keep track, for any vertex $v$ with $\deg(v)=3$ on the immune line, of the eASGs rooted in $\chim{v}$ and $\chir{v}$. The previous discussion leads to the following definition. 

\begin{definition}[eASG forest]\label{def:foreststratifiedASG}
	Let $\asg\in\Xi$. The \emph{eASG forest} of $\alpha$ is the collection 
	$$\Fs(\al)\defeq\Big( (\alpha_{\chir{v}_i})_{i=1}^{N(\alpha)},(\alpha_{\chim{w}_j},\alpha_{\chir{w}_j})_{j=1}^{M(\alpha)}\Big),$$ where
	\begin{enumerate}[label=(\arabic*)]
	\item $N(\alpha)$ and $M(\alpha)$ are the numbers of vertices on the immune line of $\alpha$ with outdegree $2$ and $3$, respectively,
		\item $(v_i)_{i=1}^{N(\alpha)}$ and $(w_j)_{j=1}^{M(\alpha)}$ are the corresponding vertices on the immune line, ordered in increasing distance to the root,
		\item for~$i\in[N(\alpha)]$, $\alpha_{\chir{v}_i}$ is the sub-eASG of $\alpha$ rooted in $\chir{v}_i$
		\item for~$j\in[M(\alpha)]$, $\alpha_{\chim{w}_j}$ and $\alpha_{\chir{w}_j}$ are the sub-eASGs of $\alpha$ rooted in $\chim{w}_j$ and $\chir{w}_j$.
	\end{enumerate}
See Fig. \ref{fig:forest} for an illustration.  	
\end{definition}
The next result formalises the connection between the ancestral type of the root of an eASG and the types of the roots of the trees in its forest. The proof of the result is provided in Section~\ref{sec:ancestraltypeinteractive}. \begin{figure}[t]\begin{minipage}{0.3\textwidth}
			\begin{center}
				\scalebox{.5}{\begin{tikzpicture}
				    \draw  plot [smooth cycle] coordinates {(5.6,0.75) (3.75, 0.75) (4.25,2) (4.55,3.25) (6.45,3.25) (6.85,2.15)};
					\draw  plot [smooth cycle] coordinates {(2.35,1.75) (2, 3.25) (3,3.25) (3.75,1.9)};
					
					\draw[opacity=0] (.5,0)--(2.5,4.5);
					\begin{scope}[every node/.style={regular polygon,regular polygon sides=7,draw}, very thick, minimum size=7pt]
					\node[opacity=0.2] (h0) at (4,0) {};
					
					\node (h00) at (4,1) {};
					\node[opacity=0.2] (h1) at (2.5,1) {};
					\node (h2) at (5.5,1) {};
					
					\node[opacity=0.2] (h3) at (1.5,2) {};
					\node[opacity=0.2] (h3m) at (1.5,3) {};
					\node[opacity=0.2] (h3mm) at (1.5,4) {};

					\node (h4) at (2.5,2) {};
					\node (h5) at (3.5,2) {};
					\node (h6) at (4.5,2) {};
					\node (h7) at (5.5,2) {};
					\node (h7r) at (6.5,2) {};

					\node (h31) at (2.25,3) {};

					\node (h33) at (2.75,3) {};
					
					\node (h9) at (4.75,3) {};
					\node (h10) at (5.5,3) {};
					\node (h11) at (6.25,3) {};
					
					\end{scope}

					\begin{scope}[>={latex[black]}, every node/.style={fill=white,circle},
					every edge/.style={draw=black,very thick}]
					
					\path [opacity=0.2] (h0) edge (h1);
					\path [opacity=0.2] (h0) edge (h00);
					\path [opacity=0.2] (h0) edge (h2);
					
					\path [opacity=0.2] (h1) edge (h3);
					\path [opacity=0.2] (h1) edge (h4);
					\path [opacity=0.2] (h1) edge (h5);
					
					\path [opacity=0.2] (h3) edge (h3m);
					\path [opacity=0.2] (h3m) edge (h3mm);

					\path [-] (h4) edge (h31);
					\path [-] (h4) edge (h33);

					\path [-] (h2) edge (h6);
					\path [-] (h2) edge (h7);
					\path [-] (h2) edge (h7r);
					
					\path [-] (h7) edge (h9);
					\path [-] (h7) edge (h10);
					\path [-] (h7) edge (h11);

					\end{scope}
					\begin{pgfonlayer}{background}
        \highlight{4.5mm}{cyan}{(4,0) -- (2.5,1) -- (1.5,2) -- (1.5,3) -- (1.5,4)}
        \end{pgfonlayer}
        \node[opacity=0.3] at (1.5,2) {\scalebox{1.5}{$\times$}} ;	
		\node[opacity=0.3] at (1.5,3) {\scalebox{1.5}{$\times$}} ;			
					\end{tikzpicture}}
		\end{center}	\end{minipage}\begin{minipage}{0.3\textwidth}
			\begin{center}
				\scalebox{.5}{\begin{tikzpicture}
				 \draw  plot [smooth cycle] coordinates {(3.5,1.75) (3,1.75) (3,3.25) (3.5,3.25) };
				\draw  plot [smooth cycle] coordinates {(2.75,2.75) (1.5,2.85) (2,4.25) (3,4.25) };
				\draw  plot [smooth cycle] coordinates {(5.6,0.75) (4.5,2) (5.5,3.25) (7.35,3.15)};
				
					\draw[opacity=0] (.5,0)--(2.5,4.5);
					\begin{scope}[every node/.style={regular polygon,very thick, regular polygon sides=7,draw}, minimum size=7pt]
					\node[opacity=0.2] (h0) at (4,0) {};
					
					\node[opacity=0.2] (h1) at (2.5,1) {};
					\node (h2) at (5.5,1) {};
					
					\node[opacity=0.2] (h3) at (1.75,2) {};
					\node  (h4) at (3.25,2) {};
					\node  (h5) at (3.25,3) {};
					\node (h6) at (4.75,2) {};
					\node (h7) at (6.25,2) {};

					\node[opacity=0.2] (h31) at (1,3) {};
					\node[opacity=0.2] (h311) at (1,4) {};
					\node  (h32) at (1.75,3) {};
					\node (h33) at (2.5,3) {};
					\node (h331) at (2.25,4) {};
					\node (h332) at (2.75,4) {};
					
					\node (h9) at (5.5,3) {};
					\node (h10) at (6.25,3) {};
					\node  (h11) at (7,3) {};
					
					\end{scope}

					\begin{scope}[>={latex[black]}, every node/.style={fill=white,circle},
					every edge/.style={draw=black,very thick}]
					
					\path [opacity=0.2] (h0) edge (h1);
					\path [opacity=0.2] (h0) edge (h2);
					
					\path [opacity=0.2] (h1) edge (h3);
					\path [opacity=0.2] (h1) edge (h4);
					
					\path[opacity=0.2] (h3) edge (h31);
					\path [opacity=0.2](h3) edge (h32);
					\path [opacity=0.2] (h3) edge (h33);
					
					\path [-] (h4) edge (h5);
					
					\path [-] (h2) edge (h6);
					\path [-] (h2) edge (h7);

					\path [-] (h7) edge (h9);
					\path [-] (h7) edge (h10);
					\path [-] (h7) edge (h11);
					
					\path [opacity=0.2] (h31) edge (h311);
					
					\path (h33) edge (h331);
					\path (h33) edge (h332);
					\end{scope}
					\node[opacity=0.3] at (1,3) {\scalebox{1.5}{$\times$}} ;	
					\node at (3.25,2) {\scalebox{1.5}{$\times$}} ;						
							\begin{pgfonlayer}{background}
        \highlight{4.5mm}{cyan}{((4,0) -- (2.5,1) -- (1.75,2)--(1,3)--(1,4)}
        \end{pgfonlayer}		
					
					\end{tikzpicture}}
		\end{center}	\end{minipage} \begin{minipage}{0.3\textwidth}
			\begin{center}
				\scalebox{.5}{\begin{tikzpicture}
			      \draw  plot [smooth cycle] coordinates {(2.25,2.75) (2,4.25) (4,4.25) (3.75,2.75)};
				  \draw  plot [smooth cycle] coordinates {(5.5,0.75) (4.25,2) (6.5,3.25) (6.75,2)};

					\draw[opacity=0] (.5,0)--(2.5,4.5);
					\begin{scope}[every node/.style={regular polygon,regular polygon sides=7,draw}, very thick, minimum size=7pt]
					
					\node[opacity=0.3] (h0) at (4,0) {};
					
					\node[opacity=0.3] (h1) at (2.5,1) {};
					\node (h2) at (5.5,1) {};

					\node[opacity=0.3] (h3) at (2.5,2) {};
					\node (h6) at (4.5,2) {};
					\node (h7) at (5.5,2) {};
					\node (h8) at (6.5,2) {};
					\node (h9) at (6.5,3) {};
					
					\node[opacity=0.3] (h31) at (1.5,3) {};
					\node (h32) at (2.5,3) {};
					\node (h33) at (3.5,3) {};
					
					\node (h323) at (2.75,4) {};
					\node (h322) at (2.25,4) {};
					
					\node (h333) at (3.75,4) {};
					\node (h332) at (3.25,4) {};
					\end{scope}
					
					\begin{scope}[>={latex[black]}, every node/.style={fill=white,circle},
					every edge/.style={draw=black,very thick}]
					
					\path [opacity=0.3] (h0) edge (h1);
					\path [opacity=0.3] (h0) edge (h2);
					
					\path [opacity=0.3] (h1) edge (h3);
					
					\path [opacity=0.3] (h3) edge (h31);
					\path [opacity=0.3] (h3) edge (h32);
					\path [opacity=0.3] (h3) edge (h33);
					
					\path [-] (h2) edge (h6);
					\path [-](h2) edge (h7);
					\path [-] (h2) edge (h8);
					\path [-] (h8) edge (h9);
					
					\path [-] (h32) edge (h322);
					\path [-] (h32) edge (h323);
					
					\path [-] (h33) edge (h332);
					\path [-] (h33) edge (h333);
					\end{scope}
					\node[opacity=0.3] at (2.5,1) {\scalebox{1.5}{$\times$}} ;	
					\node at (6.5,2) {\scalebox{1.5}{$\times$}} ;		\begin{pgfonlayer}{background}
        \highlight{4.5mm}{cyan}{(4,0) -- (2.5,1) -- (2.5,2)--(1.5,3)}
        \end{pgfonlayer}			
					\end{tikzpicture}}
		\end{center}	\end{minipage}
			
	\caption[Forest]{Three eASGs (black and grey), their immune lines (delimited in light blue), and the corresponding forests (black).}\label{fig:forest}
\end{figure}

\begin{proposition}\label{prop:prodoftrees}
	Let $\alpha\in\Xi$ without mark $\circ$ and $\Fs(\al)=( (\alpha_{\chir{v}_i})_{i=1}^{N(\alpha)},(\alpha_{\chim{w}_j},\alpha_{\chir{w}_j})_{j=1}^{M(\alpha)})$ be the corresponding eASG forest.
Then, for all $y_0\in[0,1]$, we have
	\begin{equation}\label{eq:equalityregioncodingancestralcoding}
		H_\star(\alpha,y_0)= y_0 \prod_{i=1}^{N(\alpha)} H(\alpha_{\chir{v}_i},y_0)\prod_{j=1}^{M(\alpha)}\left[H(\alpha_{\chim{w}_j},y_0)+H(\alpha_{\chir{w}_j},y_0)-H(\alpha_{\chim{w}_j},y_0)H(\alpha_{\chir{w}_j},y_0)\right].
	\end{equation}
\end{proposition}
Proposition~\ref{prop:prodoftrees} and the duality for the eASG process are crucial to derive the following representation of the ancestral type distribution.

\begin{theorem}[Representation of ancestral type distribution]\label{thm:ancestraldistributionfinite}
	Let~$\nu_0=0$.
	Then,
	\begin{equation}
	g_r(y_0)=y_0\exp\bigg(- \int_0^r \big(1-y(\xi;y_0)\big)\big(s+\gamma (1-y(\xi;y_0) )\big)  \dd \xi \bigg),\qquad y_0\in [0,1].\label{eq:ancestraltypedistribution}
	\end{equation}
Moreover, $g_{\infty}(y_0)= \lim_{r\to\infty}g_r(y_0)$ exists for all $y_0\in [0,1]$.
\end{theorem}
More explicit formulas for $g_r$ and $g_\infty$ are provided in Proposition~\ref{prop:ancestraldistributionfinite} and Corollary~\ref{coro:ancestraltypedistribution}, respectively. These corollaries and the proof of Theorem~\ref{thm:ancestraldistributionfinite} can be found in Section~\ref{sec:ancestraltypeinteractive}.
The following result is a particular case of Corollary~\ref{coro:ancestraltypedistribution} (recall the parameter regions~\eqref{eq:parameterregions}).

\begin{figure}[t]
	\centering\scalebox{.46}{
		\includegraphics{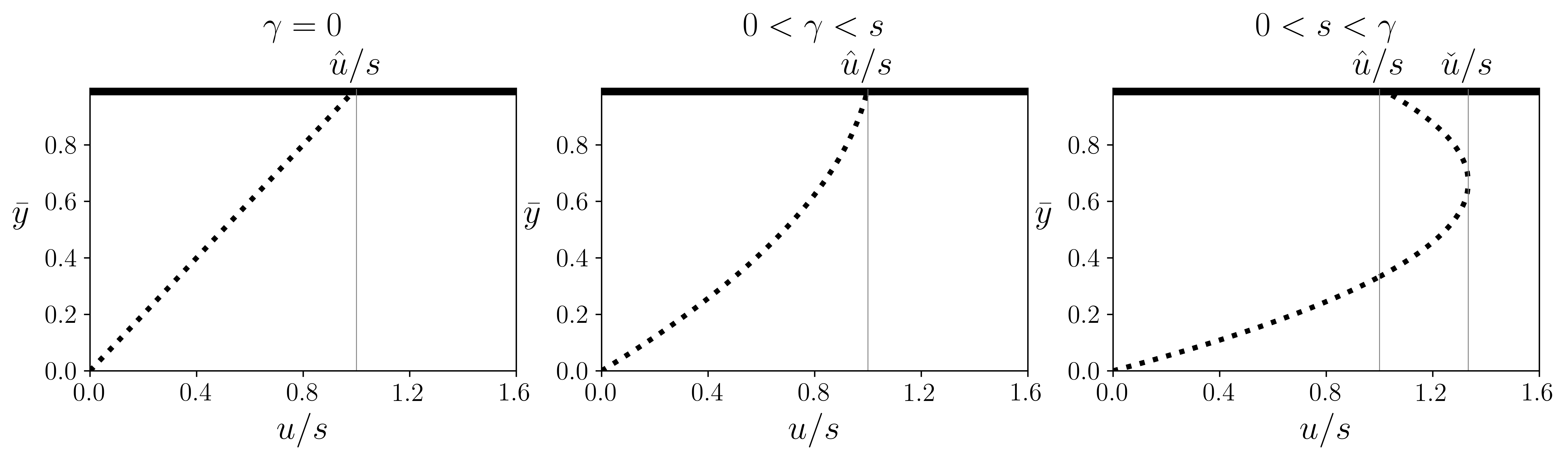}}
	\caption[Ancestral type distribution at equilibrium with pairwise interaction ($\nu_0=0$)]{Plot of the ancestral type distribution at equilibrium for the parameter regimes from Fig.~\ref{fig:plots_l_gamma}. This illustrates Corollary~\ref{coro:ancestraldistributionequilibrium}. Dotted line (resp. solid line): equilibria~$\bar{y}$ s.t. $g_{\infty}(\bar{y})=0$ (resp. $g_{\infty}(\bar{y})=1$).}
	\label{fig:plots_l_gamma_ancestral}
\end{figure} 
\begin{coro}[Long-term ancestral type distribution at equilibrium]\label{coro:ancestraldistributionequilibrium} Let~$\nu_0=0$. We have $g_{\infty}(y_{\infty}(1))=1$. Moreover, for~$y_0\in [0,1)$, in~$\Theta_1^{}$, we have $g_{\infty}(y_{\infty}(y_0))=1$; in~$\Theta_2^a$, we have $g_{\infty}(y_{\infty}(y_0))=0$; and in~$\Theta_2^b\cup\Theta_3^{}$, \begin{equation}
		g_{\infty}(y_{\infty}(y_0))=\begin{cases}
		0,&\text{if }y_0\leq \bar{y}_+,\\ 1,&\text{if }y_0> \bar{y}_+.
		\end{cases}
		\end{equation}
\end{coro}

Corollary~\ref{coro:ancestraldistributionequilibrium} is also proved in Section~\ref{sec:ancestraltypeinteractive}.

\smallskip

We conclude this section by describing the genealogical picture underlying Corollary~\ref{coro:ancestraldistributionequilibrium}. Recall that there is a natural coupling between the eASG, pASG process starting at $\rootsin$ such that $H(a_t,y_0)=H(\bar{a}_t,y_0)=\Hs(s(\bar{a_t}),y_0)$. In particular, the discussion at the end of Section~\ref{sec:mainresults:subsec:application} about the long-term output type in the sASG process translates into the analogous statements for the long-term type of the root of the eASG process. Assume that $y_0\in[0,1)$.

\begin{itemize}
	\item[$\Theta_1^{}$:] The discussion at the end of Section~\ref{sec:mainresults:subsec:application} implies~$y_{\infty}(y_0)\equiv 1$. In particular, all potential ancestral leaves in the eASGs are unfit.
	\item[$\Theta_2^a$:] The long-term type of the root of each eASG branching off the immune line is $0$ with probability $1-\ymin>0$ (this is the probability that the corresponding sASG is not absorbed). Since $s>0$, the number of eASGs branching off the immune line at selective events grows to infinity; hence, the long-term ancestral type of the root is $0$ almost surely.    
	\item[$\Theta_{2}^b\cup \Theta_3^{}$:] The long-term type of the root of each eASG branching off the immune line is $0$ with positive probability if $y_0\in\Attr(\ymin)$, and  as in case $\Theta_2^a$, the long-term ancestral type of the root is $0$ almost surely. In contrast, for $y_0\in\Attr(1)$, $y_\infty(y_0)=1$, and as in $\Theta_1^{}$, all potential ancestral leaves in the eASGs are unfit.
\end{itemize}

\subsection{Open questions}
Our tree-valued processes capture the genealogies of the model underlying~\eqref{eq:dlimitdiffeq} and allow the inference of type and ancestral-type distribution. There remain several open problems, some of which we list here.

\begin{enumerate}
	\item If $y_0\in [0,1]$ is not an unstable equilibrium located in $(0,1)$, Theorem~\ref{thm:representationequilibrium} states that $\Hs_{\infty}(y_0)$ has Bernoulli distribution. However, for unstable $y_0\in (0,1)$ the distribution of $\Hs_{\infty}(y_0)$ is unknown. \cite{MSS2018} provides some clues: the authors prove that, in their setting and for $\gamma>4$, there is an atom at~$1$~\cite[Lem.~1.19]{MSS2018} (to avoid confusion, we once more stress that they consider $1-\Hs_{\infty}(y_0)$ so that they observe the atom at~$0$), and they provide numerical evidence \cite[Fig.~2]{MSS2018} that suggests that beside this atom, the density is absolutely continuous with respect to the Lebesgue measure. Is it possible to provide more details about this distribution?

\item Suppose $\nu_0=0$. In $\Theta_1$, it follows from Proposition~\ref{prop:equilibriumpoints} and Theorem~\ref{thm:representationequilibrium} that the sASG process dies out almost surely. If we are not in $\Theta_1$ and the tree does not absorb, the total mass of the sASG process tends to $\infty$ by Proposition~\ref{prop:treegrowsinfty}. $\Hs_{\infty}(y_0)$ then depends  on $y_0$. This raises the question whether the behaviour can be explained explicitly in terms of the underlying tree structure. While this is relatively straightforward in $\Theta_2^a$ (on which we have not expanded here in order not to overload the article), it is more of a challenge in $\Theta_{2}^b\cup \Theta_3^{}$. 

\item We have derived a representation of the ancestral type distribution if $\nu_0=0$. What is the representation if $\nu_0>0$?

\item Pruning parts of the ASG upon mutations is a key step in our constructions. All our operations are specifically tailored to~\eqref{eq:dlimitdiffeq}; and in particular to the induced type propagation within the ASG. Replacing the selection and interaction terms in~\eqref{eq:dlimitdiffeq} by $y(t)(1-y(t)s(y(t))$ for a general polynomial $s$ yields a model with more general (as opposed to pairwise) interactions (see~\cite{Cordero2019} for a related stochastic model). Can our pruning operations be systemically extended to this more general setup?

\item Starting from the Moran model, time and parameters can be rescaled such that one obtains a Wright--Fisher diffusion with  frequency-dependent selection and mutation. In contrast to the deterministic limit, coalescences are still present in the corresponding ASG so that the limiting genealogical structure is a branching-coalescing graph. \cite{Cordero2019} gives a useful representation of the ASG in the case of selection only, and we suspect this carries over to the case with mutations. However, it is unclear whether the effect of mutations can be resolved on the spot as we do in Sections~\ref{sec:mainresults:subsec:pASG} and~~\ref{sec:mainresults:subsec:sASG}.
\end{enumerate}

\section{The mutation--selection equation with interaction and the law of large numbers}\label{sec:detmodelandlaw}
We now provide additional context for the mutation--selection equation and prove the results of Section~\ref{sec:mainresults:sub:detmodel} and~\ref{sec:mainresults:subsec:detModelvsMoranModel}. Moreover, we recover the equilibrium structure without interaction by taking the limit~$\gamma \to 0$ and  establish a sufficient condition for a unique equilibrium in $[0,1]$ if $\nu_0\in(0,1)$. We start by determining the equilibrium structure of~\eqref{eq:dlimitdiffeq} if $\nu_0=0$.
\begin{proof}[Proof of Proposition~\ref{prop:equilibriumpoints}]
 Since~$\nu_0=0$, the right-hand side of~\eqref{eq:dlimitdiffeq} reduces to $F(y)=(y-1)\,G(y)$ with~$G(y)=-\gamma y^2+(s+\gamma)y-u$. If $\gamma=0$, the equilibria follow immediately. If $\gamma>0$, we consider the discriminant $D=D(u,s,\gamma)\defeq(\gamma+s)^2-4u\gamma$ of the quadratic polynomial $G(y)$ as a function of~$u$. Note that $s \leq {u^\star}$, with equality if $\gamma=s$. Assume first that $u> {u^\star}$. Then $D$ is negative, and $1$ is the unique equilibrium. Next, assume that $u \leq {u^\star}$. Then, the equilibria are $1$ and the real roots of~$G$, which are  $\bar{y}_-$ and $\bar{y}_+$ as claimed. Since $D= (1-s/\gamma)^2+4(s-u)/\gamma \geq (1-s/\gamma)^2$ for $u \leq s$ (with equality if $u = s$), and $0 \leq D < (1-s/\gamma)^2$ for $s < u \leq u^\star$ (with equality if $u = u^\star$), the (in)equalities in Table~\ref{table:stability} follow easily. 
 
 \smallskip
 
As to the stabilities, note first that $F(0)=u>0$, and recall from Table~\ref{table:stability} that all zeroes of $F$ are positive. This entails that the smallest equilibrium that is a simple zero of $F$ is stable. If all equilibria are simple zeroes, the stabilities alternate. Now, consider first $\gamma=0$. If $u \neq s$, the zeroes are distinct, so the smaller (larger) equilibrium is stable (unstable). If $u=s$, the single equilibrium is a double zero at $1$. Since $F$ vanishes without changing sign at this point, the equilibrium is unstable  as an equilibrium in $\Rb$, but attracting from the left and thus stable in $[0,1]$.

\smallskip

Next, consider $\gamma>0$ and use the positions of the equilibria from Table~\ref{table:stability}. If $u<s$ or $s< u < u^\star$, all equilibria are distinct, so their stabilities alternate. If $u=s$ or $u=u^\star$, either two out of the three or all three equilibria coincide. If we have one simple and one double zero, then the equilibrium at the simple zero is always stable (because this is the smallest simple zero), while the equilibrium at the double zero is always unstable when considered in $\Rb$, or when in $(0,1)$. This also remains true in the case $\bar y_- < 1=\bar y_+$ when considered in $[0,1]$, since the equilibrium at $1$ is repelling to the left.
If the equilibrium is a triple zero,  it is stable since $F$ changes sign.
\end{proof}

If $\nu_0=0$, the two non-trivial roots of~$F$ are functions of~$s,\gamma$ and $u$, i.e. $\bar{y}_{\pm}=\bar{y}_{\pm}(s,\gamma,u)$; the equilibrium attained additionally depends on~$y_0$. To stress this, we temporarily write $y_{\infty}(y_0;s,\gamma,u)$ instead of $y_{\infty}(y_0)$. Moreover, $\bar{y}_-$ and $\bar{y}_+$ are continuous in $\gamma>0$. Proposition~\ref{prop:equilibriumpoints} yields that for $\gamma<s$, $0<\bar y_- < 1 < \bar{y}_+$. Hence, $1$ and $\bar{y}_-$ are the biologically relevant equilibria in this parameter regime. The next result establishes the continuity of $\bar{y}_-$ in $\gamma$ at $0$, and yields another way to recover the equilibrium structure in the non-interactive case.
\begin{proposition}
	For~$s,u>0$ and~$\nu_0=0$, $\lim_{\gamma \to 0}\bar{y}_-(s,\gamma,u) =  u/s.$
	In particular, $$\lim_{\gamma\to 0} y_{\infty}^{}(y_0,s,\gamma,u)=\begin{cases} \min\{\frac{u}{s},1\}, &\text{if }y_0\in [0,1),\\
	1,&\text{if }y_0=1.
	\end{cases}$$
\end{proposition}
\begin{proof}
 It follows from the definition of $\bar{y}_-$ and a straightforward application of L'H\^{o}pital's rule.
\end{proof}

We close this section by proving that the type-frequency process in the Moran model converges to the solution of the mutation--selection equation as the population size tends to $\infty$.
\begin{proof}[Proof of Proposition~\ref{sec:mainresults:prop:lln} (dynamical law of large numbers)]
	$(Y^{(N)})_{N\geq 1}$ is a \emph{density-dependent family} of Markov chains, because we can rewrite the rates of $Y^{(N)}$ as $q_{Y^{(N)}}^{}(k,k+\ell)=Nq\left(\frac{k}{N},\ell\right)$ for $\ell\in \Zb \setminus \{0\}$, where~$q:[0,1]\times \Zb\setminus \{0\}\to \R$ is continuous and given by $$q(y,1)=y(1-y)+(1-y)u\nu_0,\quad q(y,-1)=y(1-y)\big(1+s+\gamma(1-y)\big)+yu\nu_1,$$ together with $q(y,\ell)=0$ for $\vert \ell\vert>1$. Thus, in order to conclude, we only need to verify the following conditions of the dynamical law of large number for density-dependent families of Markov chains by~\citet[Thm.~3.1]{kurtz1970}, \begin{equation}\label{cond1}
	\sup_{y\in[0,1]}\sum_{\ell}|\ell|q(y,\ell)<\infty\quad\textrm{and}\quad \lim_{d\rightarrow\infty}\sup_{y\in[0,1]}\sum_{|\ell|>d}|\ell|q(y,\ell)=0.
	\end{equation}
	These conditions are clearly satisfied. Since the IVP~\eqref{eq:dlimitdiffeq} has a unique solution~$y(\,\cdot\,;y_0)$ from~$[0,\infty)$ to~$[0,1]$, the proof is completed.
\end{proof}
\begin{remark}
	In the absence of interaction ($\gamma=0$), Proposition~\ref{sec:mainresults:prop:lln} coincides with \citep[Prop.~3.1]{2015CorderoULT}. 
\end{remark}

\section{The ASG: dualities, pruning, and stratification}\label{sec:proofs:asg}
\subsection{The ASG in the Moran model with interaction}\label{sec:ASGMoran}The ASG is naturally embedded in the graphical representation of the Moran model. We have outlined the main idea in Section~\ref{sec:mainresults:subsec:mm}: consider a graphical realisation of the Moran model in~$[0,t]$ and, starting from time~$t$, trace back the individuals that potentially influence the type while ignoring the additional information contained in mutations. Here, we  provide the details.

\smallskip

Assume there are currently~$n$ lines in the ASG, i.e. there are~$n$ individuals that potentially influence the type of the sample. When a neutral arrow joins two lines in the current set, a \textit{coalescence event} takes place, i.\,e.\,\,the two lines merge into the single one at the tail of the arrow and the number of lines in the graph decreases by one (see Fig.~\ref{fig:coalescencecollision} (ii)). Since neutral arrows appear at rate~$1/N$ per ordered pair of lines, coalescence events occur at rate~$n(n-1)/N$ in our ASG of size~$n$. When a line in the current set is hit by a neutral arrow that emanates from a line that is currently not in the graph, a \textit{relocation event} occurs; i.\,e.\,\,the ASG continues with the incoming branch (the line at the tail of the arrow) and the number of lines in the graph does not change. Relocation events occur at rate~$n(N-n)/N$.

\smallskip

When a selective arrow hits the current set of lines, the hit individual has two potential parents, namely the individual at the incoming branch, and the one at the continuing branch. Which of these is the true parent of the individual at the descendant branch depends on the type at the incoming branch, but for the moment we work without types. This means that we must trace back both potential parents; we say the selective event remains \emph{unresolved}. Such events can be of two types: a \textit{binary branching} if the selective arrow emanates from a line outside the current set of lines, and a \textit{collision} if the selective arrow links two lines in the graph (see Fig.~\ref{fig:coalescencecollision} (iii)). The former increases the number of lines in the graph by one and, since selective arrows appear at rate $s/N$ per ordered pair of lines, binary branchings occur at rate~$sn(N-n)/N=\Os(1)$ as $N \to \infty$ in our ASG of size~$n$. The latter does not change the number of lines in the ASG and occurs at rate~$sn(n-1)/N=\Os(1/N)$.

\smallskip

When an interactive arrow hits a line in the ASG, the individuals at both the incoming and the checking branches are potentially influential for the types in the sample. The true parent depends on the types of both of them; but as before, we work without types. The resulting additional unresolved reproduction events can now be of various kinds, see Fig.~\ref{fig:coalescencecollision} (iv) and (v). A \textit{ternary branching} occurs if both the incoming and the checking arrows emanate from different lines currently not in the ASG. This increases the number of lines by two. If the incoming and the checking branch are identical, the event reduces to a binary branching or a collision, depending on whether the incoming/checking branch emerges from a line inside or outside the graph. It may also happen that the incoming and the checking branch are distinct, but one or both of them emanate(s) from a line within the current set of lines. Since interactive and checking arrow pairs occur at rate $\gamma /N^2$ per ordered triple of lines, in our ASG of size~$n$, ternary branchings occur at rate~$\gamma n (N-n)(N-n-1)/N^2 = \Os(1)$ as $N \to \infty$; all other kinds of events happen at rates of order $\Os(1/N)$.

\smallskip

In contrast to the original ASG (that is, without interaction), not all individuals that potentially influence the type of the sample are necessarily potential ancestors. Namely, the individual on the checking line is, in general, not ancestral; but its type may have an influence on the type of the sampled individual(s), and on which line is parental. Analogous to the Moran model, beneficial and deleterious mutations are superimposed on the lines at rate~$u \nu_0$ and~$u \nu_1$, respectively. The resulting object is called the \emph{untyped~ASG}; this refers to the fact that the initial types have not yet been assigned and the consequences of mutations are still unresolved. \begin{figure}[t]
	\begin{minipage}{0.19 \textwidth}
			\centering
			\scalebox{1}{
				\begin{tikzpicture}
					\draw[line width=0.1mm,opacity=0] (-0.1,-0.1) -- (2,1.6);
					\draw[line width=0.1mm,opacity=0.4] (0,0) -- (2,0);
					\draw[line width=0.1mm,opacity=0.4] (0,0.7) -- (2,0.7);
					\draw[line width=0.5mm] (1,0.7) -- (2,0.7);
					\draw[line width=0.5mm] (0,0) -- (1,0);
					\draw[-{triangle 45[scale=2.5]},color=black, line width=0.5mm] (1,0) -- (1,0.7) node[text=black, pos=.6, xshift=7pt]{};
			\end{tikzpicture}}\\ (i)
		\end{minipage}\begin{minipage}{0.19 \textwidth}
		\centering
		\scalebox{1}{
			\begin{tikzpicture}
				\draw[line width=0.1mm,opacity=0] (-0.1,-0.1) -- (2,1.6);
				\draw[line width=0.5mm,opacity=1] (0,0) -- (2,0);
				\draw[line width=0.5mm,opacity=1] (1,0.7) -- (2,0.7);
				\draw[line width=0.1mm] (0,0.7) -- (1,0.7);
				\draw[line width=0.5mm] (0,0) -- (1,0);
				\draw[-{triangle 45[scale=2.5]},color=black, line width=0.5mm] (1,0) -- (1,0.7) node[text=black, pos=.6, xshift=7pt]{};
		\end{tikzpicture}}\\ (ii)
	\end{minipage}\begin{minipage}{0.19 \textwidth}
			\centering
			\scalebox{1}{
				\begin{tikzpicture}
					\draw[line width=0.1mm,opacity=0] (-0.1,-0.1) -- (2,1.6);
					\draw[line width=0.5mm] (0,0.7) -- (2,0.7);
					\draw[line width=0.5mm] (0,0) -- (2,0);
					\draw[line width=0.1mm,opacity=0.4] (0,0) -- (2,0);
					\draw[-{open triangle 45[scale=2.5]},color=black, line width=0.5mm] (1,0) -- (1,0.7) node[text=black, pos=.6, xshift=7pt]{};
			\end{tikzpicture}} \\ (iii)
	\end{minipage}\begin{minipage}{0.19 \textwidth}
		\centering
		\scalebox{1}{
			\begin{tikzpicture}
				\draw[line width=0.1mm,opacity=0] (-0.1,-0.1) -- (2,1.6);
				\draw[line width=0.5mm] (0,1.4) -- (2,1.4);
				\draw[line width=0.5mm] (0,0.7) -- (2,0.7);
				\draw[line width=0.5mm] (0,0) -- (1,0);
				
				\draw[line width=0.1mm,opacity=0.4] (0,0) -- (2,0);
				\draw[line width=0.1mm,opacity=0.4] (0,0.7) -- (2,0.7);
				\draw[line width=0.1mm,opacity=0.4] (0,1.4) -- (2,1.4);
				
				\draw (1,.7) circle (.6mm)  [fill=black!100];
				\fill[white, draw=black] (0.9,1.5) rectangle (1.1,1.3); 
				\draw[line width=.5mm] (1,0) .. controls (1,.2) and (1.15,.5) .. (1.15,.7);
				\draw[line width=.5mm] (1.15,.7) .. controls (1.15,.9) and (1,1.05) .. (1,1.35);
				
				\draw[-{Stealth[length=2mm,width=2mm,open]},opacity=1,line width=.5mm] (1,.7) -- (1,1.5);
				\filldraw[white, draw=black] (1,0.1) -- (1.1,0) -- (1,-0.1) -- (.9,0) -- (1,0.1);
		\end{tikzpicture}}\\(iv)
	\end{minipage}\begin{minipage}{0.19 \textwidth}
		\centering
		\scalebox{1}{
			\begin{tikzpicture}
				\draw[line width=0.1mm,opacity=0] (-0.1,-0.1) -- (2,1.6);
				\draw[line width=0.5mm] (0,1.4) -- (2,1.4);
				\draw[line width=0.5mm] (0,0.7) -- (2,0.7);
				\draw[line width=0.5mm] (0,0) -- (2,0);
				
				\draw[line width=0.1mm,opacity=0.4] (0,0) -- (2,0);
				\draw[line width=0.1mm,opacity=0.4] (0,0.7) -- (2,0.7);
				\draw[line width=0.1mm,opacity=0.4] (0,1.4) -- (2,1.4);

				\draw (1,.7) circle (.6mm)  [fill=black!100];
				\fill[white, draw=black] (0.9,1.5) rectangle (1.1,1.3); 
				\draw[line width=.5mm] (1,0) .. controls (1,.2) and (1.15,.5) .. (1.15,.7);
				\draw[line width=.5mm] (1.15,.7) .. controls (1.15,.9) and (1,1.05) .. (1,1.35);
				
				\draw[-{Stealth[length=2mm,width=2mm,open]},opacity=1,line width=.5mm] (1,.7) -- (1,1.5);
				\filldraw[white, draw=black] (1,0.1) -- (1.1,0) -- (1,-0.1) -- (.9,0) -- (1,0.1);
		\end{tikzpicture}}\\(v)
	\end{minipage}
	\hfill
	\caption[Relocation, coalescence, and collisions]{Relocation (i), coalescence (ii), and various types of collision events (iii)--(v).}
	\label{fig:coalescencecollision}
      \end{figure}

      \smallskip

Once the untyped ASG has been constructed, the true ancestry and the types of the initial sample are obtained as outlined in Section~\ref{sec:mainresults:subsec:mm}; that is, assign types to the lines in the ASG and propagate them up to the sample while abiding the propagation rules of the Moran model.

\smallskip

Coalescence, collisions, and interactive events that are not ternary branchings vanish as~$N\to\infty$ (as they are~$\mathcal{O}(1/N)$ per ordered pair of lines). This is why they are absent in asymptotic ASG, and why the branching rates are simply $s$ and $\gamma$.
\subsection{The eASG: results related to Section~\ref{sec:mainresults:eASG}}\label{sec:proofseASG}
We now prove the duality relation between the mutation--selection equation and the eASG process of Theorem~\ref{sec:mainresults:thm:dualityASG}. 

\smallskip

Recall the notation of Section~\ref{sec:mainresults:eASG}. In particular, $\rho(\alpha)$ denotes the root of $\alpha$. Moreover, for~$\alpha\in \Xi$ and $v\in V(\alpha)$, let $\alpha_v\in\Xi$ be the subtree of $\alpha$ rooted in $v$. If $\deg(\rho(\alpha))=2$ (if $\deg(\rho(\alpha))=3$), we write $\chil{\asg}$ and $\chir{\asg}$ ($\chil{\asg}$, $\chim{\asg}$ and $\chir{\asg}$) for the  subtree of $\alpha$ rooted in the left and right (left, middle and right) child of $\rho(\alpha)$.

\smallskip 

The following lemma collects elementary properties of the function~$H$ that turn out to be useful in the subsequent proofs.
\begin{lem}[Properties of~$H$] \label{Hep} Let~$\asg\in \Xi$ and set $d=\deg(\rho(\alpha))$.
	\begin{enumerate}[label=(\arabic*)]
		\item If $d=0$, then $H(\asg,y_0)=y_0$.
		\item If $d=1$ and $\rho(\alpha)$ has mark~$\times$, then~$H(\asg,y_0)=1$.
		\item If $d=1$ and $\rho(\alpha)$ has mark~$\circ$, then~$H(\asg,y_0)=0$.
		\item If $d=2$, then
		$H(\asg,y_0)=H(\chil{\asg},y_0)H(\chir{\asg},y_0).$
		\item If $d=3$, then $H(\asg,y_0)=H(\chil{\asg},y_0)\left[H(\chim{\asg},y_0)+H(\chir{\asg},y_0)-H(\chim{\asg},y_0)H(\chir{\asg},y_0)\right].$
	\end{enumerate}
	In particular, $H(\asg,\cdot)$ is a polynomial of degree at most~$\lvert L(\asg)\rvert$.
\end{lem}
\begin{remark}\label{rem:sratifiedASGrecursicetree}
	In the setup of~\citep{MSS2018} (which we recalled in Remark~\ref{rem:cooperativebranching}), deleterious mutations and ternary branchings are captured by the local maps $\mathtt{dth}$ (`deaths') and $\mathtt{cob}$ (`cooperative branchings'), respectively. $H(a_t,y_0)$ corresponds to the concatenation of the higher-level maps~$\widehat{\mathtt{dth}}$ and $\widehat{\mathtt{cob}}$, respectively. In particular, (2) and~(5) of our Lemma~\ref{Hep} coincide with \citep[Eq.~$(1.84)$]{MSS2018}.
\end{remark}
\begin{proof}[Proof of Lemma~\ref{Hep}]
	If $d=0$, the tree consists only of $\rho(\alpha)$, and hence the probability that $\rho(\alpha)$ is unfit equals $y_0$, which proves $(1)$. If~$d=1$, $\rho(\alpha)$ has a mark, and its type is determined by the mutation and independent of~$y_0$. If the mutation is deleterious (beneficial), the type of $\rho(\alpha)$ is~$1$ (is~$0$), in agreement with~$(2)$ and~$(3)$. If~$d=2$, $\rho(\alpha)$ is unfit if and only if both its children are unfit. If $d=3$, $\rho(\alpha)$ is unfit if and only if its left child is unfit and either the middle or right child is unfit. In both cases this is because of the type propagation rule. (4) and (5) then follow because the types of these children are independent due to the independent assignments of types at the leaves of $\chil{\asg}$ and $\chir{\asg}$ ($\chil{\asg}$, $\chim{\asg}$, and $\chir{\asg}$). That~$H(\asg,\cdot)$ is a polynomial of degree at most~$\lvert L(\asg)\rvert$ follows via a straightforward inductive argument.
\end{proof}
The infinitesimal generator of the eASG process of Definition~\ref{def:eASGprocess} is given by $\tilde{\Gs} f= \tilde{\Gs}_{\nwedge}f+\tilde{\Gs}_{\pitch}f+\tilde{\Gs}_{\times}f+\tilde{\Gs}_{\circ}f,$ where \begin{align*}
	&\tilde{\Gs}_{\nwedge} f(\asg)\defeq\sum\limits_{\ell\in L(\asg)} s [f(\asg^{\nwedge}_\ell)-f(\asg)],\qquad  & \tilde{\Gs}_{\pitch}f(\asg)\defeq\sum_{\ell\in L(\asg)} \gamma [f(\asg^{\pitch}_\ell)-f(\asg)],\\
	&\tilde{\Gs}_{\times}f(\asg)\defeq\sum\limits_{\ell\in L(\asg)} u\nu_1 [f(\asg^{\times}_\ell)-f(\asg)],\qquad & \tilde{\Gs}_{\circ}f(\asg)\defeq\sum_{\ell\in L(\asg)} u\nu_0 [f(\asg^{\circ}_\ell)-f(\asg)],\end{align*} for a bounded function $f$ from $\Xi$ to~$\R$ ($\Xi$ is equipped with the discrete topology).
We require the following lemma to prove the duality relation of Theorem~\ref{sec:mainresults:thm:dualityASG}.
\begin{lem}\label{lem:generatorequality}
	For $\asg\in \Xi$ and $y_0\in [0,1]$, we have \begin{align}
		\tilde{\Gs}_{\nwedge}H(\cdot,y_0)(\asg)\, &=\, -sy_0(1-y_0)\, \partial_2H(\alpha,y_0),\label{eq:g1}\\
		\tilde{\Gs}_{\pitch}H(\cdot,y_0)(\asg)\, &=\, -\gamma y_0(1-y_0)^2\, \partial_2H(\alpha,y_0),\label{eq:g2}\\
		\tilde{\Gs}_{\times}H(\cdot,y_0)(\asg)\, &=\, u\nu_1(1-y_0)\, \partial_2H(\alpha,y_0),\label{eq:g3}\\
		\tilde{\Gs}_{\circ}H(\cdot,y_0)(\asg)\, &=\,u\nu_0y_0\, \partial_2H(\alpha,y_0),\label{eq:g4} \end{align}	
	where $\partial_2H(\alpha,y_0)$ is the partial derivative of~$H$ with respect to the second coordinate evaluated at $(\alpha,y_0)$.
\end{lem}
\begin{proof}
	We proceed by induction on the number of vertices in $\alpha$. Assume first that $|V(\alpha)|=1$. It follows from Lemma~\ref{Hep} that $H(\asg,y_0)=y_0$, and hence $\partial_2H(\alpha,y_0)=1$. Identities \eqref{eq:g1}--\eqref{eq:g4} follow for~$\asg$ from the properties of~$H$ stated in Lemma~\ref{Hep}. Now assume \eqref{eq:g1}--\eqref{eq:g4} hold for any $\asg\in \Xi$ with $\lvert V(\asg) \rvert=n$ for some $n\in \N$. We now prove that they also hold for~$\asg\in \Xi$ with $\lvert V(\asg) \rvert=n+1$. We distinguish cases according to the outdegree of~$\varrho \defeq\rho(\alpha)$. If~$\deg(\varrho)=1$, then $\varrho$ has a mutation mark and its type is determined by the mark, see Lemma~\ref{Hep} (1) and (2). Therefore, $H(\asg,\cdot)$ is constant, and the right-hand side in \eqref{eq:g1}--\eqref{eq:g4} is~$0$. The left-hand side in \eqref{eq:g1}--\eqref{eq:g4} is also~$0$, because $H(\asg^\star_\ell,y_0)=H(\asg,y_0)$ for all~$\ell\in L(\asg)$, $y_0\in [0,1]$, and $\star\in\{\nwedge,\pitch,\times,\circ\}$. Assume next that~$\deg(\varrho)=2$. The root of $\asg_\ell^\nwedge$ is still $\varrho$ with outdegree~$2$. By construction, if $\ell\in L(\chil{\asg})$ (if $\ell\in L(\chir{\asg})$), then $(\asg_\ell^{\nwedge})_{\chil{\varrho}}=(\chil{\asg})_\ell^\nwedge$  and $(\asg_\ell^{\nwedge})_{\chir{\varrho}}=\chir{\asg}$ (then $(\asg_\ell^{\nwedge})_{\chil{\varrho}}=\chil{\asg}$ and $(\asg_\ell^{\nwedge})_{\chir{\varrho}}=(\chir{\asg})_\ell^\nwedge$). As a consequence, \begin{align*}
		\tilde{\Gs}_\nwedge H(\cdot,y_0)(\asg)&=H(\chir{\asg},y_0)\sum_{\ell \in L(\chil{\asg})} s[H\big((\chil{\asg})^\nwedge_\ell,y_0\big)-H(\chil{\asg},y_0)]+H(\chil{\asg},y_0)\sum_{\ell \in L(\chir{\asg})}  s[H\big((\chir{\asg})^\nwedge_\ell,y_0\big)- H(\chir{\asg},y_0)]\\
		&=H(\chir{\asg},y_0) \left[-sy_0(1-y_0)\, \partial_2 H(\chil{\asg},y_0) \right]+H(\chil{\asg},y_0)\left[-sy_0(1-y_0)\, \partial_2 H(\chir{\asg},y_0)\right]\\
		&=-sy_0(1-y_0)\, \partial_2 H(\asg,y_0),
	\end{align*}
	where we used the induction hypothesis and Lemma~\ref{Hep}~(4). This proves~\eqref{eq:g1}. Analogous arguments lead to~\eqref{eq:g2}--\eqref{eq:g4}. Assume now that~$\deg(\varrho)=3$. If $\ell\in L(\chil{\asg})$ (if $\ell\in L(\chim{\asg})$; or $\ell\in L(\chir{\asg})$), then $(\asg_\ell^{\nwedge})_{\chil{\varrho}}=(\chil{\asg})_\ell^\nwedge$, $(\asg_\ell^{\nwedge})_{\chim{\varrho}}=\chim{\asg}$, and $(\asg_\ell^{\nwedge})_{\chir{\varrho}}=\chir{\asg}$ (then $(\asg_\ell^{\nwedge})_{\chim
		{\varrho}}=(\chim{\asg})_\ell^\nwedge$, $(\asg_\ell^{\nwedge})_{\chil{\varrho}}=\chil{\asg}$, and $(\asg_\ell^{\nwedge})_{\chir{\varrho}}=\chir{\asg}$; or $(\asg_\ell^{\nwedge})_{\chir{\varrho}}=(\chir{\asg})_\ell^\nwedge$, $(\asg_\ell^{\nwedge})_{\chim{\varrho}}=\chim{\asg}$, and $(\asg_\ell^{\nwedge})_{\chil{\varrho}}=\chil{\asg}$), where $\varrho$ is, by construction, also the root of~$\asg^\ell_\nwedge$. Using Lemma~\ref{Hep}~(5) and the induction hypothesis,
	\begin{align*}
		\tilde{\Gs}_\nwedge H(\cdot,y_0)(\asg)&=[H(\chim{\asg},y_0)+H(\chir{\asg},y_0)-H(\chim{\asg},y_0)H(\chir{\asg},y_0)]\sum_{\ell \in L(\chil{\asg})}  s\big[H\big((\chil{\asg})^\nwedge_\ell,y_0\big)-H(\chil{\asg},y_0)\big]\\
		&\quad +H(\chil{\asg},y_0)\big(1-H(\chir{\asg},y_0)\big) \sum_{\ell \in L(\chim{\asg})} s\big[H\big((\chim{\asg})^\nwedge_\ell,y_0\big)-H(\chim{\asg},y_0)\big]\\
		&\quad +H(\chil{\asg},y_0)\big(1-H(\chim{\asg},y_0)\big) \sum_{\ell \in L(\chir{\asg})}  s\big[H\big((\chir{\asg})^\nwedge_\ell,y_0\big)-H(\chir{\asg},y_0)\big]\\
		&=-sy_0(1-y_0) \Big\{ [H(\chim{\asg},y_0)+H(\chir{\asg},y_0)-H(\chim{\asg},y_0)H(\chir{\asg},y_0)]\partial_2 H(\chil{\asg},y_0) \\
		&\quad +H(\chil{\asg},y_0)\big(1-H(\chir{\asg},y_0)\big) \, \partial_2 H(\chim{\asg},y_0)+H(\chil{\asg},y_0)\big(1-H(\chim{\asg},y_0)\big) \, \partial_2 H(\chir{\asg},y_0)\Big\}\\
		&=-sy_0(1-y_0)\, \partial_2 H(\asg,y_0).
	\end{align*}
	Analogous arguments apply to $\tilde{\Gs}_\pitch H(\cdot,y_0)(\asg)$, $\tilde{\Gs}_\times H(\cdot,y_0)(\asg)$, and $\tilde{\Gs}_\circ H(\cdot,y_0)(\asg)$.
\end{proof}
Lemma~\ref{lem:generatorequality} is essential in the following proof of the duality between the eASG process and the mutation--selection equation.
\begin{proof}[Proof of Theorem~\ref{sec:mainresults:thm:dualityASG} (duality eASG)]
	Consider~$y\defeq (y(t;y_0))_{t\geq 0}$ as a (deterministic) Markov process on~$[0,1]$ with generator~$\Gs_F$ given by \begin{equation}
		\Gs_Fg(y_0)=F(y_0)\,\frac{\dd  g}{\dd  y}(y_0)\label{eq:generatordetlim}\end{equation}
	for~$g\in \Cs^1([0,1],\R)$ and $F$ of~\eqref{eq:dlimitdiffeq}. Fix~$\asg\in \Xi$ and $t\geq 0$. Clearly, $H(\asg,\cdot)\in \Cs^1([0,1],\R)$ ($H(\asg,\cdot)$ is a polynomial, see Lemma~\ref{Hep}). Since~$F$ is continuously differentiable, it follows from a classical result of ODE theory~\citep[Thm.~8.43]{kelley2004} that $y(t;\cdot)\in \Cs^1([0,1],\R)$. Hence, also $P_t^FH(\asg,\cdot)=H(\asg,y(t;\cdot))\in \Cs^1([0,1],\R)$, where $(P^{F}_{t})_{t\geq 0}$ is the transition semigroup corresponding to $y$ acting on $\Cs([0,1])$ (equipped with the uniform norm). The set $\Xi$ is countable and equipped with the discrete topology. The number of possible transitions of the eASG process at any given state is finite and each transition occurs at a finite rate. Therefore, the domain of its generator contains any bounded function from $\Xi$ to~$\R$. In particular, for $y_0\in [0,1]$, $H(\cdot,y_0)$ and $\tilde{P}_t H(\cdot,y_0)$ lie in the domain of its generator, where $(\tilde{P}_{t})_{{t}\geq 0}$ is the transition semigroup corresponding to the eASG process acting on the space of bounded, Borel measurable functions (equipped with the uniform norm). Using Lemma~\ref{lem:generatorequality}, we deduce that
	$\tilde{\Gs}H(\cdot,y_0)(\asg)=\mathcal{G}_FH(\asg,\cdot)(y_0)$ for $\asg\in \Xi$ and $y_0\in [0,1].$
	Since~$H$ is bounded and continuous, the result follows from~\citep[Prop.~1.2]{jansen2014}.
\end{proof}
\subsection{The pASG: results related to Section~\ref{sec:mainresults:subsec:pASG}}\label{sec:proofspASG}
The main goal of this section is to prove Corollary~\ref{dualpASG}. 
We start with an elementary, but important property of the type propagation. For this we need to introduce the notion of \emph{lopping}. We say that $\alpha_0\in\Xi$ (resp. $\bal_0\in\Xi^\pA$) is a lopping of $\alpha\in\Xi$ (resp. $\bal\in\Xi^\pA$), if $\alpha_0$ is obtained from~$\alpha$ (resp. $\bal)$ by removing the descendants of a subset of vertices in $V(\alpha)$ (resp.~$V(\bal)$). See Fig. \ref{fig:lopping} for an illustration. Note that the notion of lopping and admissible pruning are different: a lopping of $\alpha\in\Xi$ is also in $\Xi$, while an admissible pruning of $\alpha$ is in $\Xi^\pA$. 
\begin{figure}[t]
	\begin{center}
		\begin{minipage}[b]{.4\linewidth}
			\centering\scalebox{.6}{\begin{tikzpicture}
				\begin{scope}[every node/.style={regular polygon,regular polygon sides=7, very thick,draw}, minimum size=10pt]
				\node (h1) at (0,1.5) {};
				\node (bl2) at (0,2.5) {};
				\node (l2) at (0,3.5) {};
				\node (h2) at (-2.5,2.5) {};
				\node (l3)at (2.5,2.5) {};
				
				\node (l4) at (-3.1,3.5) {};
		        \node (l41) at (-3.1,4.5){};	
				
				\node (l6)at (-1.9,3.5) {};
				\node[opacity=0.2] (l7)at (-0.8,4.5) {};
				\node[opacity=0.2] (l8)at (0,4.5) {};
				\node[opacity=0.2] (l9)at (0.8,4.5) {};
				\node[opacity=0.2](l10)at (1.7,3.5) {};
				\node[opacity=0.2] (l11)at (2.5,3.5) {};
				\node[opacity=0.2] (l12)at (3.3,3.5) {};
				
				\node[opacity=0.2] (l13)at (3.3,4.5) {};
				\node[opacity=0.2] (l14)at (2.5,4.5) {};
				\node[opacity=0.2] (l15)at (4.1,4.5) {};
				\end{scope}
				\node[] at (0,2.5) {\scalebox{1.5}{$\circ$}};
				\node[] at (-3.1,3.5) {\scalebox{1.5}{$\times$}};
				\begin{scope}[>={latex[black]}, every node/.style={regular polygon,regular polygon sides=7,draw},
				every edge/.style={draw=black,very thick}]
				\path [-] (h1) edge (h2);
				\path [-] (h1) edge (bl2);
				\path [-] (l2) edge (bl2);
				\path [-] (l4) edge (l41);
				\path [-] (h1) edge (l3);
				\path [-] (h2) edge (l4);
				\path [-] (h2) edge (l6);
				
				\path [opacity=0.2] (l2) edge (l7);
				\path [opacity=0.2] (l2) edge (l8);
				\path [opacity=0.2] (l2) edge (l9);
				
				\path [opacity=0.3] (l3) edge (l10);
				\path [opacity=0.3] (l3) edge (l11);
				\path [opacity=0.3](l3) edge (l12);
				
				\path [opacity=0.3] (l12) edge (l13);
				\path [opacity=0.3] (l12) edge (l14);
				\path [opacity=0.3] (l12) edge (l15);
				\end{scope}
				\end{tikzpicture}}
		\end{minipage}
	\begin{minipage}[b]{.4\linewidth}
			\centering\scalebox{.6}{\begin{tikzpicture}
				\begin{scope}[every node/.style={regular polygon,regular polygon sides=7, very thick,draw}, minimum size=10pt]
				\node (h1) at (0,1.5) {};
				\node (l2) at (0,2.5) {};
				\node (h2) at (-2.5,2.5) {};
				\node (l3)at (2.5,2.5) {};
				
				\node (l4)at (-3.1,3.5) {};
				\node (l6)at (-1.9,3.5) {};
				\node[opacity=0.2] (l7)at (-0.8,3.5) {};
				\node[opacity=0.2] (l8)at (0,3.5) {};
				\node[opacity=0.2] (l9)at (0.8,3.5) {};
				\node[opacity=0.2](l10)at (1.7,3.5) {};
				\node[opacity=0.2] (l11)at (2.5,3.5) {};
				\node[opacity=0.2] (l12)at (3.3,3.5) {};
				
				\node[opacity=0.2] (l13)at (3.3,4.5) {};
				\node[opacity=0.2] (l14)at (2.5,4.5) {};
				\node[opacity=0.2] (l15)at (4.1,4.5) {};
				\end{scope}
				\node[opacity=0.2] at (2.5,4.5) {\scalebox{1.5}{$\times$}};
				\begin{scope}[>={latex[black]}, every node/.style={regular polygon,regular polygon sides=7,draw},
				every edge/.style={draw=black,very thick}]
				\path [-] (h1) edge (h2);
				\path [-] (h1) edge (l2);
				\path [-] (h1) edge (l3);
				\path [-] (h2) edge (l4);
				\path [-] (h2) edge (l6);
				
				\path [opacity=0.2] (l2) edge (l7);
				\path [opacity=0.2] (l2) edge (l8);
				\path [opacity=0.2] (l2) edge (l9);
				
				\path [opacity=0.3] (l3) edge (l10);
				\path [opacity=0.3] (l3) edge (l11);
				\path [opacity=0.3](l3) edge (l12);
				
				\path [opacity=0.3] (l12) edge (l13);
				\path [opacity=0.3] (l12) edge (l14);
				\path [opacity=0.3] (l12) edge (l15);
				\end{scope}
				\end{tikzpicture}}
		\end{minipage}
	\end{center}
	\caption[lopping]{Left: tree in $\Xi$ (black and grey) and a lopping of it (black). Right: tree in $\Xi^\pA$ (black and grey) and a lopping of it (black).}
	\label{fig:lopping}
\end{figure}
\begin{lem}\label{elprop}
Let $\alpha\in\Xi$ and $\bal\in\Xi^\pA$, and let $\alpha_0\in\Xi$ and $\bal_0\in\Xi^\pA$ be a lopping of $\alpha$ and $\bal$, respectively.
\begin{enumerate}[label=$(\alph*)$]
\item For any leaf-type configuration $c$ of $\alpha$, the type of the root of $\alpha_0$ under the leaf-type configuration $(\cpv{\ell})_{\ell\in L(\alpha_0)}$ coincides with the type of the root of $\alpha$ under $c$.
\item For any reduced leaf-type configuration $\hat{c}$ of $\bal$, the type of the root of $\bal_0$ under the reduced leaf-type configuration $(\chv{\ell})_{\ell\in \hat{L}(\bal_0)}$ coincides with the type of the root of $\bal$ under $\hat{c}$.
\end{enumerate}
\end{lem}
\begin{proof}
It follows directly from the fact that the type of a vertex is exclusively determined by the type of its children
(see Definitions~\ref{typro} and \ref{typrob}).
\end{proof}
If we type a single leaf in a pruned tree, the type of some of its ancestors are already determined under the type propagation described in Definition~\ref{typrob} (see Fig. \ref{fig:singletype} for an example).
\begin{figure}[t!]
	\begin{center}
			\centering\scalebox{.55}{\begin{tikzpicture}
				\begin{scope}[every node/.style={regular polygon,regular polygon sides=7, very thick,draw}, minimum size=15pt]
				\node (h0) at (-4,0.5) {};
				\node (hm1) at (-8,1.5) {};
				\node (hm1m) at (-8,2.5) {};
				\node (hm1l) at (-10.5,2.5) {};
				\node (hm1r) at (-5.5,2.5) {};
				\node (hm1p) at (-6.1,3.5) {};
				\node (n1) at (-6.9,4.5) {};
				\node (n) at (-6.1,4.5) {};
				\node (n2) at (-5.4,4.5) {};
				\node (hm1mi) at (-4.9,3.5) {};
				\node (h1) at (0,1.5) {};
				\node (l2) at (0,2.5) {};
				\node (h2) at (-2.5,2.5) {};
				\node (l3)at (2.5,2.5) {};
				
				\node (l7)at (-0.8,3.5) {};
				\node (l8)at (0,3.5) {};
				\node (l9)at (0.8,3.5) {};
				\node(l10)at (1.7,3.5) {};
				\node(l11)at (2.5,3.5) {};
				\node(l12)at (3.3,3.5) {};
				
				\node (l13)at (3.3,4.5) {};
				\node (l14)at (2.5,4.5) {};
				\node (l15)at (4.1,4.5) {};
				\end{scope}
				\node at (2.5,4.5) {\scalebox{1.5}{$\times$}};
				\node at (1.7,3.5) {\scalebox{1.5}{$\times$}};
				\node at (-2.5,2.5) {\scalebox{1.5}{$\times$}};
				\node at (0,1.5) {\scalebox{1.5}{$1$}};
				\node at (3.3,4.5) {\scalebox{1.5}{$1$}};
				\node at (3.3,5.1)  {\scalebox{1.5}{$\downarrow$}};
				\node at (3.3,3.5) {\scalebox{1.5}{$1$}};
				\node at (2.5,2.5) {\scalebox{1.5}{$1$}};
				\node at (-6.9,4.5)  {\scalebox{1.5}{$0$}};
				\node at (-6.9,5.1)  {\scalebox{1.5}{$\downarrow$}};
				
				\node at (-6.1,3.5)  {\scalebox{1.5}{$0$}};
				\node at (-5.5,2.5)  {\scalebox{1.5}{$0$}};
				
				\begin{scope}[>={latex[black]}, every node/.style={regular polygon,regular polygon sides=7,draw},
				every edge/.style={draw=black,very thick}]
				\path [-] (h1) edge (h2);
				\path [-] (h0) edge (hm1);
				\path [-] (h0) edge (h1);
						\path [-] (hm1m) edge (hm1);		
						\path [-] (hm1l) edge (hm1);		
						\path [-] (hm1r) edge (hm1);
				\path [-] (hm1r) edge (hm1mi);
				\path [-] (hm1r) edge (hm1p);
				\path [-] (n1) edge (hm1p);
				\path [-] (n2) edge (hm1p);
				\path [-] (n) edge (hm1p);
				\path [-] (h1) edge (l2);
				\path [-] (h1) edge (l3);

				\path  (l2) edge (l7);
				\path  (l2) edge (l8);
				\path  (l2) edge (l9);
				
				\path  (l3) edge (l10);
				\path  (l3) edge (l11);
				\path  (l3) edge (l12);
				
				\path  (l12) edge (l13);
				\path  (l12) edge (l14);
				\path  (l12) edge (l15);
				\end{scope}
				\end{tikzpicture}}
	\end{center}
	\caption[singletype]{Type propagation in a tree in $\Xi^\pA$. Leaves indicated by $\downarrow$ have been assigned the type displayed in the interior of the leaf. Some ancestors of these leaves inherit the type due to the type propagation rules.}
	\label{fig:singletype}
\end{figure}
\begin{lem}\label{propa}
Let $\bal\in\Xi^\pA$, $\ell\in \hat{L}(\bal)$, $v\in V(\bal)$, and let $\hat{c}$ be a reduced leaf-type configuration of $\bal$. 
\begin{enumerate}[label=$(\alph*)$]
\item If $\ell$ is not a firewall, $\hat{c}_\ell=1$, and $w_\ell(\bal)\prec_\bal^{} v\preceq_\bal^{} \ell$, then $\chv{v}=1$.
\item If $\hat{c}_\ell=0$ and $\rho(R_\ell(\bal))\preceq_\bal^{} v\preceq_\bal^{} \ell$, then $\chv{v}=0$.
\end{enumerate}
\end{lem}
\begin{proof}
We first prove $(a)$. By assumption, $\ell$ is not a firewall. Let
$$w_\ell(\bal)\eqdef v_0\preceq_\bal v_1\preceq_\bal\cdots\preceq_\bal v_n\defeq\ell$$
be the vertices in the path connecting $w_\ell(\bal)$ and $\ell$ (note that $n>0$ by assumption). We have to show that, for all $i\in[n]$, $\chv{v_i}=1$. We do this via a backward induction in $i\in[n]$. For $i=n$, the result is true by assumption. We assume that $n>1$ (otherwise the proof is already complete) and that $i\in[n-1]$ is such that $\chv{v_{i+1}}=1$. By definition of $w_\ell(\bal)$, $v_i$ and $v_{i+1}$ are not firewalls. Thus, $\deg(v_i)=3$, and the left child of $v_i$ is not $v_{i+1}$ but is a leaf and has mark $\times$. Since $\chv{v_{i+1}}=1$, we deduce from the type propagation rules that $\chv{v_{i}}=1$. This ends the proof of $(a)$.
\smallskip

Now we prove $(b)$. Let $\rho(R_\ell(\bal))\eqdef v_1\preceq_\bal\cdots\preceq_\bal v_n\defeq\ell$ be the vertices in the directed path connecting $\rho(R_\ell(\bal)$ and $\ell$. We have to show that for all $i\in[n]$, $\chv{v_i}=0$. We do this via a backward induction in $i\in[n]$. For $i=n$, the result is true by assumption. We assume that $n>1$ (otherwise the proof is already complete) and that $i\in[n-1]$ is such that $\chv{v_{i+1}}=0$. Since~$v_i$ and~$v_{i+1}$ belong to the same region, then $\deg(v_i)=2$, or $\deg(v_i)=3$ and~$v_{i+1}$ is the left child of~$v_i$. It follows from the type propagation rules that $\chv{v_{i}}=0$, which ends the proof of~$(b)$. 
\end{proof}
The following result shows that for the root type, assigning type $1$ (resp. $0$) to an unmarked leaf $\ell$ of a pruned tree is equivalent to applying the pruning operator $\pi_\ell^\times$ (resp. $\pi_\ell^\circ$) to the tree.
\begin{lem}\label{typ-prun}
Let $\bal\in\Xi^\pA$, $\ell\in\hat{L}(\bal)$ and let $\hat{c}$ be a reduced leaf-type configuration of $\bal$. 
\begin{enumerate}[label=$(\alph*)$]
\item $\hat{L}(\pi_\ell^\times(\bal)),\, \hat{L}(\pi_\ell^\circ(\bal))\subset \hat{L}(\bal)$.
\item If $\hat{c}_\ell=1$, the type of the root of $\bal$ under $\hat{c}$ agrees with  the type of the root of $\pi_\ell^\times(\bal)$ under  $(\hat{c}_\ell)_{\ell\in\hat{L}(\pi_\ell^\times(\bal))}$.
\item If $\hat{c}_\ell=0$, the type of the root of $\bal$ under $\hat{c}$ agrees withthe type of the root of $\pi_\ell^\circ(\bal)$ under  $(\hat{c}_\ell)_{\ell\in\hat{L}(\pi_\ell^\circ(\bal))}$.
\end{enumerate}
\end{lem}
\begin{proof}
Fix $\bal\in\Xi^\pA$, $\ell\in\hat{L}(\bal)$ and a reduced leaf-type configuration $\hat{c}$ of $\bal$. 
Part $(a)$ is a direct consequence of the definitions of $\pi_\ell^\times$ and $\pi_\ell^\circ$. 
\smallskip

Let us prove $(b)$. By assumption, $\hat{c}_\ell=1$. If $\ell$ is a firewall, $\pi_\ell^\times(\bal)$ is obtained by marking $\ell$ with $\times$ in $\bal$, and the result follows from the definition of type propagation in $\Xi^\pA$. Assume now that $\ell$ is not a firewall. We define 
\[
    v_\ell \defeq \begin{cases} \text{the child of } w_\ell(\bal) \text{ that is  } \not \preceq_{\bar \alpha}\ell, & \text{if } \deg(w_\ell(\bal))=2, \\
     \text{the left child of } w_\ell(\bal), &  \text{if } \deg(w_\ell(\bal))=3.\end{cases}
\]   
 Let $\bal_0$ be the lopping of $\bal$ obtained by removing from $\bal$ the descendants of $w_\ell(\bal)$. By Lemma \ref{elprop} the type of the root of $\bal$ under $\hat{c}$ equals the type of the root of $\bal_0$ under $(\chv{l})_{l\in \hat{L}(\bal_0)}$. Let now $\bal_1$ be the lopping of $\pi_\ell^\times(\bal)$ obtained by removing  the descendants of $v_\ell$. Thanks to  Lemma \ref{elprop}, the type of the root of $\pi_\ell^\times(\bal)$ under $(\chv{l})_{l\in \hat{L}(\pi_\ell^\times(\bal))}$ agrees with the type of the root of $\bal_1$ under $(\chv{l})_{l\in \hat{L}(\bal_1)}$. By the definition of $\pi_\ell^\times(\bal)$, the trees $\bal_0$ and $\bal_1$ are equal (after relabelling  $w_\ell(\bal)$ as $v_\ell$). Therefore, it only remains to show that 
\begin{equation}\label{equal}
\chv{v_\ell}=\chv{w_\ell(\bal)}
\end{equation}
 To this end, note first that $v_\ell$ is a child of $w_\ell(\bal)$ and is not an ancestor of $\ell$. Moreover, the child of $w_\ell(\bal)$ that is $\preceq_\bal \ell$ has, thanks to Lemma \ref{propa}, type $1$. Therefore \eqref{equal} follows from the type propagation rules. 
\smallskip

Now, we prove $(c)$. By assumption, $\hat{c}_\ell=0$. If $\rho(\bal)\in R_\ell(\bal)$, we deduce from Lemma~\ref{propa} that $\rho(\alpha)$ gets type~$0$ under $\hat{c}$. In addition, $\pi_\ell^\circ(\bal)=\rootben$, and the result follows in this case. Assume now that $\rho(\bal)\notin R_\ell(\bal)$. In $\bal$, let $v_\ell$ be the parent of $\rho(R_\ell(\bal))$ and $\chil{v}_\ell$ its left child. Denote by $\chrl{\ell}$ the vertex-type propagation of $(\hat{c}_l)_{l\in\hat{L}(\pi_\ell^\circ(\bal))}$ in $\pi_\ell^\circ(\bal)$. Assume (as subcase (i)) that $\chil{v}_\ell$ has no mark. Then, the tree $\bal_0$ obtained from $\bal$ by removing  the descendants of $\rho(R_\ell(\bal))$ is a lopping of both $\bal$ and of $\pi_\ell^\circ(\bal)$. In addition, we have
$\chv{l}=\hat{c}_l=\chrl{\ell}_l,\quad\textrm{for $l\in\hat{L}(\bal_0)$ with $l\neq v_\ell.$}$
Thus, thanks to Lemma \ref{elprop}, it is enough to show that
\begin{equation}\label{equal2}
\chv{v_\ell}=\chrl{\ell}_{v_\ell}.
\end{equation}
In addition, the children of $v_\ell$ in $\bal$ that differ from $\rho(R_\ell(\bal))$ are the children of $\rho(R_\ell(\bal))$ in $\pi_\ell^\circ(\bal)$ and they have, by construction, the same type under $\chr$ and $\chrl{\ell}$. Furthermore, Lemma~\ref{propa} implies that $\chv{\rho(R_\ell(\bal))}=0$, and hence~\eqref{equal2} follows by the type propagation rules. Finally, assume (as subcase (ii)) that $\chil{v}_\ell$ is a leaf and has mark $\times$. Let $\bal_1$ be the tree obtained by removing in $\bal$ the mark of $\chil{v}_\ell$ and extend $\hat{c}$ to the reduced leaf-type configuration $\hat{c}^{1}$ of $\bal_1$ by setting $\hat{c}^1_{\chil{v}_\ell}=1$. This way, the roots of $\bal$ and $\bal_1$ have the same type under $\hat{c}$ and $\hat{c}^1$, respectively. Moreover, applying subcase (i) to $\bal_1$ , we deduce that
the roots of $\bal_1$ and $\pi_\ell^\circ(\bal_1)$ have the same type under $\hat{c}^1$ and its restriction to $\hat{L}(\pi_\ell^\circ(\bal_1))$, respectively. The result then follows in this case, by noticing that $\pi_\ell^\circ(\bal)=\pi_{\chil{v}_\ell}^\times(\pi_\ell^\circ(\bal_1))$, and applying part $(b)$ to $\pi_\ell^\circ(\bal_1)$ at $\chil{v}_\ell$. This ends the proof.
\end{proof}

The next result is at the core of the proof of Corollary~\ref{dualpASG}.

\begin{lem}\label{admpru}
	Let $\alpha\in\Xi$ and let $\bal\in\Xi^{\pA}$ be an admissible pruning of $\alpha$. 
	\begin{enumerate}[label=$(\alph*)$]
	\item For any $\ell\in L(\alpha)\setminus\hat{L}(\bal)$, $\bal$ is an admissible pruning of $\asg_\ell^{\nwedge}$, $\asg_\ell^{\pitch}$, $\asg_\ell^{\times}$ and $\asg_\ell^{\circ}$.
	\item For any $\ell\in \hat{L}(\bal)$, $\bal_\ell^\nwedge$, $\bal_\ell^\pitch$, $\pi_\ell^\times(\bal)$ and $\pi_\ell^\circ(\bal)$ are admissible prunings of $\asg_\ell^{\nwedge}$, $\asg_\ell^{\pitch}$, $\asg_\ell^{\times}$ and $\asg_\ell^{\circ}$, respectively. 
	\end{enumerate}
 \end{lem}
\begin{proof}
Let $\alpha\in\Xi$ and let $\bal\in\Xi^{\pA}$ be an admissible pruning of $\alpha$. 
\smallskip

\underline{Proof of  $(a)$}. Let $\ell\in L(\alpha)\setminus\hat{L}(\bal)$ and $\star\in\{\nwedge,\pitch,\times,\circ\}$. We have to show that $\bal$ and $\alpha_\ell^\star$ satisfy conditions $(1)$, $(2)$ and $(3)$ in Definition \ref{adp}. Since $\bal$ is an admissible pruning of $\alpha$, we see from the definition of $\alpha_\ell^\star$ that
$$V(\bal)\subset V(\alpha)\subset V(\alpha_\ell^\star)\quad\textrm{and}\quad \hat{L}(\bal)\subset L(\alpha)\setminus\{\ell\}\subset L(\alpha_\ell^\star).$$ 
Thus, $(1)$ is satisfied. Now, let $u,v\in V(\bal)$ and assume that $u\prec_\bal^{} v$. Since $\bal$ is an admissible pruning of $\alpha$, we have $u\prec_\alpha^{} v$. Moreover, $\alpha$ is a subtree of $\alpha_\ell^\star$, and hence $u\prec_{\alpha_\ell^\star}^{} v$, which proves $(2)$. It remains to prove $(3)$. Let $c$ be a leaf-type configuration of $\alpha_\ell^\star$ and define the leaf-type configuration $\tilde{c}$ of $\alpha$ via $\tilde{c}_l=\cpv{l}$, $l\in L(\alpha)$. By construction, the type of the root of $\alpha_\ell^\star$ under ${c}$ coincides with the type of the root of $\alpha$ under $\tilde{c}$. The latter coincides with the type of the root of $\bal$ under $(\tilde{c}_l)_{l\in\hat{L}(\bal)}$, because $\bal$ is an admissible pruning of $\alpha$. Thus, (3) follows using that $(\tilde{c}_l)_{l\in\hat{L}(\bal)}=({c}_l)_{l\in\hat{L}(\bal)}$.
\smallskip

\underline{Proof of $(b)$ for $\nwedge$ and $\pitch$}. Fix $\ell\in \hat{L}(\bal)$. We will show that, for $\star\in\{\nwedge,\pitch\}$, $\bal_\ell^\star$ and $\alpha_\ell^\star$ satisfy conditions $(1)$, $(2)$ and $(3)$ in Definition \ref{adp}. Note that, setting $L_\ell^\pitch=\{\chil{\ell},\chim\ell,\chir{\ell}\}$ and $L_\ell^\nwedge=\{\chil{\ell},\chir{\ell}\}$, we have 
\[V(\bal_\ell^\star)=V(\bal)\cup L_\ell^\star\subset V(\alpha)\cup L_\ell^\star=V(\alpha_\ell^\star)\quad\textrm{and}\quad \hat{L}(\bal_\ell^\star)=(\hat{L}(\bal)\setminus \{\ell\})\cup L_\ell^\star\subset (L(\alpha)\setminus\{\ell\})\cup L_\ell^\star=L(\alpha_\ell^\star),\]
which proves $(1)$. Let $u,v\in V(\bal_{\ell}^\star)$ with $u\prec_{\bal_\ell^\star}^{} v$. Assume first that $v\in V(\bal)$. In this case, $u\in V(\bal)$ and since $\bal$ is a subtree of $\bal_{\ell}^\star$, we have $u\prec_{\bal}^{} v$. Using that $\bal$ is an admissible pruning of $\alpha$, we infer that $u\prec_{\alpha}^{} v$. Since $\alpha$ is a subtree of $\alpha_{\ell}^\star$, it follows that $u\prec_{\alpha_\ell^\star}^{} v$. Assume now that $v\in L_\ell^\star$. Using the previous argument for $u$ and $\ell$, we infer that $u\prec_{\alpha_\ell^\star}^{} \ell$. Since $v$ is a child of $\ell$ in $\alpha_{\ell}^\star$, it follows that $u\prec_{\alpha_\ell^\star}^{} v$. This achieves the proof of $(2)$. Next, we check that $\bal_\ell^\star$ and $\alpha_\ell^\star$ satisfy $(3)$. Let $c$ be a leaf-type configuration of $\alpha_\ell^\star$. We associate to $c$ the leaf-type configuration $\tilde{c}$ of $\alpha$ by setting $\tilde{c}_l=\cpv{l}$, $l\in L(\alpha)$. Denote by $\hat{c}$ the reduced leaf-type configuration of $\bal_\ell^\star$ given by $\hat{c}_l=c_l$, $l\in \hat{L}(\bal_\ell^\star)$. By construction, we have: (i) the type of the root of $\alpha_\ell^\star$ under $c$ coincides with the type of the root of $\alpha$ under $\tilde{c}$, and (ii) the type of the root of $\bal_\ell^\star$ under $\hat{c}$ coincides with the type of the root of $\bal$ under $(\chv{l})_{l\in\hat{L}(\bal)}$. Note that, for $l\in \hat{L}(\bal)$ with $l\neq \ell$, we have $\tilde{c}_l=c_l=\hat{c}_l=\chv{l}$, and $\tilde{c}_\ell=\cpv{\ell}=\chv{\ell}$. Thus $(\chv{l})_{l\in\hat{L}(\bal)}$is the restriction of $\tilde{c}$ to $\hat{L}(\bal)$. The result follows using that $\bal$ is an admissible pruning of $\alpha$.
\smallskip

\underline{Proof of $(b)$ for $\times$ and $\circ$}. Fix $\ell\in \hat{L}(\bal)$. We will show  that $\pi_\ell^\star(\bal)$ and $\alpha_\ell^\star$ satisfy conditions $(1)$, $(2)$ and $(3)$ in Definition \ref{adp}. Note that by construction $\ell\notin\hat{L}(\pi_\ell^\star(\bal))$. Thus, since $\bal$ is an admissible pruning of $\alpha$, it follows from the definition of $\pi_\ell^\star(\bal)$ and $\alpha_\ell^\star$ that
$$V(\pi_\ell^\star(\bal))\subset V(\bal)\subset V(\alpha)\subset V(\alpha_\ell^\star)\quad\textrm{and}\quad \hat{L}(\pi_\ell^\star(\bal))\subset \hat{L}(\bal)\setminus\{\ell\}\subset {L}(\alpha)\setminus\{\ell\}\subset L(\alpha_\ell^\star),$$
which proves $(1)$. Next, we prove $(2)$ for $\star=\times$. Let $u,v\in V(\pi_{\ell}^\times(\bal))$ with $u\prec_{\pi_\ell^\times(\bal)}^{} v$. We claim that $u\prec_{\bal}^{} v$. If the claim is true, then $u\prec_{\alpha_\ell^\times}^{} v$, because $\bal$ is an admissible pruning of $\alpha$, which is a subtree of $\alpha_\ell^\times$. Thus, to prove $(2)$ for $\star=\times$, it suffices to prove the claim. If $\ell$ is a firewall, the claim follows, because $\pi_\ell^\times(\bal)$ and $\bal$ have the same tree structure. Assume now that $\ell$ is not a firewall. If $w_\ell(\bal)\nprec_\bal v$ or if $w_\ell(\bal)_\bal \prec_\bal u$, then the directed path from $u$ to $v$ in $\pi_\ell^\times(\bal)$ is also a directed path in $\bal$, and the claim follows in this case. By construction of $\pi_\ell^\times(\bal)$, the only remaining possibility is that $u\prec_\bal w_\ell(\bal)$ and $w_\ell(\bal)\prec_\bal v$, in which case the claim is also true, which ends the proof of $(2)$ for $\star=\times$. Let us now prove $(2)$ for $\star=\circ$. Let $u,v\in V(\pi_{\ell}^\circ(\bal))$ with $u\prec_{\pi_\ell^\circ(\bal)}^{} v$ and denote by $\chil{v}_\ell$ the left child of the parent of $\rho(R_\ell(\bal))$. If $\chil{v}_\ell$ has no mark, it follows from the definition of $\pi_\ell^\circ (\bal)$ that $u\prec_{\bal}^{} v$. Since $\bal$ is a subtree of $\bal_\ell^\circ$, we have $u\prec_{\bal_\ell^\circ}^{} v$. For the remaining case, i.e. when $\chil{v}_\ell$ has mark $\times$, note that $\pi_\ell^\circ(\bal)=\pi_{\chil{v}_\ell}(\pi_\ell^\circ(\bal_1))$, where $\bal_1$ is obtained from $\bal$ by removing the mark of $\chil{v}_\ell$. Therefore, the result follows by combining the previous case and the result for $\star=\times$. 
\smallskip

Finally, we show that $\pi_\ell^\star(\bal)$ and $\alpha_\ell^\star$ satisfy $(3)$. Let $c$ be a leaf-type configuration of $\alpha_\ell^\star$. Thanks to Lemma~\ref{elprop}, the types of the roots of $\alpha_\ell^\star$ and $\alpha$ are equal under $c$ and $(\cpv{l})_{l\in L(\alpha)}$, respectively. This type coincides with the type of the root of $\bal$ under $(\cpv{l})_{l\in \hat{L}(\bal)}$, because $\bal$ is an admissible pruning of $\alpha$. Moreover, 
$\cpv{\ell}=0$ if $\star=\circ$, and $\cpv{\ell}=1$ if $\star=\times$. Therefore, using Lemma \ref{typ-prun}, we conclude that the type of the root of $\bal$ under $(\cpv{l})_{l\in \hat{L}(\bal)}$ agrees with the type of the root of $\pi_\ell^\star(\bal)$ under $(\cpv{l})_{l\in \hat{L}(\pi_\ell^\star(\bal))}$, achieving the proof.
\end{proof}
\begin{proof}[Proof of Corollary \ref{dualpASG}]
Let $\bal\in\Xi^\pA$ and define $\asg\in{\Xi}$ by adding a child without a mark to every marked leaf of $\bal$. Next, set $a_0=\alpha$ and $\ba_0=\bal$, and construct the eASG and the pASG processes $(a_t)_{t\geq 0}$ and $(\ba_t)_{t\geq 0}$ via the coupling described in Section~\ref{sec:mainresults:subsec:pASG} below Definition~\ref{def:pruningoperations}. Clearly $\bal$ is an admissible pruning of $\alpha$. Thus, Lemma~\ref{admpru} implies that, for all $t\geq0$, $\ba_t$ is an admissible pruning of $a_t$. Therefore, $H(a_t,y_0)=H(\ba_t,y_0)$ for all $t\geq 0$ and $y_0\in[0,1]$, and the result follows from Theorem~\ref{sec:mainresults:thm:dualityASG}.
\end{proof}

We close this section with a result that will be useful in the next subsection. It can be summarised by saying that the type of the root of a region within a pruned tree is unfit if and only if all the vertices in the region are unfit; and if the vertices outside of the region are ignored, then the root-type of the region is unfit if and only if all leaves of the region are unfit. To this end, for $\bal\in \Xi^{\pA}$, a reduced leaf-type configuration~$\hat{c}$ of~$\bal$, and a region~$R$ of $\bal$, define for $v\in R$, $\chr_v(R)$ as the vertex type of $v$ that arises under the leaf-type configuration~$\hat{c}$ if all vertices outside of~$R$ are ignored; that is, if $\deg(v)=1$ in $R$, then $\chr_v(R)$ is determined by the type of the child of~$v$ in $R$.

\begin{lem}\label{lem:roottype}
	Let $\bal\in \Xi^{\pA}$, a reduced leaf-type configuration~$\hat{c}$ of~$\bal$, and a region~$R$ of $\bal$. Then, 
	\begin{equation}
		\chr_{\rho(R)}(R)=\prod_{v\in R} \chr_v(R)=\prod_{\ell\in \hat{L}(\bal)\cap R} \hat{c}_\ell.\label{eq:observ}
	\end{equation}
	
\end{lem}
\begin{proof}
	This follows by iterating Definition~\ref{typro}-(3).
\end{proof}

\subsection{The sASG: results related to Section~\ref{sec:mainresults:subsec:sASG}}\label{sec:proofssASG}
In this section we prove the results related to the stratification. 

\smallskip

The following lemma makes precise in what sense the root type of a pruned tree agrees with the output type of its stratification; this leads to~\eqref{root=winnertype}.
\begin{lem}\label{out}
	Let $\bal \in \Xi^{\pA}$ with $\bal\neq\rootben$, $\hat{c}=(\hat{c}_\ell)_{\ell\in \hat{L}(\bal)}$ a reduced leaf-type configuration of $\bal$, and $s(\bal)=(\tau,m)\in \Upsilon$ the stratification of $\bal$. Define the primary-vertex configuration $\tilde{c}=(\tilde{c}_v)_{v\in V^1(\tau)}$ via
	\begin{equation}
		\tilde{c}_{R}\defeq\prod_{\ell\in \hat{L}(\bal)\cap R} \hat{c}_\ell,\quad\textrm{ for any region $R$ of $\bal$}.\label{eq:pvconfig}
	\end{equation}
	Then $\chr_{\rho(\bar{\alpha})}=\tf(\tilde{c},\tau)$. In particular, $H(\bal,y_0)=\Hs(s(\bal),y_0)$. 
\end{lem}
\begin{proof}
	We prove the lemma via induction on the number~$k$ of ternary branchings in $\bal$. Suppose $k=0$ and let $\hat{c}$ be a reduced leaf-type configuration of $\bal$. Because $k=0$, $\bal$ only has one region~$R$, and hence $s(\bal)=(\tau,m)$ is the tripod tree consisting only of one primary vertex $R$; $R$ is the root of $\tau$, and there are no secondary vertices. By the definition of the output type and~\eqref{eq:pvconfig}, $\tf(\tilde{c},\tau)=\tilde{c}_{R}=\prod_{v\in \hat{L}(\bal)\cap R} \hat{c}_v$. If $k=0$, either~(i) $\bal=\rootdel$, (ii)~$\bal=\rootsin$, or (iii)~$\bal$ contains no marked vertex and has only binary branchings. In~(i), $\chr_{\rho(\bar{\alpha})}=1$, which agrees with $\tf(\tilde{c},\tau)$ because $\hat{L}(\bal)\cap R=\emptyset$ and the empty product is~$1$. In~(ii), the result follows by Definition~\ref{typro} (1). In~(iii),  \eqref{eq:observ} leads to $\chr_{\rho(\bal)}=\prod_{\ell\in \hat{L}(\bal)\cap R} \hat{c}_\ell$, which also agrees with~$\tf(\tilde{c},\tau)$. Next, assume $k=n$ for some $n\in \N$. Let $\hat{c}$ be the reduced leaf-type configuration of the new $\bal$, and $s(\bal)=(\tau,m)$. Recall that for $v\in V(\bal)$, $R_v(\bal)$ is the region in~$\bal$ containing~$v$. Let $v_1,\ldots, v_\ell$ be the vertices in $R_{\rho(\bal)}(\bal)$ that have outdegree~$3$ in $\bal$, and let for $i\in[n]$, $\bal_{\chim{v}}$ and $\bal_{\chir{v}}$ be the subtrees in $\bal$ rooted in $\chim{v}$ and $\chir{v}$, respectively. By iterating Definition~\ref{typro}, and~\eqref{eq:observ}, 
	\begin{equation}
		\hat{c}_{\rho(\bal)}=\prod_{v\in R} \chr_v = \prod_{v\in R} \chr_v(R_{\rho(\bal)}(\bal)) \prod_{i=1}^n \max\{\chr_{\chim{v}_i},\chr_{\chir{v}_i}\}=\! \! \! \! \! \! \! \! \prod_{\ell\in \hat{L}(\bal)\cap R_{\rho(\bal)}(\bal)} \! \! \! \! \! \! \! \! \hat{c}_\ell \ \ \prod_{i=1}^n \max\{\chr_{\chim{v}_i},\chr_{\chir{v}_i}\}.\label{eq:auxvt}
	\end{equation}
	We make the following three observations. 1) For $i\in [n]$, by the definition of the stratification, $v_i$ is the child of $\rho(\tau)$ in $\tau$, and $R_{\chim{v}_i}(\bal)$ and $R_{\chir{v}_i}(\bal)$ are the children of $v_i$ in $\tau$. Moreover, $s(\bal_{\chim{v}_i})=(\tau_{R_{\chim{v}_i}(\bal)},m_{R_{\chim{v}_i}(\bal)})$ is the subtree of~$\tau$ at $R_{\chim{v}_i}(\bal)$; analogous for $\bal_{\chir{v}}$. 2) $\chr_{\chim{v}_i}=\chr_{\rho(\bal_{\chim{v}_i} ) }$; but $\bal_{\chim{v}_i}\in \Xi^{\pA}$ has less than $n$ vertices with outdegree~$3$ so that by the induction hypothesis, $\chr_{\rho(\bal_{\chim{v}_i} ) }=\tf(\tilde{c},\tau_{R_{\chim{v}_i}(\bal)})$; analogous for $\chr_{\chir{v}_i}$. 3) $R_\rho(\bal)=\rho(\tau)$ so that by~\eqref{eq:pvconfig}, $\prod_{\ell\in \hat{L}(\bal)\cap R_{\rho(\bal)}(\bal)}\hat{c}_\ell=\tilde{c}_{\rho(\tau)}$. Hence, continuing~\eqref{eq:auxvt} using observations 2), 3), and 1),
	\begin{align*}
		\hat{c}_{\rho(\bal)}&=\! \! \! \! \! \! \! \! \prod_{\ell\in \hat{L}(\bal)\cap R_{\rho(\bal)}(\bal)} \! \! \! \! \! \! \! \! \hat{c}_\ell \ \ \prod_{i=1}^n \max\{\tf(\tilde{c},\tau_{R_{\chim{v}_i}(\bal)}),\tf(\tilde{c},\tau_{R_{\chir{v}_i}(\bal)})\} \\
		&=\tilde{c}_{\rho(\tau)} \prod_{i=1}^n \max\{\tf(\tilde{c},\tau_{R_{\chim{v}_i}(\bal)}),\tf(\tilde{c},\tau_{R_{\chir{v}_i}(\bal)})\}= \tilde{c}_{\rho(\tau)} \! \! \! \! \! \!\prod_{v \text{ child of }\rho(\tau)} \max \{\tf(\tilde{c},\tau_w):w\text{ is child of }v\}.
	\end{align*}
This finishes the proof.
\end{proof}

\begin{proof}[Proof of Theorem~\ref{thm:duality} (duality sASG)]	Fix $t\geq 0$, $y_0\in [0,1]$, and~$\Ts\in \Upsilon_\Delta$. It is straightforward to construct $\bar{\asg}\in \Xi^\pA$ such that~$s(\bar{\asg})=\Ts$. Let $(\ba_r)_{r\geq 0}$ be a pASG process with $\ba_0=\bar{\asg}$. By~\eqref{root=winnertype} (alternatively Lemma~\ref{out}), Theorem~\ref{dualpASG}, and the definition of~$\Te$, 
	$$\Hs(\Ts,y(t;y_0))=H(\bar{\asg},y(t;y_0))=\E_{\bar{\asg}}[H(\ba_t,y_0)]=\E_{\bar{\asg}}[\Hs(s(\ba_t),y_0)]=\E_\Ts [\Hs(\Te_t,y_0)].$$
	In particular, $y(t;y_0)=\E_{\ssroot{1}} [\Hs(\Te_t,y_0)]$ follows because $\Hs(\sroot{1},y(t;y_0))=y(t;y_0)$ by the definition of~$\Hs$. 
\end{proof}
\begin{lem}\label{tran-sasg}
	Let $\bal\in\Xi^\pA$. For any $\ell\in\hat{L}(\bal)$, we have
	\begin{enumerate}
		\item $s(\bal_\ell^\nwedge)=(s(\bal))_{R_\ell(\bal)}^\nwedge$,
		\item $s(\bal_\ell^\pitch)=(s(\bal))_{R_\ell(\bal)}^\pitch$,
		\item $s(\pi_\ell^\times(\bal))=(s(\bal))_{R_\ell(\bal)}^\times$,
		\item $s(\pi_\ell^\circ(\bal))=(s(\bal))_{R_\ell(\bal)}^\circ$.
	\end{enumerate}
\end{lem}
\begin{proof}
	(1)--(4) can all be proven along the same lines. Here, we only prove (2) and (3), because they are the more delicate cases for branching and pruning, respectively. Let $\bal\in \Xi^\pA$ and assume $\hat{L}(\bal)\neq \varnothing$. Let $\ell\in\hat{L}(\bal)$. First we prove (2). We write $\bal^{\pitch}$ for $\bal^{\pitch}_\ell$ to ease the notation. $\bal^{\pitch}$ arises from $\bal$ by adding leaves $\chil{\ell}$, $\chim{\ell}$, and $\chir{\ell}$ (with parent~$\ell$) to $\bal$. Hence, $R_{\ell}(\bal^{\pitch})=R_{\ell}(\bal)\cup\{\chil{\ell}\}$, $R_{\chim{\ell}}(\bal^{\pitch})=\{\chim{\ell}\}$, and $R_{\chir{\ell}}(\bal^{\pitch})=\{\chir{\ell}\}$; all other regions are as in $\bal$. In particular, in comparison with $s(\bal)=(\tau,m)$, $s(\bal^{\pitch})=(\tau^1,m^1)$ has an additional secondary vertex~$\ell$, which is the child of $R_\ell(\bal^{\pitch})$, and two additional primary vertices $R_{\chim{\ell}}(\bal^{\pitch})$ and $R_{\chir{\ell}}(\bal^{\pitch})$, which are connected to~$\ell$. The new weights are $m^1(R_{\ell}(\bal^{\pitch}))=m(R_{\chim{\ell}}(\bal))+1$, $m^1(R_{\chim{\ell}}(\bal^{\pitch}))=m^1(R_{\chir{\ell}}(\bal^{\pitch}))=1$, and all other weights are as in $s(\bal)$. But by Definition~\ref{def:sASGoperations}-(3), this is $s(\bal)^{\pitch}_\ell$, because $s(\bal)^{\pitch}_{R_\ell(\bal)}$ arises from $s(\bal)$ by increasing the weight of $R_{\ell}(\bal)$ by one, adding a secondary vertex as a child of $R_{\ell}(\bal)$, and adding to that secondary vertex two children that are primary vertices each with weight~$1$. In particular, $s(\bal_\ell^\pitch)=(s(\bal))_{R_\ell(\bal)}^\pitch$. 
		
	\smallskip
	
	Next, we prove (3). Recall Definition~\ref{def:firewall} (firewall), and that for~$\ell\in L(\bal)$, $w_\ell(\bal)$ is the most recent ancestor of~$\ell$ that is a firewall. For the proof we write $\bal^\times\defeq \pi_\ell^\times(\bal)$ to ease the notation. To determine $\bal^\times$, by Definition~\ref{def:pruningoperations}, we have to distinguish i) $\ell$ is a firewall, ii) $\ell$ is not a firewall and $\deg(w_\ell(\bal))=2$, and iii) $\ell$ is not a firewall and $\deg(w_\ell(\bal))=3$. In i), $\ell$  is either the root, or the parent~$w$ of~$v$ has $\deg(w)=3$ and $\chil{w}=v$. In both cases, $\bal^\times$ is obtained by marking~$\ell$ with $\times$. In particular, the tree structure is as in $\bal$, but $\lvert R_\ell(\bal^\times)\rvert=\lvert R_\ell(\bal^\times)\rvert-1$. Hence, $s(\bal^\times)=(\tau,m^1)$, where $m^1(R_\ell(\bal^\times))=m(R_\ell(\bal))-1$. Now we check the corresponding operation on the stratification. If $\ell$ is the root, then $R_\ell(\bal)$ is the root of~$\tau$ by definition. Hence, $s(\bal)^\times_{R_\ell(\bal)}=(\tau,m^2)$ with $m^2(R_\ell(\bal^\times))=m(R_\ell(\bal))-1$, by Definition~\ref{def:sASGoperations}-(iii). 
	If~$\ell$ has parent~$w$ with $\deg(w)=3$ and $v=\chil{w}$, then $v$ has siblings $\chim{w}$ and $\chir{w}$. $R_\ell(\bal)$, $R_{\chim{w}}(\bal)$, and $R_{\chir{w}}(\bal)$ are non-empty. By Definition~\ref{def:stratmap}, $R_{\chim{w}}(\bal)$ and $R_{\chir{w}}(\bal)$ are children of $w$ in $\tau$; and $w$ is a child of $R_{v}(\bal)$ in $\tau$. In particular, $R_{v}(\bal)$ is not a leaf of $\tau$ so that by Definition~\ref{def:sASGoperations}, $s(\bal)^\times_{R_\ell(\bal)}=(\tau,m^2)$ with $m^2(R_\ell(\bal))=m(R_\ell(\bal))-1$. In both cases, $s(\bal^\times)=s(\bal)^\times_{R_\ell(\bal)}$. 
	Next, we consider ii) $\ell$ is not a firewall and $\deg(w_\ell(\bal))=2$. Then either iia) the parent of~$\ell$ has outdegree~$2$, or iib) the vertices between $\ell$ and $w_\ell(\bal)$ are the middle or right child in a ternary branching or have outdegree~$3$ with a marked left child. In iia), $\ell$ and $w_\ell(\bal)$ are in the same region; and $R_{w_\ell(\bal)}(\bal^\times)=R_{w_\ell(\bal)}(\bal)\setminus\{\ell\}$. In particular, $s(\bal^\times)=(\tau,m^1)$, where $\tau$ is as before, but $m^1(R_{w_\ell(\bal)}(\bal^\times))=m(R_{w_\ell(\bal)}(\bal))-1$. Because $R_{w_\ell(\bal)}(\bal)$ contains at least two leaves, $m(R_{w_\ell(\bal)}(\bal))>0$. Hence, $s(\bal)^\times_{R_\ell(\bal)}=(\tau,m^2)$ with $m^2(R_{w_\ell(\bal)}(\bal))=m(R_{w_\ell(\bal)}(\bal))-1$. This proves (3) under iia). In iib), for any $z$ such that $w_\ell(\bal)\prec_{\bal} z \prec_{\bal} \ell$, $z$ is the middle or right child in a ternary branching, $\deg(z)=3$, and $\chil{z}$ is marked. In particular, for every such $z$, $R_z(\bal)$ contains only marked leaves. W.l.o.g. we can assume $\chil{w}_\ell(\bal)$ is not on the path from $w_\ell(\bal)$ to~$\ell$ (otherwise replace $\chil{w}_\ell(\bal)$ by $\chir{w}_\ell(\bal)$). $\bal^\times$ is obtained from $\bal$ by replacing the subtree rooted in $w_\ell(\bal)$ by the subtree rooted in $\chil{w}_\ell(\bal)$ . In particular, all regions in $\bal_{\chir{w}_\ell}$, the subtree in~$\bal$ rooted in ${\chir{w}_\ell}$, are not regions of $\bal^\times$, and all regions in $\bal\setminus \bal_{\chir{w}_\ell}(\bal)$ are as before. In $s(\bal)$, we identify $R_{w_\ell(\bal)}(\bal)$ with the closest primary-vertex ancestor in $\tau$ with positive weight, and $w_\ell(\bal)$ as the child of $R_{w_\ell(\bal)}(\bal)$. Hence, $s(\bal^\times)$ arises from $s(\bal)$ by removing in $\tau$ the subtree rooted in $w_\ell(\bal)$. But the resulting tree is then $s(\bal)^\times_\ell$. To see this, note that for $z$ such that $w_\ell(\bal)\prec_{\bal} z \prec_{\bal} \ell$,  $R_z(\bal)\cap\hat{L}(\bal)=\varnothing$; but $R_\ell(\bal)=1$. By Definition~\ref{def:sASGoperations}-(3) and because $m(R_\ell(\bal))-1=0$, we remove the subtree rooted in the closest secondary-vertex ancestor of $R_\ell(\bal)$ (in $\tau$) that is a child of a primary vertex with positive weight. Every $R_z(\bal)$ is an ancestor, but $m(R_z(\bal))=0$; the first ancestor with positive weight is $R_{w_\ell(\bal)}(\bal)$. For case iii), all vertices between $\ell$ and $w_\ell(\bal)$ are the middle or right child in a ternary branching or have outdegree~$3$ with a marked left child. This case can be proven as case~iib).
\end{proof}

\begin{coro}
	The sASG process $(\Tp_t)_{t\geq0}$ is a continuous-time Markov chain with values in $\Upsilon_\Delta$
	and transition rates
	\begin{align*}
		q_{\Te}^{}(\Ts,\Ts^{\nwedge}_{v})\defeq s\,m(v),\quad  q_{\Te}^{}(\Ts,\Ts^{\pitch}_{v})\defeq\gamma\, m(v),\quad q_{\Te}^{}(\Ts,\Ts^{\times}_{v})\defeq u\nu_1\, m(v),\quad q_{\Te}^{}(\Ts,\Ts^{\circ}_{v})\defeq u\nu_0\, m(v),
	\end{align*}
	for~$\Ts=(\tau,m)\in \Upsilon$ and~$v\in V^1(\tau)$. The states~$\sroot{0}$ and~$\Delta$ are absorbing. 
\end{coro}
\begin{proof}
	This follows directly from Lemma \ref{tran-sasg}.
\end{proof}

\section{Long-term type frequency: proofs and additional results related to Section~\ref{sec:mainresults:subsec:application}}\label{sec:applicationIproofs}
In this section, we prove results that rely on the connection between the solution of the mutation--selection equation~\eqref{eq:dlimitdiffeq} and the sASG process.
We begin with proving Proposition~\ref{prop:treegrowsinfty}.
\begin{proof}[Proof of Proposition~\ref{prop:treegrowsinfty}]
	We claim that every $\Ts\in \Upsilon$ with $\Ts\neq\sroot{0}$ is transient. 
	If this  is true, the result follows, because for each $n\in\Nb$, $\{\Ts\in \Upsilon: M(\Ts)\leq n\}$ is a finite set. Let $\Ts\in \Upsilon$ with $\Ts\neq\sroot{0}$. Denote by $T_\Ts$ the first return time to~$\Ts$ of $\Te$ after its first jump. Note that before absorption the total mass $M(\Te_{\cdot})$ decreases at any jump with probability $u/(u+s+\gamma)$. Hence,
	$$\Pb_\Ts(T_\Ts=\infty)\geq \Pb_\Ts(T_{\mathrm{abs}}<T_\Ts)\geq \left(\frac{u}{u+s+\gamma}\right)^{M(\Ts)}>0,$$
	which proves the claim.
\end{proof}

We will prove that the output type is monotone in the primary vertex-type configuration.
\begin{lem}[Monotone output type] \label{lem:monotone}
For $\tau\in \Xi^\trip$, and $c$ and $c'$ two primary vertex-type configurations with $c_v\leq c_v'$ for all $v\in V^1(\tau)$, we have $\tf(c,\tau)\leq \tf(c',\tau)$. Moreover, if $c\equiv 0$ and $c'\equiv 1$, $\tf(c,\tau)=0$ and $\tf(c',\tau)=0$.
\end{lem}
\begin{proof}
	We prove the lemma by induction on the number of primary vertices. For $\tau$ with $\lvert V^1(\tau)\rvert=1$, the result is trivially true. For $\lvert V^1(\tau)\rvert=n+1$, assume $c_v\leq c_v'$ for all $v\in V^1(\tau)$. Using the definition of $\tf(c,\tau)$ and the induction hypothesis,
	\begin{align*}
		\tf(c,\tau)&=c_{\rho(\tau)} \prod_{v\,\text{child of }\rho(\tau) } 
	\max\left\{\tf(c,\tau_{w}):\, w\, \text{ is child of }v\right\}\\
	&\leq c_{\rho(\tau)}' \prod_{v\,\text{child of }\rho(\tau) } 
	\max\left\{\tf(c',\tau_{w}):\, w\, \text{ is child of }v\right\}.
	\end{align*}
The second part of the lemma can be proven along the same lines.
\end{proof}
A direct consequence of Lemma~\ref{lem:monotone} and the definition of~$\Hs$ is that for $\Ts\in \Upsilon$ with $\Ts\neq \sroot{0}$, \begin{equation}
	\partial_2 \Hs(\Ts,y_0)>0.\label{eq:monotoneH}
\end{equation}
We now deduce the long-term behaviour of $\Hs(\Te_t,y_0)$.
\begin{lem}[Sub-/super-/martingale]\label{lem:martingale}
	Let~$y_0\in [0,1]$, $F$ as defined in~\eqref{eq:dlimitdiffeq}, and $\Fs^{\Te}=(\Fs_r^{\Te})_{r\geq 0}$ be the natural filtration of~$\Te$. $(\Hs(\Te_r,y_0))_{r\geq 0}$ is a non-negative, bounded $\Fs^{\Te}\overset{\text{sub}}{\underset{\text{super}}{-}}$martingale if $F(y_0)\gtreqqless 0$. Moreover, for any $y_0\in [0,1]$, $\Hs_{\infty}(y_0)\defeq\lim_{r\to\infty} \Hs(\Te_r,y_0)\in [0,1]$ exists almost surely. 
\end{lem}
\begin{proof}
	Fix $r\geq 0$ and $y_0\in [0,1]$. Non-negativity and boundedness are immediate, because $\Hs(\cdot,y_0)\in[0,1]$. $\Hs(\Te_r,y_0)$ is measurable with respect to the $\sigma$-algebra generated by~$(\Te_t)_{t\leq r}$, and therefore $(\Hs(\Te_r,y_0))_{r\geq 0}$ is adapted with respect to $\Fs^{\Te}$. Fix $h\in[0,r]$ and $\Ts\in \Upsilon_\Delta$. 
	By the Markov property and Theorem~\ref{thm:duality},
	$$\E_{\Ts}[\Hs(\Te_r,y_0)\mid \Fs_h^{\Te}]=\E_{\Te_h}[\Hs(\Te_{r-h},y_0)]=\Hs(\Te_h,y(r-h;y_0)).$$
	If $\Te_h\in\{\ssroot{0},\Delta\}$, then the right-hand side is independent of $y_0$ so that in particular $\Hs(\Te_h,y(r-h;y_0))=\Hs(\Te_h,y(y_0))$.	If $\Te_h\notin\{\ssroot{0},\Delta\}$, $\Hs$ is monotone in the second argument by~\eqref{eq:monotoneH}. Moreover, if $F(y_0)>0$ (if $F(y_0)<0$), then $y(t;y_0)$ is non-decreasing (non-increasing) for all $t>0$. If $F(y_0)=0$, then $y(t;y_0)\equiv y_0$ for all $t>0$. Hence, $\Hs(\Te_h,y(t-r;y_0))\gtreqqless \Hs(\Te_h,y_0)$ if $F(y_0)\gtreqqless0$. In particular, $(\Hs(\Te_r,y_0))_{r\geq 0}$ is a $(\Fs_r^{\Te})_{r\geq 0}\overset{\text{sub}}{\underset{\text{super}}{-}}$martingale. A straightforward application of Doob's martingale convergence theorem yields that $\Hs_{\infty}(y_0)\in[0,1]$ exists almost surely for any $y_0\in [0,1]$. 
\end{proof}
The following lemma will be useful for a more detailed description of $\Hs_{\infty}(y_0)$.

\begin{lem}\label{lem:disteq}
Let $\Te^{(1)}$, $\Te^{(2)}$, $\Te^{(3)}$, and $\Te^{(4)}$ be four independent sASG processes, each almost surely starting in~$\ssroot{1}$. For $y_0\in [0,1]$ and $i\in[4]$,  let $\Hs_{\infty}^{(i)}(y_0)\defeq\lim_{r\to\infty}\Hs_{\infty}(\Te_r^{(i)},y_0)$. Then, \begin{align}
	(u+s+\gamma)\E\Big[G\big(\Hs_{\infty}^{(1)}(y_0)\big)\, \Big]=& u\nu_1 G(1)+ u\nu_0 G(0)+ s \E\Big[G\big(\Hs_{\infty}^{(2)}(y_0) \Hs_{\infty}^{(3)}(y_0)\big)\,\Big] \nonumber\\
	&+ s \E\Big[G\big(\Hs_{\infty}^{(2)}(y_0) [\Hs_{\infty}^{(3)}(y_0)(1-\Hs_{\infty}^{(4)}(y_0))+\Hs_{\infty}^{(4)}(y_0) ] \big) \Big]\label{eq:disteq}.
\end{align}
\end{lem}
\begin{proof}
	First note that it follows from the construction of $\Te$ that for all bounded $G\in\Cs([0,1])$ and $r\geq 0$, one has $\E_{\ssroot{1}}[G(\Hs(\Te_r,y_0))]=\E_{\rootsin}[G(H(\brea_r,y_0))]$. Now let $\brea$ be an eASG process started in~$\rootsin$, and let~$T$ be the time of the first transition. Note that by Lemma~\ref{lem:martingale}, for any $y_0\in [0,1]$, $\lim_{r\to\infty} H(\brea_r,y_0)=\Hs_{\infty}(y_0)$ in distribution. A first-step decomposition yields 
	\begin{align}
		\E[G(H(&\brea_r,y_0))]\nonumber \\
		=&\P(T\leq r)\E[G(H(\brea_r,y_0))\mid T\leq r]+\P(T>r)G(y_0) \nonumber\\
		=&\P(T\leq r)\Big\{\P(\brea_T=\rootsin_{\rho}^\times)G(1)+\P(\brea_T=\rootsin_{\rho}^\circ)G(0)+\P(\brea_T=\rootsin_{\rho}^\nwedge)\E\Big[G\big(H(\brea_r,y_0)\big)\mid \brea_T=\rootsin_{\rho}^\nwedge,T\leq r\Big]\nonumber\\
		&+\P(\brea_T=\rootsin_{\rho}^\pitch)\E\Big[G\big(H(\brea_r,y_0)\big)\mid \brea_T=\rootsin_{\rho}^\pitch,T\leq r\Big]\Big\}+\P(T>r)G(y_0)\nonumber \\
		=&\frac{\P(T\leq r)}{u+s+\gamma}\Big\{u\nu_1G(1)+u\nu_0G(0)+s\E\Big[G\big(H(\chil{\brea}_{r},y_0)H(\chir{\brea}_r,y_0)\, \big)\mid \brea_T=\rootsin_\rho^\nwedge,T\leq r\Big]\nonumber\\
		&+\gamma\E\Big[G\big(H(\chil{\brea}_r,y_0)[H(\chim{\brea}_r,y_0)(1-H(\chir{\brea}_r,y_0))+H(\chir{\brea}_r,y_0)]\, \big)\mid \brea_T=\rootsin_\rho^\pitch,T\leq r\Big]\Big\}+\P(T>r)G(y_0), \label{eq:hinfdecomp}
	\end{align} where we used Lemma~\ref{Hep}. For the binary branching term, we use the tower property and the independence of $\chil{\brea}_r$ and $\chir{\brea}_r$ to obtain
	\begin{align}
		\E[G(H(\chil{\brea}_{r},y_0)H(\chir{\brea}_r,y_0))\mid \brea_T=\rootsin_\rho^\nwedge,T\leq r]&=\E[ \, \tilde{\E}[G(H(\tilde{\brea}^{(1)}_{r-T},y_0)H(\tilde{\brea}^{(2)}_{r-T},y_0))]\mid T\leq r],\label{eq:hinfsel}
	\end{align} 
where $\tilde{\brea}^{(1)}$ and $\tilde{\brea}^{(2)}$ are two independent eASG processes starting in~$\rootsin$, and $\tilde{\E}$ is the expectation under their joint law. Analogously, we obtain
\begin{align}
	\E[G(H&(\chil{\brea}_r,y_0)[H(\chim{\brea}_r,y_0)(1-H(\chir{\brea}_r,y_0))+H(\chir{\brea}_r,y_0)])\mid \brea_T=\rootsin_\rho^\pitch,T\leq r]\nonumber \\
	=&\E[ \, \tilde{\E}[G(H(\tilde{\brea}^{(1)}_{r-T},y_0)[H(\tilde{\brea}^{(2)}_{r-T},y_0)(1-H(\tilde{\brea}^{(3)}_{r-T},y_0))+H(\tilde{\brea}^{(3)}_{r-T},y_0)]\mid T\leq r],\label{eq:hinfint}
\end{align} 
where $\tilde{\brea}^{(1)}$, $\tilde{\brea}^{(2)}$, $\tilde{\brea}^{(3)}$ are three independent eASG processes started in~$\rootsin$, and $\tilde{\E}$ is the expectation under their joint law with $\tilde{\brea}^{(1)}_0=\tilde{\brea}^{(2)}_0=\tilde{\brea}^{(3)}_0=\rootsin$. Using~\eqref{eq:hinfsel} and~\eqref{eq:hinfint} in \eqref{eq:hinfdecomp}, and then taking $r\to\infty$ yields~\eqref{eq:disteq}.
\end{proof}
In the remaining section, $\Pb$ and $\E$ are to be understood as $\Pb_{\ssroot{1}}$ and $\E_{\ssroot{1}}$, respectively.
\begin{proposition}\label{prop:Hisbernoulli}
	If $y_0\in[0,1]$ is not an unstable equilibrium of~\eqref{eq:dlimitdiffeq} located in $(0,1)$ and if $\Te_0=\ssroot{1}$, then $\Hs_{\infty}(y_0)$ has a Bernoulli distribution with parameter~$y_\infty(y_0)$.
\end{proposition}
\begin{proof}
	Assume $y_0\in [0,1]$ is not an unstable equilibrium of~\eqref{eq:dlimitdiffeq} located in $(0,1)$ and $\Te_0=\sroot{1}$. We now show that $\Hs_{\infty}(y_0)$ has a Bernoulli distribution. Set $E(y_0)\defeq\E[\Hs_{\infty}(y_0)]$ and~$V(y_0)\defeq\E[\Hs_{\infty}(y_0)^2]$. By the duality, \begin{equation}
		E(y_0)=y_{\infty}^{}(y_0),\label{equilibriumaux}
	\end{equation} so that $y_{\infty}(y_0)$ is then the corresponding parameter.
	The idea is to use that if $X$ is a random variable in~$[0,1]$, $X$ has Bernoulli distribution if and only if~$\E[X(1-X)]=0$. In particular, $\Hs_{\infty}(y_0)$ has Bernoulli distribution if and only if $E(y_0)=V(y_0)$. 
	Consider first $y_0\in\{0,1\}$. Then for any $n\in \N$, \begin{align}\label{eq:HS01}
		\E[\Hs_{\infty}(y_0)^n]&=\lim_{r\to\infty} \E[\Hs(\Te_r,y_0)^n]\nonumber\\
		&=\lim_{r\to\infty}\Big( \E[\Hs(\Te_r,y_0)^n\1_{\{T_{\ssroot{0}}<\infty\}}]+\E[\Hs(\Te_r,y_0)^n\1_{\{T_{\Delta}<\infty\}}]+\E[\Hs(\Te_r,y_0)^n\1_{\{T_{\mathrm{abs}}=\infty\}}]\Big)\nonumber\\
		&=\P(T_{\ssroot{0}}<\infty) + y_0\P(T_{\mathrm{abs}}=\infty)
	\end{align}
	where we used Lemma~\ref{lem:monotone} in the last step. In particular, $E(0)=V(0)=\P(T_{\ssroot{0}}<\infty)$ and $E(1)=V(1)=\P(T_{\Delta}=\infty)$. Hence, $\Hs_{\infty}(0)$ and $\Hs_{\infty}(1)$ have Bernoulli distribution. For unstable $y_0\in (0,1)$, we apply Lemma~\ref{lem:disteq} with $G(x)=x(1-x)$. A straightforward calculation then gives
	\begin{align}
		E(y_0)-V(y_0)&=\frac{s}{u+s+\gamma} \big(E(y_0)-V(y_0)\big)\big(E(y_0)+V(y_0)\big)\nonumber\\
		&\quad +\frac{\gamma}{u+s+\gamma}\big(E(y_0)-V(y_0)\big)\big(V(y_0)^2+V(y_0)(2-3E(y_0))+E(y_0)(2-E(y_0))\big),\label{eq:condHfirststep2}
	\end{align}
	Because $\Hs_{\infty}(y_0)\in [0,1]$, we moreover have $$0\leq V(y_0)=\E\big[\Hs_{\infty}(y_0)^2\big]\leq \E\big[\Hs_{\infty}(y_0)\big]=E(y_0).$$
	We want to further narrow down the value of~$V(y_0)$. To do so, we consider equation~\eqref{eq:condHfirststep2} with unknown $x=V(y_0)$, i.e. we rewrite \eqref{eq:condHfirststep2} as $(E(y_0)-x)\,\hat{p}(E(y_0),x)=0,$
	where $$\hat{p}\big(E(y_0),x\big)\defeq s\big(E(y_0)+x\big)+\gamma \big(x^2+x(2-3E(y_0))+E(y_0)(2-E(y_0))\big)-(u+s+\gamma).$$
	Note that $\hat{p}(E(y_0),0)=s(E(y_0)-1)+\gamma(E(y_0)(2-E(y_0))-1)-u\leq -u$ and~$\hat{p}(E(y_0),E(y_0))=F'(E(y_0))\leq 0$, where the inequality follows from~\eqref{equilibriumaux}, together with the fact that  $y_0\in (0,1)$ is not an unstable equilibrium so that $y_{\infty}(y_0)$ is attracting (from at least one side), and therefore, $F'(y_{\infty}^{}(y_0))\leq 0$. In particular, since~$\hat{p}(E(y_0),x)$ is a quadratic polynomial with positive quadratic term in~$x$, $\hat{p}(E(y_0),x)\neq 0$ for all~$x\in [0,E(y_0))$. Altogether, this implies~$V(y_0)=E(y_0)$, and hence $\Hs_{\infty}(y_0)$ has Bernoulli distribution.
\end{proof}
We now provide the proof for the representation of the equilibria of the mutation--selection equation in terms of the sASG process. 
\begin{proof}[Proof of Theorem~\ref{thm:representationequilibrium} (stochastic representation of equilibria)]
	Lemma~\ref{lem:martingale} says that $\Hs_{\infty}(y_0)$ exists almost surely for any $y_0\in [0,1]$. Proposition~\ref{prop:Hisbernoulli} states that if $y_0$ is not an unstable equilibrium located in $(0,1)$, then $\Hs_{\infty}(y_0)$ has Bernoulli distribution with parameter $y_{\infty}(y_0)$. The basic idea to prove~\eqref{eq:yinftyw1} is to take the limit~$t\to\infty$ in the duality relation and then exploit what we already know about~$\Hs_{\infty}(y_0)$.
	First, decompose~$\Hs(\Te_t,y)$ according to $\{T_{\ssroot{0}}<\infty\}$ and $\{T_{\ssroot{0}}=\infty\}$. More precisely, starting from the duality,
	\begin{equation}
		y(t;y_0)=\E[\Hs(\Te_t,y_0)\1_{\{T_{\ssroot{0}}<\infty\}}]+\E[\Hs(\Te_t,y_0)\1_{\{T_{\ssroot{0}}=\infty\}}].\label{eq:auxeq}
	\end{equation}
	Since~$\Hs(\sroot{0},y_0)=1$, the first term on the right-hand side of~\eqref{eq:auxeq} converges to~$\Pb(T_{\ssroot{0}}<\infty)$ as~$t\to\infty$. Because $\Hs(\Delta,y_0)=0$, using Lemma~\ref{lem:martingale}, we deduce that the second term on the right-hand side of~\eqref{eq:auxeq} converges to~$p(y_0)\defeq\E[\1_{\{T_{\mathrm{abs}}=\infty\}}\Hs_{\infty}(y_0)]$ as~$t\to\infty$, thus proving \eqref{eq:yinftyw1}. Next, we establish the connection between the (non-)absorption probability and $\ymin$ ($\ymax$). By Lemma~\ref{lem:monotone} and the definition of~$\Hs$, on~$\{T_{\mathrm{abs}}=\infty\}$, we have $\Hs_\infty(0)=0$ and $\Hs_\infty(1)=1$. Hence, we have $p(0)=0$ and $p(1)=\Pb(T_{\mathrm{abs}}=\infty)$. In particular, \eqref{eq:yinftyw1} with $y_0=0$ (resp. $y_0=1$) yields $\Pb(T_{\ssroot{0}}<\infty)=\ymin$ (resp. $\Pb(T_{\Delta}=\infty)=\ymax$).
\end{proof}
We are now ready to prove Corollary~\ref{coro:refinebernoulli}, which provides a more refined picture of~$\Hs_{\infty}(y_0)$ and its connection to the absorbing states of~$\Te$.
\begin{proof}[Proof of Corollary~\ref{coro:refinebernoulli}]
	First, recall that for all $y_0\in [0,1]$, by the definition of~$\Hs$
	\begin{equation}
		\Hs_{\infty}(y_0)=1\ \, \text{on}\ \,  \{T_{\ssroot{0}}<\infty\};\ \ \text{and} \ \, \Hs_{\infty}(y_0)=0\ \, \text{on} \ \, \{T_{\Delta}<\infty\}.\label{eq:Hinfonabs}
	\end{equation}
	If $\P(T_{\mathrm{abs} }=\infty)=0$, then there is nothing more to prove. Hence, let $\P(T_{\mathrm{abs} }=\infty)>0$.
	Assume~$y_0\in \Attr(\ymin)$. Theorem~\ref{thm:representationequilibrium} then yields $\E[\1_{\{T_{\mathrm{abs} }=\infty\}} \Hs_{\infty}(y_0)]=0$. In particular, $\E[ \Hs_{\infty}(y_0)\mid T_{\mathrm{abs}}=\infty]=0$; combining this with~\eqref{eq:Hinfonabs} proves the result for $y_0\in \Attr(\ymin)$. Next, assume $y_0\in \Attr(\ymax)$. Theorem~\ref{thm:representationequilibrium} gives $\P(T_{\Delta}=\infty)=\P(T_{\ssroot{0}}<\infty)+\E[ \Hs_{\infty}(y_0) \1_{\{T_{\mathrm{abs}}=\infty\}}]$, which implies $1=\E[ \Hs_{\infty}(y_0) \mid T_{\mathrm{abs}}=\infty]$; combining this with~\eqref{eq:Hinfonabs} yields the result also in this case. 
\end{proof}
Next, we prove Proposition~\ref{prop:ycritical}.
\begin{proof}[Proof of Proposition~\ref{prop:ycritical}]
	For $i\in\{0,1\}$, define~ $A_i:=\{y_0\in[0,1]:\, \P(\Hs_{\infty}(y_0)=i\mid T_{\mathrm{abs}}=\infty)=1\}$. Set $z_c:=\sup A_0$ and note that $y_c=\inf A_1$. The definitions of~$\sup$ and~$\inf$ together with the monotonicity of $\Hs_{\infty}(\cdot)$ yield
	$(y_c,1]\subseteq A_1\subseteq [y_c,1]$ and $[0,z_c)\subseteq A_0\subseteq [0,z_c].$
	Since $A_0\cap A_1=\varnothing$, we deduce that $z_c\leq y_c$. Assume first that $z_c<y_c$. Using Theorem~\ref{thm:representationequilibrium}, we infer that for any  $y_0\in(z_c,y_c)$, $\ymin<y_\infty (y_0)<\ymax$ so that there are three equilibria and $y_{\infty}(y_0)$ is stable. But this contradicts the fact that, if there are three equilibria, the middle one is unstable, and thus proves that $z_c=y_c$. Moreover, note that $y_\infty(y_0)=\ymax$ if and only if $y_0\in A_1$. Similarly, $y_\infty(y_0)=\ymin$ if and only if $y_0\in A_0$. Thus, $\ymin\leq y_c\leq \ymax$. It only remains to prove that $y_c$ is an equilibrium. If $y_c=\ymin$ or $y_c=\ymax$, the result is trivially true. For $\ymin<y_c<\ymax$, since $y_\infty(y_0)=\ymin$ for any $y_0\in[\ymin,y_c)$ and $y_\infty(y_0)=\ymax$ for any $y_0\in(y_c,\ymax]$, the result follows from the continuity of $F$.
\end{proof}

\section{Ancestral type distribution: proofs and additional results related to Section~\ref{sec:mainresults:subsec:ancestraltype}}\label{sec:ancestraltypedistriubtionproofs}
This section contains the proofs for the representation(s) of the ancestral type distribution. Moreover, we relate our results to an alternative forward process. Throughout this section we assume $\nu_0=0$.

\subsection{Backward approach to the ancestral type distribution}\label{sec:ancestraltypeinteractive}
First we prove Proposition~\ref{prop:prodoftrees}, which provides a decomposition of~$\Hs_\star$.

\begin{proof}[Proof of Proposition~\ref{prop:prodoftrees}]
	Let $\asg\in \Xi$ without mark~$\circ$. Recall that the immune line of~$\asg$ is the path connecting the root of~$\asg$ to its leftmost leaf, and that $\Fs(\al)=( (\alpha_{\chir{v}_i})_{i=1}^{N(\alpha)},(\alpha_{\chim{w}_j},\alpha_{\chir{w}_j})_{j=1}^{M(\alpha)})$ is the eASG forest of~$\asg$ (with notation adopted from Def.~\ref{def:foreststratifiedASG}). Moreover, recall that if $\nu_0=0$ and $\asg$ contains no mark~$\circ$, the ancestral type of the root of~$\asg$ is~$1$ if and only if the leftmost leaf of~$\asg$ is the ancestral leaf of the root and it is unfit. Fix $y_0\in [0,1]$ and let $C(y_0)$ be a random leaf-type configuration (as defined below Def.~\ref{typro}).
	
	\smallskip
	
	Under $C(y_0)$, the leftmost leaf is unfit with probability~$y_0$. The leaf on the immune line is the ancestral leaf of the root if all internal vertices on the immune line inherit their ancestral leaf from a vertex on the immune line. For $i\in [N(\asg)]$, $v_i$ inherits its ancestral leaf from a vertex on the immune line if and only if $\cpv{\chir{v}_i}=1$,  the probability of which is $H(\asg_{\chir{v}_i},y_0)$. Similarly, for $j\in [M(\asg)]$, $w_j$ inherits its ancestral leaf from a vertex on the immune line if and only if $\cpv{\chim{w}_j}+\cpv{\chir{w}_j}>0$, the probability of which is $[H(\asg_{\chim{w}_j},y_0)+H(\asg_{\chir{w}_j},y_0)-H(\asg_{\chim{w}_j},y_0)H(\asg_{\chir{w}_j},y_0)]$. The result follows from the independence of the corresponding eASGs emerging from the immune line.
\end{proof}

We are now ready to prove Theorem~\ref{thm:ancestraldistributionfinite}.

\begin{proof}[Proof of Theorem~\ref{thm:ancestraldistributionfinite} (representation of the ancestral type distribution)]
	Let $y_0\in [0,1]$ and recall that $g_r(y_0)=\E_{\rootsin}[H_\star(\brea_r,y_0)]$. Let $(\brea_r)_{r\geq 0}$ be an eASG process with $\brea_0=\rootsin$. In this proof, $\brea_{r,v}$ denotes the subtree of $\brea_r$ rooted in~$v\in V(\brea_r)$. For $r\geq 0$, let $\Fs(\brea_r)=( (\brea_{r,\chir{v}_i})_{i=1}^{N(\brea_r)},(\brea_{r,\chim{w}_j},\brea_{r,\chir{w}_j})_{j=1}^{M(\brea_r)})$ be the eASG forest of $\brea_r$. It follows from the definition of $(\brea_r)_{r\geq 0}$ that if $v_i$ is the $i$th vertex with outdegree~$2$ on the immune line of $\brea_r$ for some $r\geq 0$, then $v_i$ is also $i$th vertex on the immune line of $\brea_t$ for all $t\geq r$; the same holds true for vertices with outdegree~$3$ on the immune line. Moreover, $(N(\brea_r))_{r\geq 0}$ and $(M(\brea_r))_{r\geq 0}$ are independent Poisson processes with rate~$s$ and $\gamma$, respectively.
	
	\smallskip
	
	Using the definition of $g_r$, Proposition~\ref{prop:prodoftrees}, and the independence of the eASGs rooted in~$(v_i)_{i=1}^{N(\brea_r)}$ and~$(w_j)_{j=1}^{M(\brea_r)}$, we obtain
	\begin{align}
		g_r(y_0)=&y_0 \, \E\bigg[\prod_{i=1}^{N(\brea_r)} H(\brea_{r,\chir{v}_i},y_0)\bigg]\E\bigg[\prod_{j=1}^{M(\brea_r)}\left[H(\brea_{r,\chim{w}_j},y_0)+H(\brea_{r,\chir{w}_j},y_0)-H(\brea_{r,\chim{w}_j},y_0)H(\brea_{r,\chir{w}_j},y_0)\right]\bigg]\label{eq:mainfactors}.
	\end{align}
	We begin by dealing with the first non-trivial factor in~\eqref{eq:mainfactors}. For $i\in \N$, let $T_{i}^{N}=\inf\{r\geq 0: N(\brea_r)=i\}$. We decompose the term according to the values of~$N(\brea_r)$. The subtrees rooted in $\chir{v}_1,\ldots,\chir{v}_{N(\brea_r)}$ are conditionally independent given $(T_i^N)_{i=1}^{N(\brea_r)}$. Hence, using the tower property of the conditional expectation yields
	\begin{align}
		\E\bigg[\prod_{i=1}^{N(\brea_r)} H(\brea_{r,\chir{v}_i},y_0)\bigg]&=\sum_{n=0}^{\infty}\P(N(\brea_r)=n) \E_n\bigg[\E\bigg[ \prod_{i=1}^n H(\brea_{r,\chir{v}_i},y_0)\mid (T_k^{N})_{k=1}^n\bigg]\bigg],\nonumber \\
		&= \sum_{n=0}^{\infty}\P(N(\brea_r)=n) \E_n\bigg[\prod_{i=1}^n\tilde{ \E}\left[  H(\tilde{\brea}_{r-T_i^N},y_0)\right]\bigg]\label{eq:1proofancestralthm},
	\end{align} where~$\E_n[\cdot]$ denotes the expectation conditional on~$N(\brea_r)=n$, $(\tilde{\brea}^i)_{i\in \N}$ is a sequence of independent eASG processes started in~$\rootsin$ with $\tilde{\E}$ denoting the corresponding expectation, and $\tilde{\brea}$ is a shorthand for~$\tilde{\brea}^1$. Now, we use the well-known connection between Poisson processes and the uniform distribution. Conditional on~$N(\brea_r)=n$, the jump times of~$N(\brea_r)$ have the same distribution as an ordered independent sample of size~$n$ from the uniform distribution on~$[0,r]$~\citep[Thm.~2.4.6]{norris1998markov}. Let $U_1,\ldots,U_n$ be independent uniformly distributed random variables in~$[0,r]$. Since $\prod_{i=1}^n\tilde{ \E}[  H(\tilde{\brea}_{r-T_i^N},y_0)]$ is a function that is symmetric in the arrival times of the Poisson process, we deduce that \begin{equation}
		\E_n\bigg[   \prod_{i=1}^n \tilde{ \E}[H(\tilde{\brea}_{r-T_i^N},y_0)]\bigg]=\E\bigg[ \prod_{i=1}^n \tilde{\E}\left[  H(\tilde{\brea}_{U_i},y_0)\right]\bigg],\label{eq:2proofancestralthm}
	\end{equation}
	since~$r-U_i$ is again uniform on~$[0,r]$. Moreover, $\tilde{\brea}^1_{U_1},\ldots, \tilde{\brea}^n_{U_n}$ are independent, and hence 
	\begin{align}
		\E\bigg[ \prod_{i=1}^n \tilde{\E}\left[  H(\tilde{\brea}_{U_i},y_0)\right]\bigg]=\left(\E\left[ \tilde{\E}\left[  H(\tilde{\brea}_{U_1},y_0)\right]\right]\right)^n =\bigg(\frac{1}{r}\int_0^r \tilde{\E}\left[  H(\tilde{\brea}_{\xi},y_0)\right] \dd \xi \bigg)^n=\bigg(\frac{1}{r}\int_0^r y(\xi,y_0) \dd \xi \bigg)^n,\label{eq:3proofancestralthm}
	\end{align}
	where we used the duality in Theorem~\ref{sec:mainresults:thm:dualityASG}. Combining~\eqref{eq:2proofancestralthm} and~\eqref{eq:3proofancestralthm} into~\eqref{eq:1proofancestralthm}, and using that~$N(\brea_r)$ is Poisson distributed with parameter~$sr$ yields
	\begin{align*} \E\bigg[\prod_{i=1}^{N(\brea_r)} H(\brea_{r,\chir{v}_i},y_0)\bigg]&=\sum_{n=0}^{\infty} \frac{(sr)^n}{n!}e^{-sr} \bigg(\frac{1}{r}\int_0^r y(\xi,y_0) \dd \xi \bigg)^n=\exp\bigg(-s\int_0^r(1-y(\xi,y_0))\dd \xi \bigg).\end{align*} 
	Next, we consider the second non-trivial factor in \eqref{eq:mainfactors}. For $j\in \N$, let $T_{j}^{M}=\inf\{r\geq 0: M(\brea_r)=i\}$. Applying the same techniques that led to~\eqref{eq:1proofancestralthm} and~\eqref{eq:2proofancestralthm}, we obtain for $m\in \N$,
	\begin{align} \E_m\bigg[\prod_{j=1}^{m}& \left[H(\brea_{r,\chim{w}_j},y_0)+H(\brea_{r,\chir{w}_j},y_0)-H(\brea_{r,\chim{w}_j},y_0)H(\brea_{r,\chir{w}_j},y_0)\right]\bigg]\nonumber \\
		=&\left(\E\left[ \tilde{\E}\left[H(\tilde{\brea}^1_{U_1},y_0)+H(\tilde{\brea}^2_{U_1},y_0)-H(\tilde{\brea}^1_{U_1},y_0)H(\tilde{\brea}^2_{U_1},y_0)\right]\right]\right)^m\nonumber\\
		=&\left(\E\left[ 2\, \tilde{\E}\left[H(\tilde{\brea}_{U_1},y_0)\right]-\left(\tilde{\E}\left[H(\tilde{\brea}_{U_1},y_0)\right]\right)^2 \right]\right)^m \nonumber\\
		=&\left(\E\left[ 2y(U_1;y_0)-y(U_1;y_0)^2  \right]\right)^m, \label{eq:2ndfactor}\end{align}
	where $\E_m$ now is the expectation conditional on $M(\brea_r)=m$, and in the last step we used the duality of Theorem~\ref{sec:mainresults:thm:dualityASG}.
	Decomposing the second non-trivial factor in~\eqref{eq:mainfactors} according to $M(\brea_r)$ and using~\eqref{eq:2ndfactor}, we have
	\begin{align*}
		\E\bigg[\prod_{j=1}^{M(\brea_r)}& \left[H(\brea_{r,\chim{w}_j},y_0)+H(\brea_{r,\chir{w}_j},y_0)-H(\brea_{r,\chim{w}_j},y_0)H(\brea_{r,\chir{w}_j},y_0)\right]\bigg]\\ =&\sum_{m=0}^{\infty} \frac{(\gamma r)^m}{m!}e^{-\gamma r} \bigg(\frac{1}{r}\int_0^r y(\xi;y_0)(2-y(\xi;y_0)) \dd \xi \bigg)^m 
		=\exp\bigg(-\gamma\int_{0}^{r} \big(1-y(\xi;y_0)\big)^2 \dd \xi\bigg).\end{align*}
	Altogether, we obtain~\eqref{eq:ancestraltypedistribution}.
\end{proof}

Having established Theorem~\ref{thm:ancestraldistributionfinite}, more explicit expressions for the ancestral type distribution can be derived. We now provide these additional results. We first consider $\gamma=0$, because this extends the representation of the long-term ancestral type distribution derived in~\citep[Thm.~23]{BCH17} to a finite time horizon.
\begin{proposition}[Ancestral type distribution without interaction] \label{prop.ancestraldistributionfinitehorizonwithoutinteraction}
	Assume~$\nu_0=0$ and~$\gamma=0$. Then,
	\begin{equation}
		g_r(y_0)=y_0\exp\bigg(- s\int_0^r \big ( 1-y(\xi;y_0)  \big ) \dd \xi \bigg),\qquad y_0\in [0,1].\label{eq:ancestraltypedistributionsimple}
	\end{equation}
	In particular,
	\begin{equation}
		g_r(y_0)=\begin{cases} y_0\,\frac{u-s\,y(r;y_0)}{u-sy_0},&\text{if }y_0\in [0,1)\setminus \{\frac{u}{s}\},\\
			y_0 \,\exp({-rs(1-y_0)}),&\text{if }y_0\in \{\frac{u}{s},1\}.\label{eq:ancestraltypedistributionsimple2}
		\end{cases}
	\end{equation}
	Furthermore,~$g_{\infty}(y_0)=\lim_{r\to\infty}g_r(y_0)$ exists and is given as follows. 
	\begin{enumerate}[label=(\roman*)]
		\item If~$s=0$,~$g_{\infty}(y_0)=y_0$ for all~$y_0\in [0,1]$.
		\item If~$u\leq s$,~$g_{\infty}(y_0)=\begin{cases}
			0,&\text{if }y_0\in [0,1),\\
			1,&\text{if }y_0=1.
		\end{cases}$
		\item If~$u>s$,~$g_{\infty}(y_0)=y_0\,\frac{u-s}{u-sy_0}$.
	\end{enumerate}
\end{proposition}
\begin{proof}
	We just proved that~\eqref{eq:equalityregioncodingancestralcoding} holds for~$\gamma\geq 0$. If $\gamma=0$, \eqref{eq:equalityregioncodingancestralcoding} reduces to~\eqref{eq:ancestraltypedistributionsimple}. Given~\eqref{eq:ancestraltypedistributionsimple}, \eqref{eq:ancestraltypedistributionsimple2} follows by standard integration techniques. To see this, consider~$y_0<u/s$. Then~$y(r;y_0)$ is increasing and hence \begin{align*}
		-s\int_0^r \big ( 1-y(\xi;y_0) \big ) \dd \xi=\int_{y_0}^{y(r;y_0)}\frac{-s}{u-s \eta}\dd \eta=\ln\left(\frac{u-s y(r;y_0)}{u-s y_0}\right),
	\end{align*} where we substituted~$y(\xi;y_0)=\eta$ and used the differential equation for $y$. Together with~\eqref{eq:ancestraltypedistributionsimple}, this leads to~\eqref{eq:ancestraltypedistributionsimple2} in this case. We can proceed similarly for~$y_0\in (u/s,1)$. For~$y_0\in \{u/s,1\}$, one has~$y(r;y_0)\equiv y_0$ and the result follows from~\eqref{eq:ancestraltypedistributionsimple}. $(i)$--\,$(iii)$ are a consequence of~\eqref{eq:ancestraltypedistributionsimple2} and the form of~$y_{\infty}(y_0)$ from Corollary~\ref{coro:convergencetoequilibrium} if~$\gamma=0$.
\end{proof}

We now state the explicit representation of the ancestral type distribution for $\gamma>0$.

\begin{proposition}[Ancestral type distribution with interaction]\label{prop:ancestraldistributionfinite}
	Suppose~$\nu_0=0$. Let $\gamma>0$, and~$\bar{y}_-, \bar{y}_+$ be given as in \eqref{eq:equilibria}. Set~$\sigma\defeq((\gamma+s)^2-4u\gamma)/\gamma^2.$
	For~$y_0\in \{\bar{y}_-,\bar{y}_+,1\}\cap[0,1],$ we have
	\begin{equation}
		g_r(y_0)=y_0\exp\left(-r(1-y_0)\big(s+\gamma(1-y_0)\big)\right). \label{eq:grforequi}
	\end{equation}
	For~$y_0\in [0,1]\setminus\{\bar{y}_-,\bar{y}_+,1\}$ and 
	\begin{enumerate}
		\item for $u<{u^\star}$,
		\begin{equation}\label{eq:grforequiusustar}
			g_r(y_0)\ =\ y_0 \left(\frac{\bar{y}_--y(r;y_0)}{\bar{y}_--y_0} \right)^{\frac{\bar{y}_+}{\sqrt{\sigma}}}  \bigg(\frac{\bar{y}_+-y_0}{\bar{y}_+-y(r;y_0)}\bigg)^{\frac{\bar{y}_-}{\sqrt{\sigma}}} ;
		\end{equation}
		\item for $u={u^\star}$ (recall that then $\bar{y}_-=\bar{y}_+$),
		\begin{equation}\label{eq:grforequiueustar}
			g_r(y_0)\ =\ y_0 \, \frac{y(r;y_0)-\bar{y}_-}{y_0-\bar{y}_-} \exp\bigg(- \bar{y}_-\,\frac{y(r;y_0)-y_0}{(y(r;y_0)-\bar{y}_-)(y_0-\bar{y}_-)} \bigg);		\end{equation}
		\item for $u>{u^\star},$
		\begin{align}
				g_r(y_0)\ =\ &y_0\ \sqrt{\frac{u-y(r;y_0)\big(s+\gamma\big(1-y(r;y_0)\big)\big) }{u-y_0\big(s+\gamma(1-y_0)\big) }}  \nonumber \\
				&\times \exp\left(-\frac{\gamma+s}{\gamma\sqrt{-\sigma}}  \left(\arctan\left(\frac{2y(r;y_0)-\frac{\gamma+s}{\gamma}}{\sqrt{-\sigma}}\right)-\arctan\left(\frac{2y_0-\frac{\gamma+s}{\gamma} }{\sqrt{-\sigma}}\right)\right)
				\right).\label{eq:grforequiugustar}
		\end{align}
	\end{enumerate}
\end{proposition}

\begin{proof}
Assume $\gamma>0$. Consider first $y_0\in \{1,\bar{y}_-,\bar{y}_+\}\cap [0,1]$. Here, $y(\,\cdot\,;y_0)\equiv y_0$ so that $$- \int_0^r \big(1-y(\xi;y_0)\big)\big(s+\gamma (1-y(\xi;y_0) )\big) \dd \xi=-r(1-y_0)(s+\gamma(1-y_0)).$$ This leads to~\eqref{eq:grforequi} via Theorem~\ref{thm:ancestraldistributionfinite}. For $y_0\in [0,1]\setminus \{1,\bar{y}_-,\bar{y}_+\}$, we apply classical integration techniques to~\eqref{eq:ancestraltypedistribution}. Consider (1) and assume first that~$u<s(\leq u^\star)$ so that $\bar{y}_-<1<\bar{y}_+$ (recall Table~\eqref{table:stability}). For~$y_0<\bar{y}_-$,~$y(r;y_0)$ is increasing. By substituting~$\eta=y(\xi,y_0)$, using~\eqref{eq:dlimitdiffeq}, and partial fraction decomposition (where it is useful to note that $\bar{y}_+,\bar{y}_-$ satisfy $u-y(s+\gamma(1-y))=\gamma\,(y-\bar{y}_+)(y-\bar{y}_-)$), we obtain 
	\begin{align*}
		- \int_0^r \big(1-y(\xi;y_0)\big) & \big(s+\gamma (1-y(\xi;y_0) )\big) \dd \xi\\
		&= -\int_{y_0}^{y(r;y_0)} \frac{s+\gamma(1-\eta)}{u-\eta (s+\gamma(1-\eta))}\dd \eta\\
		&=-\frac{1}{\gamma}\frac{s+\gamma(1-\bar{y}_-)}{\bar{y}_+-\bar{y}_-}  \int_{y_0}^{y(r;y_0)} \frac{1}{\bar{y}_--\eta}\dd \eta+\frac{1}{\gamma}\frac{s+\gamma(1-\bar{y}_+)}{\bar{y}_+-\bar{y}_-}  \int_{y_0}^{y(r;y_0)} \frac{1}{\bar{y}_+-\eta}\dd \eta\\
		&=\frac{1}{\gamma}\frac{s+\gamma(1-\bar{y}_-)}{\bar{y}_+-\bar{y}_-}  \log\bigg(\frac{\bar{y}_--y(r;y_0)}{\bar{y}_--y_0} \bigg)-\frac{1}{\gamma}\frac{s+\gamma(1-\bar{y}_+)}{\bar{y}_+-\bar{y}_-} \log\bigg(\frac{\bar{y}_+-y(r;y_0)}{\bar{y}_+-y_0} \bigg).
	\end{align*}
	Note that $s/\gamma+(1-\bar{y}_-)=\bar{y}_+$, $s/\gamma+(1-\bar{y}_+)=\bar{y}_-$, and $\bar{y}_+-\bar{y}_-=\sqrt{\sigma}>0$ so that~\eqref{eq:grforequiusustar} follows from Theorem~\ref{thm:ancestraldistributionfinite}in this case. A similar argument applies if~$1>y_0>\bar{y}_-$; but now~$y(r;y_0)$ is decreasing. For (1) with $u\in [s,u^{\star})$, we can proceed similarly; the only subtlety lies in the monotonicity of~$y(r;y_0)$ depending on~$y_0$ and the equilibria.
	Next consider (2), i.e. $u=u^{\star}$. Here, we have $\bar{y}_-=\bar{y}_+$, $\sigma=0$, and~$y(r;y_0)$ is increasing for all~$y_0\in [0,1]$. Hence,
	\begin{align*}
		-\int_{y_0}^{y(r;y_0)} \frac{s+\gamma(1-\eta)}{u-\eta (s+\gamma(1-\eta))}\dd \eta&=\int_{y_0}^{y(r;y_0)} \left(-\frac{s+\gamma(1-\bar{y}_-)}{\gamma(\eta-\bar{y}_-)^2}+\frac{1}{\eta-\bar{y}_-}\right)\dd \eta\\
		&=\bar{y}_-\bigg(\frac{1}{y(r;y_0)-\bar{y}_-}-\frac{1}{y_0-\bar{y}_-} \bigg)+\log\bigg(\frac{y(r;y_0)-\bar{y}_-}{y_0-\bar{y}_-} \bigg). 
	\end{align*}
	This leads to~\eqref{eq:grforequiueustar}. Finally, we treat (3), i.e.~$u>u^{\star}$, where $\sigma<0$. Again,~$y(r;y_0)$ is increasing. Here,
	\begin{align*}
		-\int_{y_0}^{y(r;y_0)} & \frac{s+\gamma(1-\eta)}{u-\eta (s+\gamma(1-\eta))}\dd \eta\\
		&=\frac{1}{2}\int_{y_0}^{y(r;y_0)} \frac{-(s+\gamma)+2\gamma\eta}{u-(s+\gamma)\eta+\gamma \eta^2}\dd \eta-\frac{1}{2}\int_{y_0}^{y(r;y_0)} \frac{s+\gamma}{u-(s+\gamma)\eta+\gamma \eta^2}\dd \eta\\
		&=\frac{1}{2}\log \Bigg(\frac{u-y(r;y_0)\big(s+\gamma\big(1-y(r;y_0)\big)\big) }{u-y_0\big(s+\gamma(1-y_0)\big) }\Bigg)-\frac{1}{2}\int_{y_0}^{y(r;y_0)}\frac{s+\gamma}{u-(s+\gamma)\eta+\gamma \eta^2}\dd \eta.
	\end{align*}
	In the last term, we substitute~$\mu=\phi(\eta)\defeq (2\eta-(1+\frac{s}{\gamma}))/\sqrt{-\sigma}$ and obtain
	\begin{align*}
		-\frac{1}{2}\int_{y_0}^{y(r;y_0)}\frac{s+\gamma}{u-(s+\gamma)\eta+\gamma \eta^2}\dd \eta &=-\frac{1}{\sqrt{-\sigma}}\Big(1+\frac{s}{\gamma} \Big)\int_{\phi(y_0)}^{\phi(y(r;y_0))} \frac{1}{1+\mu^2}\dd \mu\\
		&=-\frac{1}{\sqrt{-\sigma}}\Big(1+\frac{s}{\gamma} \Big)\Big[\arctan\big(\phi(y(r;y_0))\big)-\arctan\big(\phi(y_0)\big)\Big].
	\end{align*}
	This gives~\eqref{eq:grforequiugustar} and ends the proof.\end{proof}

Explicit expressions for the long-term ancestral type distribution can now be obtained from Proposition~\ref{prop:ancestraldistributionfinite}.
\begin{coro}[Long-term ancestral type distribution]\label{coro:ancestraltypedistribution}
	Let~$\nu_0=0$ and~$\gamma>0$. Then for all $y_0\in [0,1]$, $\lim_{r\to\infty}g_r(y_0)$ exists. Moreover, $g_{\infty}(1)=1$ and for $y_0\in [0,1)$ and 
	\begin{enumerate}[label=(\roman*)]
		\item~$u<{u^\star}$, we have \begin{equation}
			g_{\infty}(y_0)\ =\ \1_{\{y_0>\bar{y}_+\}}y_0 \bigg(\frac{1-\bar{y}_-}{y_0-\bar{y}_-} \bigg)^{\frac{\bar{y}_+}{\sqrt{\sigma}}}  \bigg(\frac{y_0-\bar{y}_+}{1-\bar{y}_+}\bigg)^{\frac{\bar{y}_-}{\sqrt{\sigma}}};
		\end{equation}
		\item~$u={u^\star}$ (recall that then $\bar{y}_-=\bar{y}_+$), we have 
		\begin{equation}
			g_{\infty}(y_0)\ =\ \1_{\{y_0>\bar{y}_-\}} y_0 \, \frac{1-\bar{y}_-}{y_0-\bar{y}_-} \exp\bigg(- \bar{y}_-\frac{1-y_0}{(1-\bar{y}_-)(y_0-\bar{y}_-)} \bigg);
		\end{equation}
		\item~$u>{u^\star}$, we have 
		\begin{align}
			g_{\infty}(y_0)\ =\ &y_0\, \sqrt{\frac{u-s}{u-y_0(s+\gamma(1-y_0))}} 
			\, \exp\left(-\frac{\gamma+s}{\gamma \sqrt{-\sigma}}\left[\arctan\left(\frac{1-\frac{s}{\gamma } }{\sqrt{-\sigma}}\right)-\arctan\left(\frac{2y_0-\frac{\gamma+s}{\gamma } }{\sqrt{-\sigma}}\right)\right]\right).
		\end{align}
	\end{enumerate}
\end{coro}

\begin{proof}
	Combining Corollary~\ref{coro:convergencetoequilibrium} with Proposition~\ref{prop:ancestraldistributionfinite} yields the result.
\end{proof}

Finally, we consider the ancestral type distribution at equilibrium, i.\,e.\,\,$$\big (1-g_{\infty}(y_{\infty}(y_0)),g_{\infty}(y_{\infty}(y_0))\big ).$$ 
\begin{proof}[Proof of Corollary~\ref{coro:ancestraldistributionequilibrium} (long-term ancestral type distribution at equilibrium)]
	In the following, we throughout use Corollaries~\ref{coro:convergencetoequilibrium} and~\ref{coro:ancestraltypedistribution}.
	For $y_0=1$, we have~$y_{\infty}(1)=1$ and hence~$g_{\infty}(y_{\infty}(1))=1$ regardless of~$u$. For the remainder, assume~$y_0\in [0,1)$. In $\Theta_1$, $y_{\infty}(y_0)=1$ so~$g_{\infty}(y_{\infty}(y_0))\equiv 1$. In $\Theta_2^a$, we have~$y_{\infty}(y_0)=\bar{y}_-$. Moreover in $\Theta_2^a$, $u<u^\star$ so that~$g_{\infty}(y_{\infty}(y_0))=0$. In $\Theta_2^b$ (recall that here, $\bar{y}_-=\bar{y}_+$), $y_{\infty}(y_0)=\bar{y}_+$ if $y_0\leq\bar{y}_+$, and $y_{\infty}(y_0)=1$ if $y_0>\bar{y}_+$. Moreover in $\Theta_2^b$, $u=u^\star$. Hence, ~$g_{\infty}(y_{\infty}(y_0))=\1_{\{y_0>\bar{y}_+\}}$. In $\Theta_3$, $g_{\infty}(y_0)=\bar{y}_-$ if $y_0<\bar{y}_+$, and $g_{\infty}(y_0)=1$ if $y_0>\bar{y}_+$. Moreover in $\Theta_3$, $u<u^\star$ so that~$g_{\infty}(y_{\infty}(y_0))=\1_{\{y_0>\bar{y}_+\}}$.
\end{proof}

\subsection{A forward approach to the ancestral type distribution}\label{sec:anctype:forward}
The expression for the ancestral type distribution in~\eqref{eq:ancestraltypedistribution} arises also as the ancestral type distribution in an alternative forward model. We first explain this result and then exhibit the connection to the original model.

\smallskip

Fix $y_0\in [0,1]$ and $f_0> 0$ (and let as before $u>0$ and $\gamma,s\geq 0$). Consider the solution $(z_0(t),z_1(t))_{t\geq 0}$ to the IVP given by $z_0(0)=(1-y_0)f_0$,  $z_1(0)=y_0f_0$,
\begin{equation} \frac{\dd z_0}{\dd t}(t)=z_0(t)\left( s - u+\gamma\frac{z_0(t)}{z_0(t)+z_1(t)} \right),\quad \text{and} \quad \frac{\dd z_1}{\dd t}(t)=u\, z_0(t). \label{eq:alternativemodel} 
\end{equation}
In this model, $z_0(t)$ ($z_1(t)$) is the \emph{absolute frequency} at time~$t$ of the fit (the unfit) type in a population with variable size in which only the fit type reproduces, no one dies (in particular, reproduction is not coupled to the death of another individual), and there are only deleterious mutations. It is convenient to think of it as arising under a law of large numbers of a two-type branching process with interaction in which only the fit type reproduces. Let $f(t)\defeq z_0(t)+z_1(t)$ be the \emph{population size} at time~$t$. If $\tilde y(t)\defeq z_1(t)/f(t)$ denotes the relative frequency of the unfit type, $f(t)$ satisfies $$\frac{\dd }{\dd t} f(t) = \frac{\dd }{\dd t}\big(z_0(t)+ z_1(t)\big)=(1-\tilde y(t))(s+\gamma(1-\tilde y(t))f(t),$$ augmented by $f(0)=z_0(0)+z_1(0)=f_0$. Hence, \begin{equation}
	f(t)=f_0\exp\left(\int_0^t\Big(1-\tilde y(\xi)\Big)\Big(s+\gamma\big(1-\tilde y(\xi)\big)\Big)\dd \xi\right).\label{eq:popsize}
\end{equation} For the ancestral type distribution of this model, the key is to think in terms of absolute frequencies. Since the individuals unfit at time~$0$ neither reproduce nor mutate, the size of their descendant population remains constant at $y_0f_0$, while the population size grows by a factor of $f(t)/f_0$. It is therefore clear that the proportion of individuals at time~$t$ that have unfit ancestors at time~$0$ is~$y_0 f_0/ f(t)=g_r(y_0)$, in line with~\eqref{eq:ancestraltypedistribution}.

\smallskip

The mutation-selection equation arises from~\eqref{eq:alternativemodel} as the dynamics of the relative frequency  of the unfit type. To see this, check that $\tilde y(t)$ indeed satisfies~\eqref{eq:dlimitdiffeq} with $\tilde y(0)=y_0$. Conversely, starting from the solution of~\eqref{eq:dlimitdiffeq} with $y_0\in[0,1]$, we obtain \eqref{eq:alternativemodel} via the transformation, 
\begin{equation}\label{eq:thompson}
	z_0(t)\defeq (1-y(t; y_0)) \, f(t) \quad \text{and}\quad z_1(t)\defeq y(t;y_0)\,  f(t),
\end{equation} for $f(t)$ in~\eqref{eq:popsize} for some $f_0> 0$. For the noninteractive case, the transformation~\eqref{eq:thompson} goes back to \citet{Thompson} and  is frequently used in deterministic population genetics; it allows to transform the quadratic system~\eqref{eq:dlimitdiffeq}  into a linear one.

\section*{Acknowledgements}
We are grateful to Jan Swart for making us aware of the connection to the cooperative branching process and for many helpful comments on an earlier version of the manuscript, which helped to shape the concept of the stratified ASG and triggered the introduction of the tripod trees. We also thank him and Anja Sturm for an open exchange and stimulating discussions related to their work in progress. It is our pleasure to thank Anton Wakolbinger for helpful discussions. We are grateful to two unknown referees for helpful comments on the manuscript. This project received financial support from Deutsche Forschungsgemeinschaft~(CRC 1283 `Taming Uncertainty`, Project~C1).

\end{document}